\documentclass[12pt,reqno]{amsart}
\usepackage{color}

\usepackage{amsmath}

\usepackage{amsthm}
\usepackage{amssymb}
\usepackage{graphics}
\usepackage{latexsym}
\usepackage{comment}
\usepackage{ulem}

\numberwithin{equation}{section}
\newtheorem{thm}{Theorem}[section]
\newtheorem{prop}[thm]{Proposition}
\newtheorem{lem}[thm]{Lemma}
\newtheorem{cor}[thm]{Corollary}

\theoremstyle{remark}
\newtheorem{rem}{Remark}[section]
\newtheorem{defn}{Definition}

\newcommand{\BBB}{\mathbb}
\newcommand{\R}{{\BBB R}}
\newcommand{\Z}{{\BBB Z}}

\newcommand{\N}{{\BBB N}}
\newcommand{\C}{{\BBB C}}
\newcommand{\LR}[1]{{\langle {#1} \rangle }}

\newcommand{\al}{\alpha}

\newcommand{\ga}{\gamma}

\newcommand{\te}{\theta}

\newcommand{\vp}{\varphi}

\newcommand{\e}{\varepsilon}
\newcommand{\ta}{\tau}

\newcommand{\p}{\partial}
\newcommand{\la}{\lambda}

\newcommand{\de}{\delta}
\newcommand{\om}{\omega}
\newcommand{\Om}{\Omega}

\newcommand{\supp}{\operatorname{supp}}

\newcommand{\I}{\infty}

\newcommand{\D}{{\mathcal D}}
\newcommand{\E}{{\mathcal E}}

\newcommand{\RE}{\operatorname{Re}}
\newcommand{\IM}{\operatorname{Im}}

\newcommand{\EQS}[1]{\begin{align} #1 \end{align}}
\newcommand{\EQQS}[1]{\begin{align*} #1 \end{align*}}

\newcommand{\Sp}{\mathcal{S'}}
\newcommand{\F}{\mathcal{F}}
\newcommand{\1}{{\mathbf 1}}
\newcommand{\dP}{\dot{P}}

\newcommand{\ti}{\widetilde}

\newcommand{\ov}{\overline}

\newcommand{\les}{\lesssim}
\newcommand{\gts}{\gtrsim}

\title[DNLS with nonvanishing boundary conditions]
{Local well-posedness for the derivative nonlinear Schr\"odinger equation with nonvanishing boundary conditions}

\author[L. Molinet and T. Tanaka]{Luc Molinet and Tomoyuki Tanaka}
\address[L. Molinet]{Institut Denis Poisson, Universit\'e de Tours, Universit\'e d'Orl\'eans, CNRS, Parc Grandmont, 37200 Tours, France}
\email[L. Molinet]{Luc.Molinet@lmpt.univ-tours.fr}
\address[T. Tanaka]{Graduate School of Engineering Science, Yokohama National University, Yokohama, Kanagawa, 240-8501, Japan}
\email[T. Tanaka]{tanaka-tomoyuki-fp@ynu.ac.jp}

\begin{document}

\begin{abstract}
  We consider the derivative nonlinear Schr\"odinger equation on the real line, with a background function $\psi(t,x)\in L^\infty(\mathbb{R}^2)$ that satisfies suitable conditions.
  Such a function may, for example, be a non-decaying solution of the equation, such as a dark soliton.
  By developing the energy method with correction terms, we prove that the Cauchy problem for perturbations around such an $L^\infty$ function is unconditionally locally well-posed in $ H^s(\mathbb{R}) $ for $ s>3/4 $.
  As a byproduct, we also establish local well-posedness in the Zhidkov space.
\end{abstract}

\maketitle
\setcounter{page}{001}

\section{Introduction}

We consider the Cauchy problem for the derivative nonlinear Schr\"odinger equation (DNLS for short):
\EQS{\label{eq1}
  i \p_t v+\p_x^2 v
  =i\la|v|^2 \p_x v+ i \mu v^2\p_x \ov{v}, \quad (t,x)\in\R^2,
}
where $\la,\mu\in\R$ are constants and $v(t,x):\R^2\to \C$ is an unknown function.
The equation \eqref{eq1} with $\la=2\mu$ arises in plasma physics, and its solution describes the nonlinear evolution of Alfv\'en waves propagating along a static magnetic field (see, e.g., \cite{MOMT,M76}).
We point out that in this case, the right hand side takes the divergence form, i.e., $i\mu\p_x(|v|^2 v)$.
Moreover, \eqref{eq1} with $\la=2\mu$ is also known to be completely integrable \cite{KN78}, which implies that the equation admits infinitely many conservation laws.

\subsection{Settings}\label{subs_set}
It is well-known that \eqref{eq1} possesses, at least formally, two conservation laws for solutions that do vanish at infinity.
These are often referred to as the mass and the energy, given respectively by
\EQQS{
  M[v]
   &=\int_\R |v(x)|^2 dx,\\
  E[v]
   &=\int_\R |v'|^2dx+\frac{\la+\mu}{2}\IM\int_\R |v|^2v\bar{v}'dx
    +\frac{(\la+\mu)\mu}{6}\int_\R |v|^6 dx,
}
where the prime denotes spatial derivatives.
Moreover, the momentum
\EQQS{
  P[v]
   &=\frac{1}{2}\IM\int_\R v\bar{v}' dx+\frac{\mu}{4}\int_\R |v|^4 dx.
}
is also conserved for \eqref{eq1} with $\la=2\mu$.
This quantity can be interpreted as the Hamiltonian in the sense that the equation takes the form $\p_t v=\p_x P'[v]$, where $P'[v]$ denotes the functional derivative of $P$.
These functionals are well-defined for functions belonging to $L^2$-based Sobolev spaces $H^s(\R)$ for some $s>0$, and they play a fundamental role in the study of the long time behavior of solutions.
It is therefore natural to consider the Cauchy problem for \eqref{eq1} with the initial data in $H^s(\R)$.
Indeed, a substantial body of research exists on this topic; see Subsection \ref{subs_prev} for previous results on well-posedness in $H^s(\R)$.

In contrast, one can find solutions outside the $L^2$-framework; for example, a solution to \eqref{eq1} with nonvanishing boundary conditions at infinity can be constructed through an elementary (though tedious) computation.
Concretely, the function
\EQS{\label{soliton1}
  \vp(t,x)=A e^{i\theta(x+2A^2t)} \sqrt{1-h(x+2A^2 t)}
}
is the solution to \eqref{eq1} with $(\la,\mu)=(2,1)$, where $A>0$ is a constant and
\EQQS{
  h(x)=\frac{2}{\sqrt{2}\cosh(A^2x)+1},\quad
  \theta(x)=\frac{A^2}{4}\int_0^x \frac{3h(y)^2-2 h(y)}{1-h(y)}dy.
}
We remark that $\te(x)$ is well-defined, since $1-h(x)\ge 1-h(0)=3-2\sqrt{2}>0$.
It is easy to see that $\lim_{|x|\to\I}|\vp(t,x)|=A>0$ for any $t\in\R$ and $h(x)$ is a Schwartz function.
Such solutions are called \textit{dark solitons} in the literature, and they do not belong to $L^2(\R)$.
Moreover, dark solitons are classified as \textit{black solitons} or \textit{gray solitons}, depending on whether they satisfy $|v(0,0)|=0$ or $|v(0,0)|>0$, respectively (see Chapter 5 of \cite{Agr}).
In the case of \eqref{soliton1}, we have $|\vp(0,0)|=(\sqrt{2}-1)A>0$.
There are also other types of solutions, called \textit{kinks}, which do not decay at infinity and whose profiles have different limits as $x\to-\I$ and $x\to\I$ (see, e.g., \cite{LCTP}).

In this work, we are interested in the well-posedness of the Cauchy problem for \eqref{eq1} under nonvanishing boundary conditions, within a framework that accommodates solutions such as \eqref{soliton1}.
One approach to handling this type of problem is to decompose the solution $v(t,x)$ as
\EQS{\label{decom1}
  v(t,x)=u(t,x)+\psi(t,x),
}
where $\psi\in L^{\I}(\R^2; \C)$ is a given function (possibly a special solution such as \eqref{soliton1}), and the goal is to construct $u(t,x)\in L^\I([0,T];H^s(\R))$ at low regularity.
For that purpose, using the decomposition \eqref{decom1}, we rewrite \eqref{eq1} as
\EQS{\label{eq2}
  \begin{aligned}
    i\p_t u+\p_x^2 u
    &=F(u,\psi)-\Psi,
  \end{aligned}
}
where
\EQS{\label{def_F}
  &\begin{aligned}
    F(u,\psi)&:=i\la(|u|^2\p_x u+|u|^2\p_x \psi +2\p_x u \RE u\bar{\psi}
     +2\p_x \psi \RE u\bar{\psi} +|\psi|^2 \p_x u)\\
    &\quad+i\mu(u^2\p_x \bar{u}+u^2\p_x \bar{\psi}+2u\psi \p_x\bar{u}
     +2u\psi \p_x\bar{\psi}+\psi^2\p_x \bar{u}),
  \end{aligned}\\
 \label{def_psi}
  &\Psi(\psi;t,x):=i\p_t \psi+\p_x^2 \psi-i\la|\psi|^2\p_x \psi-i\mu\psi^2 \p_x \bar{\psi}.
}
In what follows, we study the Cauchy problem for \eqref{eq2} in $H^s(\R)$, with a given function $\psi$.
Certain conditions are imposed on $\psi$; see \eqref{hyp_psi} and \eqref{hyp_psi2}.
These hypotheses are not artificial, but rather natural enough to capture certain special solutions.
Indeed, a straightforward calculation shows that the function \eqref{soliton1} satisfies \eqref{hyp_psi2} (and thus also \eqref{hyp_psi}).
The presence of a function $\psi$ that does not belong to $H^s(\R)$ causes certain difficulties, particularly when $s<1$, even if $\psi$ is sufficiently regular.
We will discuss this issue in Subsection \ref{subs_main} and Section \ref{nonregular}.

\subsection{Resonant and Nonresonant Interactions}

In the equation \eqref{eq1}, the nonlinearity contains a derivative, which makes it difficult to obtain closed estimates.
This kind of difficulty is commonly referred to as \textit{derivative loss}.
To eliminate derivative loss, the resonance function plays an important role.
To see this, we present a heuristic argument.
For simplicity, we may assume $(\la,\mu)=(2,1)$.
Consider the pullback of the solution defined by $\theta(t,x)=U(-t)v(t,x)$, where $v(t)\in C([0,T];H^s(\R))$ is the solution to \eqref{eq1} and $U(t)$ is the linear propagator (see Subsection \ref{subs_notation} for its definition).
By the Duhamel principle, the equation \eqref{eq1} can be rewritten as
\EQS{\label{eq_INT1}
  \theta(t)=v(0)-i\int_0^t \p_x U(-t')(|U(t')\te(t')|^2 U(t')\te(t'))dt'.
}
Taking the Fourier transform in spatial variable to \eqref{eq_INT1}, one obtains
\EQQS{
  \hat{\theta}(t,\xi)
  =\hat{v}(0,\xi)+\int_0^t
   \int_{\xi=\xi_1+\xi_2+\xi_3} e^{it'\Om}
   \xi\hat{\te}(t',\xi_1)\hat{\bar{\te}}(t',\xi_2)
   \hat{\te}(t',\xi_3)d\xi_1d\xi_2d\xi_3dt',
}
where the phase function $\Om$, called the \textit{resonance function}, is defined by
\EQS{\label{eq_INT4}
  \Om(\xi_1,\xi_2,\xi_3)
  :=(\xi_1+\xi_2+\xi_3)^2-\xi_1^2+\xi_2^2-\xi_3^2
  =2(\xi_1+\xi_2)(\xi_2+\xi_3).
}
If $\Om$ is sufficiently large (i.e., the interaction is nonresonant), one can integrate by parts in time, gaining a factor of $\Om^{-1}$ in the denominator.
This gain allows one to cancel the derivative $\xi$ in the nonlinear term.
This idea, now referred to as the \textit{normal form reduction}, was used in \cite{BIT11} for the KdV equation to show the unconditional uniqueness.
It has since been refined by many authors (see, e.g., \cite{GKO13,K19p}).
We also essentially adopt this approach, as developed in \cite{MV15}, to handle nonresonant interactions via using the Bourgain-type space $X^{s-1,1}$ (see Subsection \ref{subs_funct} for its definition).
We remark that nonlinear interactions are (nearly) resonant when $\Om$ is not sufficiently large.
Needless to say, the threshold for what constitutes ``sufficiently large" depends on the specific problem under consideration.

In the above observation, it is important to consider where the derivative falls.
As we will discuss in Section \ref{sec_apri}, from the perspective of the resonance function, the nonlinear term $v^2\p_x \bar{v}$ is significantly easier to handle than the term $|v|^2\p_x v$.
Moreover, it is well-known that the term $|v|^2\p_x v$ can be eliminated via a gauge transformation, which we discuss in the next subsection.

\subsection{Gauge Transformation}\label{subs_gaug}

The gauge transformation is a nonlinear transform that transforms the original equation into another one in which the most problematic term (or interaction) is removed.
In the present context, this term is $|v|^2 \p_x v$.
To illustrate this, we consider \eqref{eq1} with $\la\neq0$, and define a new function $\theta(t,x)$ by
\EQS{\label{eq_INT3}
  \theta(t,x):=v(t,x)\exp\bigg(i\de\int_{-\I}^x |v(t,y)|^2dy\bigg),
}
where $v$ is the solution to \eqref{eq1} and $\de\in\R$ is a constant to be chosen later.
It is not difficult to check that this $\theta$ satisfies the equation:
\EQQS{
  i\p_t  \theta+\p_x^2  \theta
  =i(\la+2\de)| \theta|^2\p_x  \theta+i(\mu+2\de) \theta^2\p_x\bar{ \theta}
   +\frac{\de}{2}(\la-3\mu-2\de)| \theta|^4 \theta.
}
By choosing $\de$ so that $\la+2\de=0$, the unfavorable term $| \theta|^2 \p_x  \theta$ is eliminated.
Then, we study the Cauchy problem for this transformed equation instead of \eqref{eq1}, and invert the gauge transform in the final step.
This idea is widely used in several dispersive equations with derivative nonlinearities (see, e.g., \cite{KT23p,Taka99,Tao04}).

It is important for functions to belong to $L^2(\R)$ in order for the gauge transformation to be well-defined.
Now, we come back to our problem \eqref{eq2}.
As in \eqref{def_F}, there are three terms in which the derivative lands on $u$:
\EQS{\label{eq_INT2}
  |u|^2 \p_x u,\quad
  \p_x u \RE u \bar{\psi}, \quad
  |\psi|^2 \p_x u.
}
As discussed above, the first term can be handled using the gauge transformation.
Moreover, if we try to define a new function $w(t,x)$ by
\EQQS{
  w(t,x)
  :=u(t,x)\exp\bigg(i\int_{-\I}^x \de_1|u(t,y)|^2
   +\de_2 \RE u(t,y) \ov{\psi(t,y)} dy\bigg),
}
we may hope to remove both the first and second terms in \eqref{eq_INT2} by choosing $\de_1,\de_2$ appropriately.
However, this definition of $w$ poses a problem, since we cannot guarantee that $u\bar{\psi}$ belongs to $ L^1(\R)$.
Here, $\psi$ is a given function satisfying \eqref{hyp_psi}.
For this reason, we do not attempt to apply the gauge transformation, and instead treat the equation \eqref{eq2} directly.
This leads to several difficulties, discussed in Subsection \ref{subs_main}.
Finally, since $ \psi$ does not belong to $ L^2(\R) $ we notice that we would encounter the same type of difficulty if one attempted to remove the third term using a gauge transformation.
On the other hand, this term is harmless within our approach, since we employ an energy method.
By integration by parts, we can shift the derivative onto $ \psi$, which is acceptable under our hypotheses.

\subsection{Previous Results on Well-Posedness}\label{subs_prev}

The local well-posedness for \eqref{eq1} has been extensively studied over the past four decades, beginning with the work of Tsutsumi and Fukuda \cite{TF,TF2}.
They established the local existence of solutions in $H^s(\R)$ for $s>3/2$, using the parabolic regularization and compactness arguments.
They also constructed global solutions in $H^2(\R)$ by utilizing the conservation law associated with the $H^2$-norm.
Subsequently, Hayashi \cite{H93} and Hayashi and Ozawa \cite{HO92} introduced the gauge transformation \eqref{eq_INT3} and constructed global solutions in $H^1(\R)$ under certain conditions on the size of the initial data.
Takaoka \cite{Taka99} combined this idea with a contraction argument in Bourgain spaces \cite{B93}, and proved the local well-posedness in $H^s(\R)$ for $s\ge 1/2$.
The work of Bourgain \cite{B93} revolutionized the well-posedness theory for dispersive equations, and is referred to as the \textit{Fourier restriction norm method} (see also \cite{KM93}).
This method effectively captures the time oscillation of the resonance function $\Om$, as defined in \eqref{eq_INT4}, and has been applied in various settings.
Biagioni and Linares \cite{BL01} showed that the data-to-solution map (also called the flow map) is not uniformly continuous in $H^s(\R)$ for $s<1/2$, which implies that the result of \cite{Taka99} is optimal at least when one relies on a contraction argument.
See also \cite{CKSTT01,CKSTT02,FHI17,O96,W15} for results on global well-posedness and \cite{H06,MST} for related topics.
It is worth mentioning that these results are established without using the complete integrability of \eqref{eq1} (which would require the assumption $\la=2\mu$), and are robust under certain perturbations.
Recently, Harrop-Griffiths, Killip, Ntekoume, and Vi\c{s}an \cite{HGKNV} established the global well-posedness for \eqref{eq1} with $\la=2\mu$ in $H^s(\R)$ for $s\ge 0$, exploiting the method of commuting flows, which is rooted in the complete integrability of the equation.
See also \cite{BP22,JLPS18,KV19} for results along these lines.

We now turn to results concerning the well-posedness under nonvanishing boundary conditions.
To this end, it is convenient to work with the Zhidkov space $Y^s(\R)$, introduced by Zhidkov \cite{Z01} (see Definition \ref{def_Zhid}).
The Zhidkov space is an effective function space for capturing solutions such as dark solitons and bore-like data \cite{ILS98}.
Van Tin \cite{VT22} showed the local well-posedness for \eqref{eq1} with $(\la,\mu)=(0,1)$ in $Y^4(\R)$.
He also obtained the well-posedness in space of the form $\varphi+ H^k(\R)$ for $k=1,2$, where $\varphi\in Y^4(\R)$.
In the case $k=1$, the result requires that $\|\p_x \varphi\|_{H^2}$, along with another norm, is sufficiently small.
It is worth mentioning that he excludes the nonlinearity $|u|^2\p_x u$ from the outset.
For further results on the well-posedness of dispersive equations in Zhidkov spaces, see, e.g., \cite{May,G05,P23}.
To the best of our knowledge, no results had been established for the equation \eqref{eq1} with general values of $(\la,\mu)$ under nonvanishing boundary conditions.
In this paper, we establish the (unconditional) local well-posedness in Theorems \ref{thm1} and \ref{thm2}, under assumptions \eqref{hyp_psi} and \eqref{hyp_psi2}, respectively.
As a corollary, we also obtain the well-posedness in the Zhidkov space in Theorem \ref{thm3}, which, in turn, improves upon the result of \cite{VT22}.

\subsection{Main Results, Difficulties and Strategy}\label{subs_main}

In this subsection, we first fix notation related to well-posedness and then introduce the main results.
We also discuss the difficulties and our approach to addressing them.
We begin by stating the assumptions on $\psi$.
Throughout this paper, we assume that  the given function $\psi(t,x)$ satisfies the following hypotheses:
\EQS{\label{hyp_psi}
  \begin{cases}
    J_x^{s+1+\e}\psi\in L^\I(\R^2),\\
    \p_t \psi\in L^\I(\R^2),\\
    \Psi\in L^\I(\R;H^{s+\e}(\R)),
  \end{cases}
}
where $\e>0$ is a fixed sufficiently small constant and $s\ge 0$.
See Subsection \ref{subs_notation} for the definition of $J_x$.
We remark that any constant function satisfies \eqref{hyp_psi}.
The aim of this paper is to prove the two local well-posedness results (Theorems \ref{thm1} and \ref{thm2}) for \eqref{eq2}.
For that purpose, we first recall the definition of a solution.

\begin{defn}[Solution]
  Let $T>0$ and $s>1/2$.
  We say that $u\in L^\I([0,T];H^s(\R))$ is a solution to \eqref{eq2} with the initial data $u(0,x)=u_0(x)$ if $u$ satisfies \eqref{eq2} in the distributional sense, i.e., it holds that for any test function $\vp\in C_c^\I([-T,T]\times \R)$,
  \EQS{\label{def_sol}
    \int_0^\I \int_\R
    \big\{(\ov{i\p_t {\vp}+\p_x^2{\vp}}) u +\bar{\vp}(F(u,\psi)-\Psi)\big\}dxdt'
    +i\int_\R \bar{\vp}(0,\cdot)u_0dx=0.
  }
\end{defn}

\begin{rem}\label{rem_sol}
  By a standard abstract argument, we can deduce some properties of the solution $u\in L^\I([0,T];H^s(\R))$.
  It holds that $u\in C([0,T];H^{s'}(\R))$ for any $s'<s$ and $u\in C_w([0,T];H^{s}(\R))$.
  Moreover, $u$ satisfies $u(0,x)=u_0(x)$ and the Duhamel formula associated with \eqref{eq2} in $C([0,T];H^{-2}(\R))$.
  See Remark 2.1 in \cite{MPV19} for details.
\end{rem}

We also recall the notion of unconditional local well-posedness, introduced by Kato \cite{Kato95}.
Unconditional uniqueness (abbreviated as UU) means that uniqueness holds without reference to auxiliary spaces.
For \eqref{eq1}, the unconditional local well-posedness in $H^s(\R)$ for $s>3/2$ follows from the classical energy method.
Thus, it is meaningful to investigate UU for $s\le 3/2$.
This was done by Win \cite{W08}  in $H^1(\R)$, while Mosincat and Yoon \cite{MY20} proved the UU in $H^{1/2}(\R)$.

\begin{defn}\label{def_UWP}
  We say that the Cauchy problem for \eqref{eq2} is unconditionally locally well-posed in $ H^s(\R)$ if the following conditions hold:
  \begin{itemize}
    \item For any initial data $ u_0\in H^s(\R) $, there exist $ T=T(\|u_0\|_{H^s})>0 $ and a solution $ u \in C([0,T]; H^s(\R)) $ to \eqref{eq2} emanating from $ u_0 $.
    \item The solution $u$ is unique in the class $ L^\I([0,T]; H^s(\R))$.
    \item For any $ R>0$, the flow map $ u_0 \mapsto u $ is continuous from the ball of radius $R$ centered at the origin in $ H^s(\R) $  into $C([0,T(R)]; H^s(\R))$.
  \end{itemize}
\end{defn}

We are now ready to present the main results of this article.

\begin{thm}\label{thm1}
  Assume that $\psi$ satisfies \eqref{hyp_psi} with $s\ge 0$ as stated below.
  Then, the Cauchy problem for \eqref{eq2} is unconditionally locally well-posed in $H^s(\R)$ under the following conditions:
  \begin{itemize}
    \item for any $s>3/4$ when $\la=2\mu$ or  $ \lambda=0 $.
    \item for any $s\ge 1$ otherwise.
  \end{itemize}
  Moreover, the maximal time of existence $T$ satisfies $T\ge g(\|u_0\|_{H^{\frac{3}{4}+}})>0$ when $\la=2\mu$ or $ \lambda=0$, and $T \ge g(\|u_0\|_{H^1})>0$ otherwise, where $g$ is a smooth decreasing function.
\end{thm}

\begin{rem}
As explained in the paragraph just after Theorem \ref{thm2}, the results for $ \lambda=2\mu $ and $ \lambda=0 $ are not of exactly the same nature.
Indeed, in the first case we estimate the difference of two solutions in $ H^{s-1}$, whereas in the second case, the difference has to be  estimated in $ H^{s-\frac{1}{2}}$.
\end{rem}

If we assume further $\p_x \psi\in L^\I(\R;L^1(\R))$, we can show the unconditional local well-posedness below $s=1$ even when $ \lambda\neq0,2\mu$.
Once again, the function \eqref{soliton1} serves as an example which satisfies all the required conditions.

\begin{thm}\label{thm2}
  Let $3/4<s<1$.
  Assume that $\psi$ satisfies \eqref{hyp_psi} with this value of $s$ and $\p_x \psi\in L^\I(\R;L^1(\R))$.
  Then, the Cauchy problem for \eqref{eq2} is unconditionally locally well-posed in $H^s(\R)$ with a maximal time of existence $T\ge g(\|u_0\|_{H^{\frac 34 +}})>0$, where $g$ is a smooth decreasing function.
\end{thm}

We now explain the structure of the proofs.
First, it is worth noticing that Bourgain's method \cite{B93} does not appear to be applicable due to the presence of the nonlinear term $|u|^2 \p_x u$, and more importantly, the linear term $ |\psi|^2 \p_x u $, which involves a derivative acting on $ u$. (In sharp contrast, a linear term involving $\bar{u}$ with a derivative falling on $\bar{u}$ does not pose any difficulty, thanks to the good resonance relation.)
The proofs for these theorems are based on the approach developed by the first author and Vento \cite{MV15}.
This method allows us to exploit nonlinear dispersive properties, such as nonresonance structures, within the framework of the energy method.
It was later adapted to the Cauchy problem with nonvanishing boundary conditions in \cite{P23} (see also \cite{May}), incorporating refined Strichartz estimates from \cite{KT1}.
A relatively direct application of \cite{May,MV15,P23} suffices to prove Theorem \ref{thm1}, which establishes the unconditional local well-posedness of \eqref{eq2} in $H^s(\R)$ for $s > 3/4 $ in the divergence form case $(\la=2\mu)$, and for $s\ge 1$ otherwise.
In the non-divergence form case, however, the tools developed in \cite{MV15} do not seem to be sufficient to extend Theorem \ref{thm1} below $H^1(\R)$.
More precisely, given two solutions to \eqref{eq2} belonging to $ L^\I_T H^s $, it is difficult to evaluate their difference at the $H^{s-1}$ level when $s<1$, due to some {\it low$\times$high$\times$high$\to$low} interactions.
Such interactions arise, for example, in the nonlinear term $u\bar{u}\p_x w$, and do not appear to be well-defined at this level of regularity.
To overcome this difficulty, we perform the energy estimate for the difference in $H^{s-1/2}$ instead of $H^{s-1}$.
However, this makes it more difficult to recover the derivative loss, even in some nonresonant cases.
At this stage it is worth noticing that this issue arises solely from the first nonlinear term in \eqref{eq1}.
In contrast, when $\la=0$, the estimates for the difference can be closed in $ H^{s-1/2}$, thanks to the strong nonresonance properties of the second nonlinear term.
Now, in the case $ \lambda\neq 0$, we address this issue by introducing correction terms, also known as the \textit{modified energy} (see \eqref{def_E}), following the approach in \cite{MPV18}.
The modified energy effectively cancels the most problematic nonresonant interactions, which allows us to control the difference in $H^{s-1/2}$.
Finally, one more difficulty emerges: a derivative loss appears in the higher-order terms generated by the time derivative of the modified energy.
This again stems from certain {\it low$\times$high$\times$high$\to$low} interactions.
This issue is resolved by extracting a key cancellation through a delicate analysis, as observed in \cite{EW19,MPV18} for quadratic nonlinearities.
See Section \ref{nonregular} for details.

Next, we present the well-posedness result in the Zhidkov space $Y^s(\R)$ for \eqref{eq1}.
See Subsection \ref{subs_funct} for its definition.
By Lemma \ref{lem_Zhidkov1}, any function in $Y^s(\R)$ can be decomposed into the sum of a function in $H^s(\R)$ and a smooth bounded function.
This allows us to apply Theorem \ref{thm1} to establish the result.

\begin{thm}\label{thm3}
  For any $s>3/4$ when $ \lambda=2\mu$, and for any $ s\ge 1 $ otherwise, the following statements hold:
  \begin{itemize}
    \item(Existence) Let $v_0\in Y^s(\R)$.
    Then, there exist $T=T(\|v_0\|_{Y^s})>0$ and $v\in C([0,T];Y^s(\R))$ such that $v$ satisfies \eqref{eq1} with $v(0,x)=v_0(x)$.
    Moreover, it holds that $v(t)-v_0\in C([0,T];H^s(\R))$.
    \item(Uniqueness) Let $\de>0$, and let $v_1,v_2\in C([0,\de];Y^s(\R))$ be two solutions to \eqref{eq1} on $[0,\de]$, both emanating from the same initial data $v_0\in Y^s(\R)$.
    Assume that $v_1(t)-v_0,v_2(t)-v_0\in L^\infty([0,\de];H^s(\R))$.
    Then it holds that $v_1(t)=v_2(t)$ on $[0,\de]$.
    \item(Continuous dependence) Let $R>0$ and $S:\{u_0\in Y^s(\R)\ |\ \|v_0\|_{Y^s}\le R\}\to C([0,T];Y^s(\R))$ be the flow map defined as above.
    Then, $S$ is continuous.
  \end{itemize}
\end{thm}

Note that $H^s(\R)\hookrightarrow Y^s(\R)$ when $s>1/2$.
The uniqueness in Theorem \ref{thm3} is based on the unconditional uniqueness of Theorem \ref{thm1}, which requires the assumption that $v_j(t)-v_0\in L^\infty([0,T];H^s(\R))$.

\subsection{Plan of This Paper}

This paper is organized as follows.
In Section \ref{notation}, we fix the notation and collect some fundamental estimates, including Sobolev inequalities and Leibniz rules.
In Section \ref{sec_stri}, we discuss refined Strichartz estimates which will be used for resonant interactions throughout this article.
In Section \ref{sec_apri}, we first introduce Bourgain type estimates that are essential for handling nonresonant interactions.
We also derive a key a priori estimate for solutions.
In Section \ref{regular}, we obtain the a priori estimate for the difference of two solutions.
This estimate plays a crucial role in the proof of Theorem \ref{thm1}.
In Section \ref{nonregular}, we define the modified energy and establish an a priori estimate for the difference, which is the key ingredient for the proof of Theorem \ref{thm2}.
In Section \ref{sec_proof}, we present the proofs of Theorems \ref{thm1}, \ref{thm2} and \ref{thm3}, with a primary focus on Theorem \ref{thm2}.

\section{Notation, Function Spaces and Basic Estimates}\label{notation}

\subsection{Notation}\label{subs_notation}

Throughout this paper, $\N$ denotes the set of nonnegative integers.
For any positive numbers $a$ and $b$, we write $a\les b$ when there exists a positive constant $C$ such that $a\le Cb$.
We also write $a\sim b$ when $a\les b$ and $b\les a$ hold.
Moreover, we denote $a\ll b$ if the estimate $b\les a$ does not hold.
For $a\in\R$, we denote by $a+$ (respectively, $a-$) a number slightly greater (respectively, slightly smaller) than $a$.
For two nonnegative numbers $a,b$, we denote $a\vee b:=\max\{a,b\}$ and $a\wedge b:=\min\{a,b\}$.
We also write $\LR{\cdot}=(1+|\cdot|^2)^{1/2}$.

For $u=u(t,x)$, $\F u=\tilde{u}$ denotes its space-time Fourier transform, whereas $\F_x u=\hat{u}$ (resp. $\F_t u$) denotes its Fourier transform in space (resp. time).
We define $D_x^s g:=\F_x^{-1}(|\xi|^s \F_x g)$ and $J_x=\LR{D_x}$.
We also denote the unitary group associated to the linear part of \eqref{eq1} by $U(t)=e^{it\p_x^2}$, i.e.,
\EQQS{
  U(t)u=\F_x^{-1}(e^{-it\xi^2}\F_x u).
}

In the present paper, we fix a smooth cutoff function $\chi$:
let $\chi\in C_c^\I(\R)$ satisfy
\EQQS{
  0\le \chi\le 1, \quad \chi|_{-1,1}=1\quad
  \textrm{and}\quad \supp\chi\subset[-2,2].
}
We set $\phi(\xi):=\chi(\xi)-\chi(2\xi)$.
For any $l\in\N\setminus\{0\} $, we define
\EQQS{
  \phi_{2^l}(\xi):=\phi(2^{-l}\xi),\quad
  \psi_{2^l}(\ta,\xi):=\phi_{2^l}(\ta+\xi^2).
}
By convention, we also denote
\EQQS{
  \phi_1(\xi)=\chi(\xi)\quad
  \textrm{and}\quad
  \psi_1(\ta,\xi)=\chi(\ta+\xi^2).
}
Any summations over capitalized variables such as $K, L, M$ and $N$ are presumed to be dyadic.
We mainly work with non-homogeneous dyadic decompositions, i.e., these variables ranges over numbers of the form $\{2^k; k\in\N\}$.
We call those numbers \textit{non-homogeneous dyadic numbers}.
It is worth pointing out that $\sum_{N\in 2^\N}\phi_N(\xi)=1$ for any $\xi\in\R$,
\EQQS{
  \supp(\phi_N)\subset\{N/2\le |\xi|\le 2N\},\ N\ge 2,\quad
  \textrm{and}\quad
   \supp(\phi_1)\subset\{|\xi|\le 2\}.
}
On the other hand, we use homogeneous decompositions only in Section \ref{nonregular}, i.e., $N\in 2^\Z$.
In this case, it holds that
\EQQS{
  \sum_{N\in 2^\Z}\phi_N(\xi)=
  \begin{cases}
    1,\quad \xi\neq0,\\
    0,\quad \xi=0\\
  \end{cases}
}
and that $\supp(\phi_N)\subset\{N/2\le |\xi|\le 2N\}$ for $N\in 2^\Z$.

Finally, we define the Littlewood--Paley multipliers $P_N$, $R_N$ and $Q_L$ by
\EQQS{
  P_N u:=\F_x^{-1}(\phi_N \F_x u),\quad
  R_N u:=\F_t^{-1}(\phi_N F_t u) \quad\textrm{and}
  \quad Q_L u=\F^{-1}(\psi_L \F u).
}
We also set $P_{\ge N}:=\sum_{K\ge N}P_K, P_{\le N}:=\sum_{K\le N}P_K, Q_{\ge L}:=\sum_{K\ge L}Q_K$ and $Q_{\le L}:=\sum_{K\le L}Q_K
$.
We represent homogeneous (resp. non-homogeneous) decompositions by $\dot{P}_N$ (resp. $P_N$).

\subsection{Function Spaces}\label{subs_funct}
For $1\le p\le \I$, $L^p(\R)$ is the standard Lebesgue space with the norm $\|\cdot\|_{L^p}$.
For $s\in\R$, $H^s(\R)$ is the Sobolev space equipped with the norm $\|f\|_{H^s}:=\|N^s \|P_N f\|_{L_x^2} \|_{\ell_N^2}$. After this section, we explicitly use the notation $L_t^p, H_x^s$, and so on, in place of $L^p, H^s$, etc., for the sake of completeness.
We also define the Zhidkov space $Y^s(\R)$ for $s\in \R$ as follows:
\EQS{\label{def_Zhid}
  Y^s(\R)
  :=\{f\in \D'(\R)\ |\ f\in L^\I(\R), f'\in H^{s-1}(\R)\}
}
with the norm $\|f\|_{Y^s}:=\|f\|_{L^\I}+\|\p_x f\|_{H^{s-1}}$.
By the Sobolev embedding, it is worth noticing that we have $H^s(\R)\hookrightarrow Y^s(\R)$ if $s>1/2$.

In this paper we will use the frequency envelope method (see for instance \cite{Tao04} and \cite{KT1}) in order to show the continuity result with respect to initial data.
To this aim, we have to slightly modulate the classical Sobolev spaces in the following way:
for $s\in\R$ and a dyadic sequence $\{\om_N\}$, we define $H_\om^s(\R)$ with the norm
\EQS{\label{def_H}
  \|u\|_{H_\om^s}
  :=\bigg(\sum_{N\in 2^\N}\om_N^2 N^{2s}\|P_N u\|_{L^2}^2\bigg)^{1/2}.
}
For consistency and with a slight abuse of notation, we equip $H^s(\R)$ with the norm given in \eqref{def_H}, setting $\om_N\equiv 1$.
If $B_x$ is one of spaces defined above, for $1\le p\le \I$ and $T>0$, we define the space-time spaces $L_t^p B_x :=L^p(\R;B_x)$ and $L_T^p B_x :=L^p([0,T];B_x)$ equipped with the norms (with obvious modifications for $p=\I$)
\EQQS{
  \|u\|_{L_t^p B_x}=\bigg(\int_\R\|u(t,\cdot)\|_{B_x}^p dt\bigg)^{1/p}\quad
  \textrm{and}\quad
  \|u\|_{L_T^p B_x}=\bigg(\int_0^T\|u(t,\cdot)\|_{B_x}^p dt\bigg)^{1/p},
}
respectively.
For $s,b\in\R$, we introduce the Bourgain spaces $X^{s,b}$ associated to the operator $i\p_x^2$ endowed with the norm
\EQQS{
  \|u\|_{X^{s,b}}
  =\Bigg(\int_{\R^2} \LR{\xi}^{2s}\LR{\ta+\xi^2}^{2b}|\tilde{u}(\ta,\xi)|^2d\ta d\xi\Bigg)^{1/2}.
}
We also use a slightly stronger space $X_\om^{s,b}$ with the norm
\EQQS{
  \|u\|_{X_\om^{s,b}}
  :=\bigg(\sum_{N\in 2^\N}\om_N^2 N^{2s}\|P_N u\|_{X^{0,b}}^2\bigg)^{1/2}.
}
Hereafter, for simplicity, we will write $\sum_{N\in 2^\N}$ as $\sum_{N\ge 1}$ whenever there is no risk of confusion.
We define the function spaces $Z^s $ (resp. $Z^s_\om $), with $s\in \R$, as $Z^s:= L_t^\I H^s\cap X^{s-1,1}$ (resp. $Z^s_\om:= L_t^\I H_\om^s\cap X_\om^{s-1,1}$), endowed with the natural norm
\EQQS{
  \|u\|_{Z^s}=\|u\|_{L_t^\I H^s}+\|u\|_{X^{s-1,1}} \quad
  (\text{resp}.\  \|u\|_{Z^s_\om}=\|u\|_{L_t^\I H^s_\om}+\|u\|_{X^{s-1,1}_\om}) .
}
We also use the restriction in time versions of these spaces.
Let $T>0$ be a positive time and $B$ be a normed space of space-time functions.
The restriction space $B_T$ will be the space of functions $u:[0,T]\times\R\to\R$ or $\C$ satisfying
\EQQS{
  \|u\|_{B_T}
  :=\inf\{\|\tilde{u}\|_B \ |\ \tilde{u}:\R^2 \to\R\ \textrm{or}\ \C,\ \tilde{u}=u\ \textrm{on}\ [0,T]\times\R\}<\I.
}

Finally, we introduce a bounded linear operator from $X_{\om,T}^{s-1,1}\cap L_T^\I H_\om^s$ into $Z_\om^s$ with a bound which does not depend on $s$ and $T$. The existence of this operator ensures that actually $ Z^s_{\om,T}= L_T^\I H^s_\om\cap X^{s-1,1}_{\om,T}$.
Following \cite{MN08}, we define $\rho_T$ as
\EQS{\label{def_ext}
  \rho_T(u)(t):=U(t)\chi(t)U(-\mu_T(t))u(\mu_T(t)),
}
where $\mu_T$ is the continuous piecewise affine function defined by
\EQQS{
  \mu_T(t)=
  \begin{cases}
    0 &\textrm{for}\quad t\notin]0,2T[,\\
    t &\textrm{for}\quad t\in [0,T],\\
    2T-t &\textrm{for}\quad t\in [T,2T].
  \end{cases}
}

\begin{lem}\label{extensionlem}
  Let $\de\ge 1$, and suppose that the dyadic sequence $\{\om_N\}$ of positive numbers satisfies $\om_N\le \om_{2N}\le \de\om_N$ for $N\ge1$.
  Let $0<T\le 1$ and $s\in\R$.
  Then,
  \EQQS{
    \rho_T:&X_{\om,T}^{s-1,1} \cap L_T^\I H_\om^s\to Z^s_\om\\
    &u\mapsto \rho_T(u)
  }
  is a bounded linear operator, i.e.,
  \EQS{\label{eq2.1}
    \|\rho_T(u)\|_{L_t^\I H_\om^s}
    +\|\rho_T(u)\|_{X^{s-1,1}_\om}\les
    \|u\|_{L_T^\I H_\om^s}
    +\|u\|_{X_{\om,T}^{s-1,1}},
  }
  for all $u\in X_{\om,T}^{s-1,1}\cap L_T^\I H_\om^s$.
  Moreover, it holds that
  \EQS{\label{eq2.1single}
  \|\rho_T(u)\|_{L_t^\I H_\om^s}
  \les
  \|u\|_{L_T^\I H_\om^s}
  }
  for all $u\in L_T^\I H_\om^s$.
  Here, the implicit constants in \eqref{eq2.1} and \eqref{eq2.1single} can be chosen independent of $0<T\le 1$ and $s\in\R$.
\end{lem}

\begin{proof}
  See Lemma 2.4 in \cite{MPV19} for $ \om_N\equiv 1$ but it is obvious that the result does not depend on $ \om_N$.
\end{proof}

\subsection{Basic Estimates}

In this subsection, we collect some fundamental estimates.
For details of the following two estimates, see \cite{LP} for instance.

\begin{lem}
  Let $s\in(0,1)$ and $1\le p<\I$.
  Assume that $1< p_1,p_2,q_1,q_2\le \I$ satisfy $1/p=1/p_1+1/p_2=1/q_1+1/q_2$.
  Then we have the estimate
  \EQS{\label{eq2.6}
    \|D_x^{s}(uv)\|_{L^p}
    \les \|J_x^s u\|_{L^{p_1}}\|v\|_{L^{p_2}}+\|u\|_{L^{q_1}}\|J_x^sv\|_{L^{q_2}}.
  }
\end{lem}

\begin{lem}
  Let $s\ge0$.
  Then we have the estimate
  \EQS{\label{eq2.7}
    \|J_x^s(uv)\|_{L^{\I}}
    \les \|J_x^{s+}u\|_{L^{\I}}\|J_x^{s+}v\|_{L^{\I}}.
  }
\end{lem}

Well-known estimates are adapted for our setting $H_\om^s(\R)$.

\begin{lem}\label{Lem24}
  Let $s>0$ and $\de\ge 1$.
  Suppose that the dyadic sequence $\{\om_N\}$ of positive numbers satisfies $\om_N\le \om_{2N}\le \de\om_N$ for $N\ge1$.
  Then we have the estimate
  \EQS{\label{eq2.2}
    \|uv\|_{H_\om^s}
    \les \|u\|_{H_\om^s}\|v\|_{L^\I}+\|u\|_{L^\I}\|v\|_{H_\om^s}.
  }
  Moreover, if $\de<2^\e$ for a given $\e>0$, then
  \EQS{
    \label{eq2.2.1}
    \|uv\|_{H_\om^s}
    \les \|u\|_{H_\om^s}\|v\|_{L^\I}+\|u\|_{L^2}\|J_x^{s+}v\|_{L^{\I}}.
  }
 \end{lem}

\begin{proof}
  The proof of \eqref{eq2.2} essentially follows from that of Lemma A.8 in \cite{Tao}. (See also Lemma 2.2 in \cite{MT21} for the proof \eqref{eq2.2}.)
  On the other hand, the proof of \eqref{eq2.2.1} is similar to \eqref{eq2.2} if we notice that
  \EQQS{
    \bigg(\sum_{N\gg1}\om_N^2 N^{2s}\|P_{N}v\|_{L^\I}^2\bigg)^{1/2}
    \les \|J_x^{s+}v\|_{L^{\I}}.
  }
  Here, we canceled $\om_N$ and the summation over $N$ by introducing $N^{-(0+)}$.
  This completes the proof.
\end{proof}

\begin{rem}
  In \eqref{eq2.2.1}, we lose $\e$ regularity for the $L^\I$ norm.
  However, when $\om_N\equiv 1$, Grafakos-Oh \cite{GO14} proved an estimate without a loss, i.e., for $s>0$,
  \EQS{\label{estGO14}
   \|uv\|_{H^s}
   \les \|u\|_{H^s}\|v\|_{L^\I}+\|u\|_{L^2}\|J_x^s v\|_{L^\I}.
  }
  To be precise, the above estimate \eqref{estGO14} is a special case of Theorem 1 in \cite{GO14}.
\end{rem}

To bound nonlinear terms at low regularity we will also need classical product estimates in Sobolev spaces (see, for instance, \cite{Lanlivre13}).
Below, we provide a version with a frequency envelope that can be proven in the same way as the classical one.

\begin{lem}\label{Lem25}
  Let $\de\ge 1$, and suppose that the dyadic sequence $\{\om_N\}$ of positive numbers satisfies $\om_N\le \om_{2N}\le \de\om_N$ for $N\ge1$.
  Assume that $s_1+s_2\ge 0, s_1\wedge s_2\ge s_3, s_3< s_1+s_2-1/2$.
  Then
  \EQS{\label{eq2.3}
    \|uv\|_{H_\om^{s_3}}\les \|u\|_{H_\om^{s_1}}\|v\|_{H_\om^{s_2}}.
  }
\end{lem}

Similarly, we obtain the following estimate, which will be used to estimate $|\psi|^2 w$ in $H^{-1/4+}$ in the proof of \eqref{estZ1}.
We remark that it is important to take $s_2=s_3$ when we apply \eqref{eq_2.3apl4}.

\begin{cor}\label{cor26}
  Let $s_1+s_2>s_3\vee 0$, $s_1\ge 0$, $ s_1>s_3$, and $s_2\ge s_3$
  Then
  \EQS{\label{eq_2.3apl4}
    \|uv\|_{H^{s_3}}\les \|J_x^{s_1}u\|_{L^{\I}}\|v\|_{H^{s_2}}.
  }
\end{cor}

\begin{rem}
  An immediate application of \eqref{eq2.3} and \eqref{eq2.2.1} is the following:
  \EQS{\label{eq_2.3apl}
    \|F(u,\psi)\|_{H^{-1/4}}
    \les (\|u\|_{H^{3/4}}+\|J_x\psi\|_{L^{\I}})^2\|u\|_{H^{3/4}},
  }
  where $F(u,\psi)$ is defined in \eqref{def_F}.
  Indeed, for $|u|^2 \p_x u$, we see from \eqref{eq2.3} that $\||u|^2 \p_x u\|_{H^{-1/4}}
  \les \||u|^2\|_{H^{3/4}}\|\p_x u\|_{H^{-1/4}}
  \les \|u\|_{H^{3/4}}^3$.
  For terms such as $|\psi|^2\p_x u$ (i.e., terms not in divergence form), we notice that $|\psi|^2\p_x u=\p_x(|\psi|^2 u)-\bar{\psi}\p_x \psi u-\psi \p_x \bar{\psi}u$.
  We can treat the second and the third terms easily with an trivial embedding $L^2\hookrightarrow H^{-1/4}.$
  For the first term, we use \eqref{eq2.2.1}.
  By a similar argument, we obtain
  \EQS{\label{eq_2.3apl1}
    \|F(u,\psi)\|_{H_\om^{s-1}}
    \les (\|u\|_{H^{3/4}}
     +\|J_x^{s+1}\psi\|_{L^{\I}})^2\|u\|_{H_\om^{s}}
  }
  for $s\ge 3/4$ and $\{\om_N\}_N$ satisfies $\om_N\le \om_{2N}\le \de \om_N$ with $\de<2^{\e/2}$.
\end{rem}

We will frequently use the following lemma, which can be seen as a variant of integration by parts.

\begin{lem}\label{lem_comm1}
  Let $N\in 2^{\Z}\cup\{0\}$.
  Then,
  \EQQS{
    \bigg|\int_\R \Pi_N(u,v)wdx\bigg|
    +\bigg|\int_\R \ti{\Pi}_N(u,v)wdx\bigg|
    \les\|u\|_{L^2}\|v\|_{L^2}\|\p_x w\|_{L^\I},
  }
  where
  \EQS{\label{def_pi}
    \Pi_N(u,v)&:=(\p_x v) P_N^2u +(\p_x u) P_N^2v,\\
    \label{def_pi2}
    \ti{\Pi}_N(u,v)&:=v \p_x P_N^2u +u \p_x P_N^2v.
  }
\end{lem}

\begin{proof}
  We only show the estimate of \eqref{def_pi} since the estimate of \eqref{def_pi2} follows from an almost parallel argument.
  Observe that
  \EQQS{
    \int_\R \Pi_N(u,v)wdx = \int_\R u\{P_N^2(w\p_x v)-\p_x (wP_N^2 v)\}dx.
  }
  So, it suffices to show that
  \EQS{
    \|P_N^2(w\p_x v)-\p_x (wP_N^2 v)\|_{L^2}
    \les \|\p_x w\|_{L^\I}\|v\|_{L^2}.
  }
  Note that
  \EQQS{
    P_N^2(w\p_x v)-\p_x (wP_N^2 v)
    =[\p_x P_N^2, w]v-P_N^2(v\p_x w)-(P_N^2 v) \p_x w.
  }
  For the first term, as in Lemma 2.4 in \cite{MT21}, we obtain that
  \EQQS{
    |([\p_x P_N^2, w]v)(x)|
    &=\bigg|\int_\R \F_x^{-1}(\xi \phi_N^2(\xi))(x-y)
      (w(y)-w(x))v(y)dy\bigg|\\
    &\les \|\p_x w\|_{L^\I}\int_\R|(x-y)\F_x^{-1}(\xi \phi_N^2(\xi))(x-y)v(y)|dy,
  }
  which together with the Minkowski inequality and the change of variables implies that
  \EQS{\label{eq_comm2}
    \|[\p_x P_N^2, w]v\|_{L^2}
    \les \|\p_x w\|_{L^\I}\|v\|_{L^2}.
  }
  Other terms $\|P_N^2(v\p_x w)\|_{L^2}$ and $\|(P_N^2 v) \p_x w\|_{L^2}$ can be bounded by $\|\p_x w\|_{L^\I}\|v\|_{L^2}$ easily.
  This completes the proof.
\end{proof}

\begin{rem}
  If we replace $w$ with $P_{\ll N}w$, $v$ automatically has the Fourier support $\supp\hat{v}(\xi)\subset\{|\xi|\sim N\}$ by the impossible frequency interaction.
  Then, \eqref{eq_comm2} becomes
  \EQQS{
  \|[\p_x P_N, P_{\ll N} w]v\|_{L^2}
  =\|[\p_x P_N, P_{\ll N} w]P_{\sim N}v\|_{L^2}
  \les \|\p_x P_{\ll N} w\|_{L^\I}\|P_{\sim N}v\|_{L^2}.
  }
  In particular, we have
  \EQS{\label{eq_comm3}
   \|[P_N, P_{\ll N} w]\p_x v\|_{L^2}
   \les \|\p_x P_{\ll N} w\|_{L^\I}\|P_{\sim N}v\|_{L^2}
  }
  since $[P_N, P_{\ll N} w]\p_x v=[\p_x P_N, P_{\ll N} w]P_{\sim N}v- P_N( \p_x P_{\ll N} w P_{\sim N} v)$.
\end{rem}

\section{Refined Strichartz Estimates}\label{sec_stri}

In this section, we establish refined Strichartz estimates, developed by \cite{KT1}, which play an important role in our estimates.
Specifically, we use Proposition \eqref{prop_stri1} for resonant interactions.
See Case 1 in the estimates for $J_t^{(1)}$ in the proof of Proposition \ref{prop_apri}, for example.
First we recall the standard Strichartz estimate (see, for instance, Theorem 2.3.3 in \cite{C03}).

\begin{defn}
  A pair $(p,q)\in\R^2$ is called admissible if $2\le p\le \I$ and
  \EQQS{
    \frac{2}{q}=\frac{1}{2}-\frac{1}{p}.
  }
\end{defn}

\begin{lem}\label{lem_stri}
  Let $(p,q)$ be admissible.
  Then, for any $u\in L^2(\R)$
  \EQQS{
    \|e^{it\p_x^2}u\|_{L_t^q L_x^p}\les \|u\|_{L_x^2}.
  }
\end{lem}

\begin{prop}\label{prop_stri1}
  Let $0<T<1$, $N\ge 1$.
  Suppose that $(p,q)\in\R^2$ is admissible and $1\le \tilde{q}\le q$.
  Assume that $u\in C([0,T];L^2(\R))$ satisfies
  \EQQS{
    i\p_t u+\p_x^2 u= F
  }
  on $[0,T]$ for some $F\in L^{\tilde{q}}([0,T];L^2(\R))$.
  Then,
  \EQQS{
    \|P_N u\|_{L_T^{\tilde{q}}L_x^p}
    \les N^{\frac 1q}T^{-\frac 1q}\|P_N u\|_{L_T^{\tilde{q}}L_x^2}
     +N^{\frac 1q -1}T^{1-\frac 1q}\|P_N F\|_{L_T^{\tilde{q}}L_x^2}.
  }
\end{prop}

\begin{proof}
  For each $N$, define a series of time intervals $\{I_{j,N}\}_{j\in J}$ so that $\bigcup_{j\in J_N}I_{j,N}=[0,T]$ and $|I_{j,N}|\sim N^{-1}T$ and $\#J\sim N$.
  On $t\in I_{j,N}$, we have
  \EQQS{
    P_N u(t)
    =e^{i(t-c_{j,N})\p_x^2}P_N u(c_{j,N})
      - i \int_{c_{j,N}}^t e^{i(t-t')\p_x^2} P_N
      F dt',
  }
  where $c_{j,N}$ is chosen so that $\|P_{N}u(t)\|_{L_x^2}^{\tilde{q}}$ attains its minimum on $I_{j,N}$.
  Then, we see from the Minkowski inequality that
  \EQQS{
    \|P_N u\|_{L_T^{\tilde{q}}L_x^p}
    &\le \bigg[\sum_{j\in J_N}
      \| e^{i(t-c_{j,N})\p_x^2}P_N u(c_{j,N})
       \|_{L^{\tilde{q}}(I_{j,N};L_x^p)}^{\tilde{q}} \bigg]^{1/\tilde{q}}\\
    &\quad +\bigg[\sum_{j\in J_N}
      \bigg\|\int_{c_{j,N}}^t e^{i(t-t')\p_x^2}P_N F(t') dt'
       \bigg\|_{L^{\tilde{q}}(I_{j,N};L_x^p)}^{\tilde{q}} \bigg]^{1/\tilde{q}}
    =:A_1+A_2.
  }
  The H\"older inequality in time and the Strichartz estimate (Lemma \ref{lem_stri}) show that
  \EQQS{
    A_1
    &\le\bigg[\sum_{j\in J_N}
      |I_{j,N}|^{1-\tilde{q}/q}
       \| e^{i(t-c_{j,N})\p_x^2}P_N u(c_{j,N})
       \|_{L^{q}(I_{j,N};L_x^p)}^{\tilde{q}} \bigg]^{1/\tilde{q}}\\
    &\les \bigg[\sum_{j\in J_N}
      |I_{j,N}|^{1-\tilde{q}/q}
      \|P_N u(c_{j,N})
       \|_{L_x^{2}}^{\tilde{q}} \bigg]^{1/\tilde{q}}\\
    &\les N^{\frac 1q} T^{-\frac 1q} \bigg[\sum_{j\in J_N}
      \int_{I_{j,N}} \| P_N u(t')
       \|_{L_x^{2}}^{\tilde{q}}dt' \bigg]^{1/\tilde{q}}
    =N^{\frac 1q}T^{-\frac 1q}\|P_N u\|_{L_T^{\tilde{q}} L_x^2}.
  }
  Similarly, for $A_2$, we see from the H\"older inequality in time, the Minkowski inequality and the Strichartz estimate (Lemma \ref{lem_stri}) that
  \EQQS{
    A_2
    &\le\bigg[\sum_{j\in J_N}
      |I_{j,N}|^{1-\tilde{q}/q}
      \bigg(\int_{I_{j,N}}
      \| e^{i(t-t')\p_x^2}P_N F(t')\|_{L^q(I_{j,N};L_x^p)} dt'
       \bigg)^{\tilde{q}} \bigg]^{1/\tilde{q}}\\
    &\les \bigg[\sum_{j\in J_N} |I_{j,N}|^{\tilde{q}(1-1/q)}
      \int_{I_{j,N}} \| P_N F(t')\|_{L_x^2}^{\tilde{q}} dt'
      \bigg]^{1/\tilde{q}}
    \les N^{\frac 1q -1}T^{1-\frac 1q}\|F\|_{L_T^{\tilde{q}} L_x^2},
  }
  which completes the proof.
\end{proof}

\section{A Priori Estimate}\label{sec_apri}

In this section, we derive Proposition \ref{prop_apri}, which is one of the main estimates in this article.
For that purpose, we first collect technical tools which are based on the argument of \cite{MV15}.

\subsection{Preliminary Technical Estimates}

Let us denote by $\1_T$ the characteristic function of the interval $[0,T]$.
As pointed out in \cite{MV15}, $\1_T$ does not commute with $Q_L$.
To avoid this difficulty, following \cite{MV15}, we further decompose $\1_T$ as
\EQS{
  \1_T=\1_{T,R}^{\mathrm{low}}+\1_{T,R}^{\mathrm{high}},
  \quad \mathrm{with} \quad
  \F_t(\1_{T,R}^{\mathrm{low}})(\ta)=\chi(\ta/R)\F_t(\1_T)(\ta),
}
for some $R>0$ to be fixed later.
Note that $\1_{T,R}^{\mathrm{low}}(t),\1_{T,R}^{\mathrm{high}}(t) \in \R$.

\begin{lem}[Lemma 3.5 in \cite{MPV19}]
  Let $1\le p\le \I$ and let $L$ be a non-homogeneous dyadic number.
  Then the operator $Q_{\le L}$ is bounded in $L_t^p L_x^2$ uniformly in $L$.
  In other words,
  \EQS{\label{eq4.3}
    \|Q_{\le L}u\|_{L_t^p L_x^2}\les \|u\|_{L_t^p L_x^2},
  }
  for all $u\in L_t^p L_x^2$ and the implicit constant appearing in \eqref{eq4.3} does not depend on $L$.
\end{lem}

\begin{lem}[Lemma 3.6 in \cite{MPV19}]
  For any $R>0$ and $T>0$, it holds
  \EQS{\label{eq4.1}
    \|\1_{T,R}^{\mathrm{high}}\|_{L^1}&\les T\wedge R^{-1},\\
    \label{eq4.1.1}
    \|\1_{T,R}^{\mathrm{low}}\|_{L^1}&\les T
  }
  and
  \EQS{\label{eq4.2}
    \|\1_{T,R}^{\mathrm{high}}\|_{L^\I}
    +\|\1_{T,R}^{\mathrm{low}}\|_{L^\I}\les 1.
  }
\end{lem}

\begin{lem}[Lemma 3.7 in \cite{MPV19}]
  Assume that $T>0$, $R>0$, and $L\gg R$.
  Then it holds that
  \EQS{\label{eq4.6}
    \|Q_L (\1_{T,R}^{\mathrm{low}}u)\|_{L_{t,x}^2}
    \les \|Q_{\sim L}  u\|_{L_{t,x}^2},
  }
  for all $u\in L^2(\R_t \times \R_x)$.
\end{lem}

\begin{defn}
  Define $\Om:\R^{4}\to\R$ as
  \EQQS{
    \Om(\xi_1,\xi_2,\xi_3,\xi_{4})
    :=\xi_1^2-\xi_2^2+\xi_3^2-\xi_4^2.
  }
  for $(\xi_1,\xi_2,\xi_3,\xi_{4})\in\R^{4}$.
\end{defn}

\begin{lem}\label{lem_res1}
  Let $(\xi_1,\xi_2,\xi_3,\xi_{4})\in\R^{4}$ satisfy $\xi_1-\xi_2+\xi_3-\xi_4=0$.
  Assume that $|\xi_1|\sim |\xi_2|\gts |\xi_3|\gg|\xi_4| $ or $|\xi_1|\sim |\xi_2|\gts |\xi_4|\gg|\xi_3| $.
  Then,
  \EQS{\label{eq_res1}
    |\Om(\xi_1,\xi_2,\xi_3,\xi_{4})|\gts |\xi_1|(|\xi_3|\vee |\xi_4|).
  }
  On the other hand, when $|\xi_1|\sim |\xi_3|\gts |\xi_2|\gg |\xi_4|$, $|\xi_1|\sim |\xi_3|\gts |\xi_4|\gg |\xi_2|$ or $|\xi_1|\sim |\xi_3|\gg |\xi_2|\vee |\xi_4|$,
  \EQS{\label{eq_res2}
    |\Om(\xi_1,\xi_2,\xi_3,\xi_{4})|
    \gts |\xi_1|^2.
  }
\end{lem}

\begin{proof}
  First we treat \eqref{eq_res1} with $|\xi_3|\gg |\xi_4|$.
  Since $\xi_2=\xi_1+\xi_3-\xi_4$, we see that
  \EQQS{
    |\Om(\xi_1,\xi_2,\xi_3,\xi_{4})|
    =2|(\xi_1-\xi_4)(\xi_3-\xi_4)|
    \gts |\xi_1||\xi_3|
  }
  since $|\xi_3|\gg|\xi_4|$.
  By the same argument, we can show \eqref{eq_res1} with $|\xi_4|\gg|\xi_3|$.
  The estimate \eqref{eq_res2} follows similarly.
\end{proof}

\begin{lem}\label{resonance}
For $ j=1,2,3,4$, let  $L, L_j, N_j \ge 1 $ be dyadic numbers and let $ u_j\in {\mathcal S}(\R^2) $ with $ u_j=P_{N_j} u_j $.
Assume that for
\begin{equation}\label{za}
\Om(\xi_1,\xi_2,\xi_3,\xi_4) \sim L \gg 1
\end{equation}
for any $(\xi_1,\xi_2,\xi_3,\xi_4)\in \R^4 $ such that for $j=1,2,3,4$, $ |\xi_j|\sim N_j $ whenever $ N_j\neq 1 $ and $ |\xi_j|\les N_j $ otherwise.
Then
$$
I:=\int_{\R^2} Q_{L_1}u_1 \, \overline{Q_{L_2} u_2 }\, _{L_3} u_3 \,\overline{Q_{L_4} u_4}\, dxdt=0
$$
unless $ \max L_j \gts L $.
\end{lem}
\begin{proof} By the Plancherel theorem, we have
\EQQS{
  I=\int_{\substack{\tau_1-\tau_2+\tau_3-\tau_4=0,\\ \xi_1-\xi_2+\xi_3-\xi_4=0}}
   \widehat{Q_{L_1} u_1}(\tau_1,\xi_1) \overline{\widehat{Q_{L_2}u_2}}(\tau_2,\xi_2) \widehat{Q_{L_3} u_3}(\tau_3,\xi_3)
    \overline{\widehat{Q_{L_4} u_4}}(\tau_4,\xi_4).
}
Since for $\tau_1-\tau_2+\tau_3-\tau_4=0 $, it holds that
\EQQS{
  (\tau_1+\xi_1^2) -(\tau_2+\xi_2^2)+(\tau_3+\xi_3^2) -(\tau_4+\xi_4^2) =\Om(\xi_1,\xi_2,\xi_3,\xi_4).
}
The result follows by making use of \eqref{za} and the properties of the projectors $ Q_{L_j} $.
This completes the proof.
\end{proof}

\subsection{Estimates for Solutions to \eqref{eq1}}

\begin{lem}\label{lem47}
Let $\e>0$.
Suppose that the dyadic sequence $\{\om_N\}$ of positive numbers satisfies $\om_N\le \om_{2N}\le 2^\e\om_N$ for $N\ge1$.
Let $ 0<T<1$, $ s\ge 3/4$  and $ u\in L^\infty_T H^s_\om $  be a solution to \eqref{eq1} associated with an initial datum
 $ u_0\in H^s_\om(\R) $.
Assume that there exists $K>0$ such that
\EQQS{
  \|u\|_{L_T^{\I}H_\om^{3/4}}+\|J_x^{s+1}\psi\|_{L_{T,x}^\I}
     \le K.
}
Then $ u\in Z^s_{\om,T}$ and it holds
\begin{equation}\label{estXregular}
  \|u\|_{Z^s_{\om,T}}
  \le C(K)(\|u\|_{L^\infty_T H^s_\om} +\|\Psi\|_{L_T^\I H_x^{s-1}}).
\end{equation}
Moreover, for any pair $(u_1, u_2)\in (L^\infty_T H^s)^2 $ of solutions
to \eqref{eq1} associated with a pair of initial data $ (u_{0,1}, u_{0,2})\in (H^s(\R))^2 $ it holds that for $3/2-s\le \te\le s$
\begin{equation}\label{estdiffXregular}
  \|u_1-u_2\|_{Z^{\te-1/2}_{T}}
  \le C(\|u_1\|_{L^\infty_T H_x^{s}}, \|u_2\|_{L^\infty_T H_x^{s}},
   \|J_x^{s+1}\psi\|_{L_{T,x}^\I})
   \|u_1-u_2\|_{L^\infty_T H_x^{\te-1/2}}.
\end{equation}
\end{lem}

\begin{proof}
According to the extension Lemma \ref{extensionlem}, it is clear that we only have to estimate the $ X^{s-1,1}_{\om,T} $-norm of $ u $ to prove \eqref{estXregular}.
As noticed in Remark \ref{rem_sol} (see also Remark 1.4 of \cite{MT21}),
$ u $ satisfies the Duhamel formula of
\eqref{eq1} and $\|u_0\|_{H^\theta_\om} \le \|u\|_{L^\infty_T H^\theta_\om} $ for any $ \theta\le s $.
Hence, standard linear estimates in Bourgain's spaces lead to
\EQQS{
  \|u\|_{X^{s-1,1}_{\om,T}}
  &\les  \|u_0\|_{H^{s-1}_\om}
    +\| F(u,\psi) \|_{X^{s-1,0}_{\om,T}} + \| \Psi \|_{X^{s-1,0}_{\om,T}} \\
  &\les \|u_0\|_{H^{s-1}_\om} +\| F(u,\psi) \|_{L_T^2 H_\om^{s-1}}
   + \| \Psi \|_{L_T^2 H_\om^{s-1}}\\
  &\le C(K) (\| u \|_{L^\infty_T H^s_\om} +\|\Psi\|_{L_T^\I H_x^{s-1}}),
}
by using \eqref{eq2.2} and \eqref{eq2.2.1} (see also \eqref{eq_2.3apl}).
In the same way, using \eqref{eq2.2} and \eqref{eq2.2.1}, we get
\EQQS{
  \|u_1-u_2\|_{X^{\te-3/2,1}_T}
  &\les \|u_{0,1}-u_{0,2}\|_{H_x^{\te-3/2}}
   +\|F(u_1,\psi)-F(u_2,\psi)\|_{L^2_T H_x^{\te-3/2}} \\
  &\le C(\|u_1\|_{L^\infty_T H_x^{s}}, \|u_2\|_{L^\infty_T H_x^{s}},
   \|J_x^{s+1}\psi\|_{L_{T,x}^\I})
   \|u_1-u_2\|_{L^\infty_T H_x^{\te-1/2}},
}
 which completes the proof.
\end{proof}

The following four lemmas are used in the proofs of Propositions \ref{prop_apri}, \ref{prop_dif1}, and \ref{prop_dif2}.
In particular, these estimates allow us to derive the desired estimates for nonresonant interactions.

\begin{lem}\label{lem_Bourgain1}
  Let $0<T<1$ and $N,N_1,N_2,N_3,N_4\ge 1$.
  Assume that $u_k\in Z^0$ is a function with spatial Fourier support in $\{|\xi_k|+1\sim N_k\}$ for $k=1,2,3,4$.
  Moreover, assume that $N_1\sim N_2\gts N_3\gg N_4$.
  For any $t>0$, we set
  \EQS{\label{def_I}
    I_t(u_1,u_2,u_3,u_4)
    :=\RE\int_0^t \int_\R
      \Pi_N(u_1,u_2)u_3u_4dxdt'.
  }
  Then it holds that for $0<t<T$
  \EQS{\label{eq_I1}
    \begin{aligned}
      &|I_t(\bar{u}_1,\bar{u}_2,u_3,u_4)|
       +|I_t(u_1,\bar{u}_2,\bar{u}_3,u_4)|
       +|I_t(u_1,\bar{u}_2,u_3,\bar{u}_4)|\\
      &\les (N_3N_4)^{1/2}
       \|u_3\|_{L_t^\I L_x^2}\|u_4\|_{L_t^\I L_x^2}
       (T^{1/4}N_1^{-1/4} \|u_1\|_{L_t^\I L_x^2}\|u_2\|_{L_t^\I L_x^2}\\
      &\quad\quad+\|u_1\|_{X^{-1,1}}
       \|\1_t u_2\|_{L_{t,x}^2}
       +T^{1/4}N_1^{-1/12}N_3^{-1/3}
        \|u_1\|_{X^{-1,1}}
        \|u_2\|_{L_{t}^\I L_x^2}\\
      &\quad\quad+\|\1_t u_1\|_{L_{t,x}^2}
        \|u_2\|_{X^{-1,1}}
       +T^{1/4}N_1^{-1/12}N_3^{-1/3}
        \|u_1\|_{L_{t}^\I L_x^2}
        \|u_2\|_{X^{-1,1}})\\
      &\quad+T^{1/2}N_1^{-1}(N_3N_4)^{1/2}
        \|u_1\|_{L_t^\I L_x^2}\|u_2\|_{L_t^\I L_x^2}
        (N_3\|u_3\|_{X^{-1,1}}\|u_4\|_{L_t^\I L_x^2}\\
      &\quad\quad+N_4\|u_3\|_{L_t^\I L_x^2}\|u_4\|_{X^{-1,1}}).
    \end{aligned}
  }
\end{lem}

\begin{proof}
  We first prove the estimate for $I_t(u_1,\bar{u}_2,u_3,\bar{u}_4)$.
  Setting $R=N_1^{1/3}N_3^{4/3}$, we split $I_t$ as
  \EQS{\label{eq_4.11}
    \begin{aligned}
      |I_t(u_1,\bar{u}_2,u_3,\bar{u}_4)|
      &\le |I_{\I}( \1_{t,R}^{\text{high}} u_1, \1_t \bar{u}_2, u_3, \bar{u}_4)|
       +|I_{\I}( \1_{t,R}^{\text{low}} u_1, \1_{t,R}^{\text{high}} \bar{u}_2,u_3, \bar{u}_4)|\\
      &\quad+|I_{\I}( \1_{t,R}^{\text{low}} u_1, \1_{t,R}^{\text{low}} \bar{u}_2,u_3, \bar{u}_4)|
      =:I_{\I,1}+I_{\I,2}+I_{\I,3}.
    \end{aligned}
  }
  For $I_{\I,1}$, we see from \eqref{eq4.1} that $\|\1_{t,R}^{\text{high}}\|_{L_t^1}\les T^{1/4}N_1^{-1/4}N_3^{-1}$,
  which together with Lemma \ref{lem_comm1} implies that
  \EQS{\label{eq_time}
  \begin{aligned}
    I_{\I,1}
    &\les N_3
    \|\1_{t,R}^{\text{high}}\|_{L_t^1}
    \|u_1\|_{L_t^\I L_x^2}\|u_2\|_{L_t^\I L_x^2}
    \|u_3\|_{L_{t,x}^\I}\|u_4\|_{L_{t,x}^\I}\\
    &\les T^{1/4}N_1^{-1/4}N_3^{1/2}N_4^{1/2}
     \prod_{j=1}^4\|u_j\|_{L_t^\I L_x^2}.
  \end{aligned}
  }
  Using \eqref{eq4.2}, we can estimate $I_{\I,2}$ by the same bound.
  Next, we estimate $I_{\I,3}$.
  Lemma \ref{lem_res1} shows that $|\Om|\gts N_1N_3\gg R$ since $N_1\gg1$.
  By using this property, defining $L:= N_1N_3$, Lemma \ref{resonance} leads to
  \EQS{\label{eq_4.12}
    \begin{aligned}
      I_{\I,3}
      &\le |I_{\I}(Q_{\gts L} (\1_{t,R}^{\text{low}} u_1), \overline{\1_{t,R}^{\text{low}} u_2},u_3, \bar{u}_4)|\\
      &\quad+|I_{\I}(Q_{\ll L} (\1_{t,R}^{\text{low}} u_1),
       \overline{Q_{\gts L}(\1_{t,R}^{\text{low}}u_2)},u_3, \bar{u}_4)|\\
      &\quad +|I_{\I}(Q_{\ll L}^+ (\1_{t,R}^{\text{low}} u_1),
       \overline{Q_{\ll L}(\1_{t,R}^{\text{low}}u_2)},
       Q_{\gts L} u_3, \bar{u}_4)|\\
      &\quad +|I_{\I}(Q_{\ll L}^+ (\1_{t,R}^{\text{low}} u_1),
        \overline{Q_{\ll L}(\1_{t,R}^{\text{low}}u_2)},
       Q_{\ll L} u_3, \overline{Q_{\gts L} u_4})|\\
      &=:I_{\I,3,1}+I_{\I,3,2}+I_{\I,3,3}+I_{\I,3,4}.
    \end{aligned}
  }
where we used that $\1_{t,R}^{\text{low}}$ is real-valued.
  We also see from \eqref{eq4.1} that for a function $u$
  \EQS{\label{eq4.7}
    \begin{aligned}
      \|\1_{t,R}^{\mathrm{low}}u\|_{L_{t,x}^2}
      \le\|\1_{t}u\|_{L_{t,x}^2}
        +\|\1_{t,R}^{\mathrm{high}}u\|_{L_{t,x}^2}
      \les \|\1_{t}u\|_{L_{t,x}^2}
        +T^{1/4}R^{-1/4}\|u\|_{L_{t}^\I L_x^2}.
    \end{aligned}
  }
  For $I_{\I,3,1}$, the H\"older inequality,  Lemma \ref{lem_comm1} and \eqref{eq4.1.1}, \eqref{eq4.6}, \eqref{eq4.7} and \eqref{eq2.1} give
  \EQQS{
      I_{\I,3,1}
      &\les N_3
      \|Q_{\gts L}(\1_{t,R}^{\text{low}}u_1)\|_{L_{t,x}^2}
      (\|\1_{t}u_2\|_{L_{t,x}^2}+\|\1_{t,R}^{\text{high}}u_2\|_{L_{t,x}^2})
      \|u_3\|_{L_{t,x}^\I}\|u_4\|_{L_{t,x}^\I}\\
      &\les N_3^{1/2}N_4^{1/2}\|u_1\|_{X^{-1,1}}
       \|\1_{t}u_2\|_{L_{t,x}^2}
       \|u_3\|_{L_{t}^\I L_x^2}\|u_4\|_{L_{t}^\I L_x^2}\\
      &\quad+ T^{1/4} N_1^{-1/12}N_3^{1/6}N_4^{1/2}
       \|u_1\|_{X^{-1,1}}
       \|u_2\|_{L_{t}^\I L_x^2}
       \|u_3\|_{L_{t}^\I L_x^2}\|u_4\|_{L_{t}^\I L_x^2}.
  }
  By the same way as in \eqref{eq4.3}, we can obtain
  \EQQS{
    I_{\I,3,2}
    &\les N_3^{1/2}N_4^{1/2}\|\1_t u_1\|_{L_{t,x}^2}
     \|u_2\|_{X^{-1,1}}
     \|u_3\|_{L_{t}^\I L_x^2}\|u_4\|_{L_{t}^\I L_x^2}\\
    &\quad+ T^{1/4} N_1^{-1/12}N_3^{1/6}N_4^{1/2}
     \|u_1\|_{L_{t}^\I L_x^2}
     \|u_2\|_{X^{-1,1}}
     \|u_3\|_{L_{t}^\I L_x^2}\|u_4\|_{L_{t}^\I L_x^2}.
  }
  Next, we consider $I_{\I,3,3}$.
  Lemmas \ref{lem_comm1} and \ref{extensionlem}, the H\"older inequality, the Bernstein inequality and \eqref{eq4.3} show that
  \EQS{\label{eq_4.9}
    \begin{aligned}
      I_{\I,3,3}
      &\les N_3^{3/2} \|Q_{\ll L}
       (\1_{t,R}^{\text{low}} u_1)\|_{L_{t,x}^2}
       \|Q_{\ll L}(\1_{t,R}^{\text{low}} u_2)\|_{L_{t}^\I L_{t,x}^2}
       \|Q_{\gts L} u_3\|_{L_{t,x}^2}
       \|u_4\|_{L_{t,x}^\I}\\
      &\les T^{1/2}N_1^{-1}N_3^{3/2}N_4^{1/2}
        \|u_1\|_{L_t^\I L_x^2}
        \|u_2\|_{L_t^\I L_x^2}
        \|u_3\|_{X^{-1,1}}
        \|u_4\|_{L_t^\I L_x^2}.
    \end{aligned}
  }
  In the same manner, we can get
  \EQQS{
    I_{\I,3,4}
    \les T^{1/2}N_1^{-1}N_3^{1/2}N_4^{3/2}
      \|u_1\|_{L_t^\I L_x^2}
      \|u_2\|_{L_t^\I L_x^2}
      \|u_3\|_{L_t^\I L_x^2}
      \|u_4\|_{X^{-1,1}}.
  }
  For the estimate for $I_t(\bar{u}_1,\bar{u}_2,u_3,u_4)$ and $I_t(u_1,\bar{u}_2,\bar{u}_3,u_4)$, the same proof works.
  To be precise, on $I_t(\bar{u}_1,\bar{u}_2,u_3,u_4)$, the resonance relation is $|\xi_1^2+\xi_2^2-\xi_3^2-\xi_4^2|\sim N_1^2\gts N_1N_3$.
  Therefore, we can obtain stronger result than the right hand side of \eqref{eq_I1}, but it is enough for our purpose.
  This completes the proof.
\end{proof}

\begin{rem}\label{rem_Bourgain1}
  In Lemma \ref{lem_Bourgain1}, if we replace the hypothesis $N_1\sim N_2\gts N_3\gg N_4$ by $N_1\sim N_2\gts N_4\gg N_3$, the resultant is similar to \eqref{eq_I1}.
  We only have to replace $u_3,u_4,N_3,N_4$ by $u_4,u_3,N_4,N_3$, respectively.
\end{rem}

\begin{lem}\label{lem_Bourgain2}
  Let $0<T<1$ and $N_1,N_2,N_3,N_4\ge 1$.
  Assume that $u_k\in Z^0$ is a function with spatial Fourier support in $\{|\xi_k|+1\sim N_k\}$ for $k=1,2,3,4$.
  For any $t>0$, we set
  \EQS{\label{def_I2}
    I_t^{(2)}(u_1,u_2,u_3,u_4)
    :=\RE\int_0^t \int_\R
      u_1u_2u_3u_4dxdt'.
  }
  If $N_1\sim N_2\gts N_3\gg N_4$, then for $0<t<T$, it holds that
  \EQS{\label{eq_I2.1}
    \begin{aligned}
      &|I_t^{(2)}(u_1,\bar{u}_2,\bar{u}_3,u_4)|
       +|I_t^{(2)}(u_1,\bar{u}_2,u_3,\bar{u}_4)|\\
      &\les (N_3^{-1}N_4)^{1/2}
       \|u_3\|_{L_t^\I L_x^2}\|u_4\|_{L_t^\I L_x^2}
       \Bigl(T^{1/4}N_1^{-1/4} \|u_1\|_{L_t^\I L_x^2}\|u_2\|_{L_t^\I L_x^2}\\
      &\quad\quad+\|u_1\|_{X^{-1,1}}
       \|\1_t u_2\|_{L_{t,x}^2}
       +T^{1/4}N_1^{-1/12}N_3^{-1/3}
        \|u_1\|_{X^{-1,1}}
        \|u_2\|_{L_{t}^\I L_x^2}\\
      &\quad\quad+\|\1_t u_1\|_{L_{t,x}^2}
        \|u_2\|_{X^{-1,1}}
       +T^{1/4}N_1^{-1/12}N_3^{-1/3}
        \|u_1\|_{L_{t}^\I L_x^2}
        \|u_2\|_{X^{-1,1}}\Bigr)\\
      &\quad+T^{1/2}N_1^{-1}(N_3^{-1}N_4)^{1/2}
        \|u_1\|_{L_t^\I L_x^2}\|u_2\|_{L_t^\I L_x^2}
        \Bigl(N_3 \|u_3\|_{X^{-1,1}}\|u_4\|_{L_t^\I L_x^2}\\
      &\quad\quad+N_4\|u_3\|_{L_t^\I L_x^2}\|u_4\|_{X^{-1,1}}\Bigr),
    \end{aligned}
  }
  On the other hand, if $N_1\sim N_2\gg N_3\vee N_4$ or $N_1\sim N_2\gts N_3\gg N_4$, then
  it holds that for $0<t<T$
  \EQS{
    \begin{aligned}\label{eq_I2.2}
      &|I_t^{(2)}(\bar{u}_1,\bar{u}_2,u_3,u_4)|\\
      &\les N_1^{-1} (N_3N_4)^{1/2}
       \|u_3\|_{L_t^\I L_x^2}\|u_4\|_{L_t^\I L_x^2}
       (T^{1/4}N_1^{-1/4} \|u_1\|_{L_t^\I L_x^2}\|u_2\|_{L_t^\I L_x^2}\\
      &\quad\quad+\|u_1\|_{X^{-1,1}}
       \|\1_t u_2\|_{L_{t,x}^2}
       +T^{1/4}N_1^{-5/12}
        \|u_1\|_{X^{-1,1}}
        \|u_2\|_{L_{t}^\I L_x^2}\\
      &\quad\quad+\|\1_t u_1\|_{L_{t,x}^2}
        \|u_2\|_{X^{-1,1}}
       +T^{1/4}N_1^{-5/12}
        \|u_1\|_{L_{t}^\I L_x^2}
        \|u_2\|_{X^{-1,1}})\\
      &\quad+T^{1/2}N_1^{-2}(N_3N_4)^{1/2}
        \|u_1\|_{L_t^\I L_x^2}\|u_2\|_{L_t^\I L_x^2}
        (N_3\|u_3\|_{X^{-1,1}}\|u_4\|_{L_t^\I L_x^2}\\
      &\quad\quad+N_4\|u_3\|_{L_t^\I L_x^2}\|u_4\|_{X^{-1,1}}).
    \end{aligned}
  }
\end{lem}

\begin{proof}
  The proofs for estimates on $|I_t^{(2)}(u_1,\bar{u}_2,\bar{u}_3,u_4)|$ and $|I_t^{(2)}(u_1,\bar{u}_2,u_3,\bar{u}_4)|$ are essentially the same as that of Lemma \ref{lem_Bourgain1} except for not using Lemma \ref{lem_comm1}.
  On the other hand, for $|I_t^{(2)}(\bar{u}_1,\bar{u}_2,u_3,u_4)|$, a stronger resonance relation \eqref{eq_res2} is available in this configuration.
  The proof of Lemma \ref{lem_Bourgain1} works with replacing $R=N_1^{1/3}N_3^{4/3}$ and $L=N_1N_3$ by $R=N_1^{5/3}$ and $L=N_1^2$, respectively.
\end{proof}

\begin{rem}\label{rem_Bourgain2.1}
  The right hand side of \eqref{eq_I2.2} is bounded by that of \eqref{eq_I2.1}.
  This strong bound enables us to avoid using the modified energy in the proof of the estimate for the difference in the nonregular case (Section \ref{nonregular}).
\end{rem}

\begin{lem}\label{lem_Bourgain3}
  Let $0<T<1$ and $N,N_1,N_2,N_3,N_4\ge 1$ and let $\psi$ satisfy \eqref{hyp_psi}.
  Assume that $u_k\in Z_T^0$ is a function with spatial Fourier support in $\{|\xi_k|\sim N_k\}$ for $k=1,2,3$.
  Moreover, assume that $N_1\sim N_2\gts N_3\gg N_4$.
  Then it holds that for $0<t<T$
  \EQS{\label{eq_I3}
    \begin{aligned}
      &|I_t(\bar{u}_1,\bar{u}_2,u_3,P_{N_4}\psi)|
       +|I_t(u_1,\bar{u}_2,\bar{u}_3,P_{N_4}\psi)|
       +|I_t(u_1,\bar{u}_2,u_3,P_{N_4}\bar{\psi})|\\
      &\les N_3^{1/2}\|u_3\|_{L_t^\I L_x^2}
        \|P_{N_4}\psi\|_{L_{t,x}^\I}
        (T^{1/4}N_1^{-1/4}\|u_1\|_{L_t^\I L_x^2}
        \|u_2\|_{L_t^\I L_x^2}\\
      &\quad\quad+\|u_1\|_{X^{-1,1}}
       \|\1_t u_2\|_{L_{t,x}^2}
        +T^{1/4}N_1^{-1/12}N_3^{-1/3}\|u_1\|_{X^{-1,1}}\|u_2\|_{L_{t}^\I L_x^2}\\
      &\quad\quad+\|\1_t u_1\|_{L_{t,x}^2}\|u_2\|_{X^{-1,1}}
       +T^{1/4}N_1^{-1/12}N_3^{-1/3}\|u_1\|_{L_{t}^\I L_x^2}
       \|u_2\|_{X^{-1,1}})\\
      &\quad+T^{1/2}N_1^{-1}N_3^{1/2}
        \|u_1\|_{L_t^\I L_x^2}
        \|u_2\|_{L_t^\I L_x^2}
        (N_3 \|u_3\|_{X^{-1,1}}\|P_{N_4}\psi\|_{L_{t,x}^\I}\\
      &\quad\quad+\|u_3\|_{L_t^\I L_x^2}
        \|P_{N_4}\p_t\psi\|_{L_{t,x}^\I}).
    \end{aligned}
  }
\end{lem}

\begin{proof}
  First we consider $I_t(u_1,\bar{u}_2,u_3,P_{N_4}\bar{\psi})$.
  Putting $R=N_1^{1/3}N_3^{4/3}$, as in \eqref{eq_4.11}, we split $I_t$ as follows:
  \EQQS{
    |I_t(u_1,\bar{u}_2,u_3,P_{N_4}\bar{\psi})|
    &\le |I_t(\1_{t,R}^{\text{high}}u_1,\1_t \bar{u}_2,
       u_3,P_{N_4}\bar{\psi})|
     +|I_t(\1_{t,R}^{\text{low}}u_1,
      \1_{t,R}^{\text{high}}\bar{u}_2,
       u_3,P_{N_4}\bar{\psi})|\\
    &\quad+|I_t(\1_{t,R}^{\text{low}}u_1,
        \1_{t,R}^{\text{low}}\bar{u}_2,
         u_3,P_{N_4}\bar{\psi})|
      =:I_{\I,1}+I_{\I,2}+I_{\I,3}.
  }
  For $I_{\I,1}$ and $I_{\I,2}$, we can use the same argument (with Lemma \ref{lem_comm1}) as \eqref{eq_time} to get
  \EQQS{
    I_{\I,1}+I_{\I,2}
    \les T^{1/4}N_1^{-1/4}N_3^{1/2}\|u_1\|_{L_t^\I L_x^2}
    \|u_2\|_{L_t^\I L_x^2}
    \|u_3\|_{L_t^\I L_x^2}
    \|P_{N_4}\psi\|_{L_{t,x}^\I}
  }
  For the contribution of $I_{\I,3}$, putting $L=N_1N_3\gg R$, first we note that it holds that if we have
  \EQQS{
    |\ta_1+\xi_1^2-\ta_2-\xi_2^2+\ta_3+\xi_3^2-\ta_4-\xi_4^2|\gts L
  }
  under $|\ta_j+\xi_j^2|\ll L$ for $j=1,2,3$ and $\xi_4^2\ll L$,
  then it holds that $|\ta_4|\gts L$ by the impossible interactions.
  Therefore, we see from Lemmas \ref{lem_res1} and \ref{resonance}  that
  \EQQS{
    I_{\I,3}
    &\le |I_\I(Q_{\gts L}(\1_{t,R}^{\text{low}}u_1),
      \overline{\1_{t,R}^{\text{low}}u_2},
       u_3,P_{N_4}\bar{\psi})|\\
    &\quad+|I_\I(Q_{\ll L}(\1_{t,R}^{\text{low}}u_1),
      \overline{Q_{\gts L}(\1_{t,R}^{\text{low}}u_2}),
       u_3,P_{N_4}\bar{\psi})|\\
    &\quad+|I_\I(Q_{\ll L}(\1_{t,R}^{\text{low}}u_1),
      \overline{Q_{\ll L}(\1_{t,R}^{\text{low}}u_2)},
       Q_{\gts L}^+ u_3,P_{N_4}\bar{\psi})|\\
    &\quad+|I_\I(Q_{\ll L}(\1_{t,R}^{\text{low}}u_1),
       \overline{Q_{\ll L}(\1_{t,R}^{\text{low}}u_2)},
       Q_{\ll L} u_3,P_{N_4}R_{\gts L}\bar{\psi})|\\
    &=: I_{\I,3,1}+I_{\I,3,2}+I_{\I,3,3}+I_{\I,3,4}.
  }
  Then, following the proof of Lemma \ref{lem_Bourgain1}, we have
  \EQQS{
    I_{\I,3,1}
    &\les N_3^{1/2}\|u_1\|_{X^{-1,1}}
     (\|\1_t u_2\|_{L_{t,x}^2}
      +T^{1/4}R^{-1/4}\|u_2\|_{L_{t}^\I L_x^2})
     \|u_3\|_{L_{t}^\I L_x^2}\|P_{N_4}\psi\|_{L_{t,x}^\I},\\
    I_{\I,3,2}
    &\les N_3^{1/2}(\|\1_t u_1\|_{L_{t,x}^2}
     +T^{1/4}R^{-1/4}\|u_1\|_{L_{t}^\I L_x^2})
     \|u_2\|_{X^{-1,1}}
     \|u_3\|_{L_{t}^\I L_x^2}\|P_{N_4}\psi\|_{L_{t,x}^\I},\\
    I_{\I,3,3}
    &\les T^{1/2} N_1^{-1} N_3^{3/2}
     \|u_1\|_{L_{t}^\I L_x^2}
     \|u_2\|_{L_{t}^\I L_x^2}
     \|u_3\|_{X^{-1,1}}
     \|P_{N_4}\psi\|_{L_{t,x}^\I}.
  }
  For $I_{\I,3,4}$, the Young inequality in time shows that
  \EQQS{
    I_{\I,3,4}
    &\les T N_3^{3/2}
     \|u_1\|_{L_{t}^\I L_x^2}
     \|u_2\|_{L_{t}^\I L_x^2}
     \|u_3\|_{L_{t}^\I L_x^2}
     \|R_{\gts L}P_{N_4}\psi\|_{L_{t,x}^\I}\\
    &\les T N_1^{-1} N_3^{1/2}
     \|u_1\|_{L_{t}^\I L_x^2}
     \|u_2\|_{L_{t}^\I L_x^2}
     \|u_3\|_{L_{t}^\I L_x^2}
     \|P_{N_4}\p_t\psi\|_{L_{t,x}^\I}.
  }
  Collecting these estimates above, we obtain \eqref{eq_I3} for $I_t(u_1,\bar{u}_2,u_3,P_{N_4}\bar{\psi})$.
  Other two estimates follow similarly.
  This completes the proof.
\end{proof}

\begin{lem}\label{lem_Bourgain4}
  Let $0<T<1$ and $N_1,N_2,N_3,N_4\ge 1$ and let $\psi$ satisfy \eqref{hyp_psi}.
  Assume that $u_k\in Z_T^0$ is a function with spatial Fourier support in $\{|\xi_k|\sim N_k\}$ for $k=1,2,3$.
  If $N_1\sim N_2\gts N_3\gg N_4$, then it holds that for $0<t<T$
  \EQS{\label{eq_I4.1}
    &\begin{aligned}
      &|I_t^{(2)}(u_1,\bar{u}_2,u_3,P_{N_4}\bar{\psi})|
       +|I_t^{(2)}(u_1,\bar{u}_2,\bar{u}_3,P_{N_4}\psi)|\\
      &\les N_3^{-1/2}
       \|u_3\|_{L_t^\I L_x^2}\|P_{N_4}\psi\|_{L_{t,x}^\I}
       (T^{1/4}N_1^{-1/4} \|u_1\|_{L_t^\I L_x^2}\|u_2\|_{L_t^\I L_x^2}\\
      &\quad\quad+\|u_1\|_{X^{-1,1}}
       \|\1_t u_2\|_{L_{t,x}^2}
       +T^{1/4}N_1^{-1/12}N_3^{-1/3}
        \|u_1\|_{X^{-1,1}}
        \|u_2\|_{L_{t}^\I L_x^2}\\
      &\quad\quad+\|\1_t u_1\|_{L_{t,x}^2}
        \|u_2\|_{X^{-1,1}}
       +T^{1/4}N_1^{-1/12}N_3^{-1/3}
        \|u_1\|_{L_{t}^\I L_x^2}
        \|u_2\|_{X^{-1,1}})\\
      &\quad+T^{1/2}N_1^{-1}
        \|u_1\|_{L_t^\I L_x^2}\|u_2\|_{L_t^\I L_x^2}
        (N_3^{1/2}\|u_3\|_{X^{-1,1}}\|P_{N_4}\psi\|_{L_{t,x}^\I}\\
      &\quad\quad+\|u_3\|_{L_t^\I L_x^2}\|P_{N_4}\p_t\psi\|_{L_{t,x}^\I}).
    \end{aligned}
  }
  On the other hand, if $N_1\sim N_2\gg N_3\vee N_4$ or $N_1\sim N_2\gts N_3\gg N_4$, then it holds that for $0<t<T$
  \EQS{
    \begin{aligned}\label{eq_I4.2}
      &|I_t^{(2)}(\bar{u}_1,\bar{u}_2,u_3,P_{N_4}\psi)|\\
      &\les N_1^{-1} N_3^{1/2}
       \|u_3\|_{L_t^\I L_x^2}\|P_{N_4}\psi\|_{L_{t,x}^\I}
       (T^{1/4}N_1^{-1/4} \|u_1\|_{L_t^\I L_x^2}\|u_2\|_{L_t^\I L_x^2}\\
      &\quad\quad+\|u_1\|_{X^{-1,1}}
       \|\1_t u_2\|_{L_{t,x}^2}
       +T^{1/4}N_1^{-5/12}
        \|u_1\|_{X^{-1,1}}
        \|u_2\|_{L_{t}^\I L_x^2}\\
      &\quad\quad+\|\1_t u_1\|_{L_{t,x}^2}
        \|u_2\|_{X^{-1,1}}
       +T^{1/4}N_1^{-5/12}
        \|u_1\|_{L_{t}^\I L_x^2}
        \|u_2\|_{X^{-1,1}})\\
      &\quad+T^{1/2}N_1^{-2}N_3^{1/2}
        \|u_1\|_{L_t^\I L_x^2}\|u_2\|_{L_t^\I L_x^2}
        (N_3\|u_3\|_{X^{-1,1}}\|P_{N_4}\psi\|_{L_{t,x}^\I}\\
      &\quad\quad+\|u_3\|_{L_t^\I L_x^2}
        \|P_{N_4}\p_t\psi\|_{L_{t,x}^\I}).
    \end{aligned}
  }
\end{lem}

\begin{proof}
  The proof of Lemma \ref{lem_Bourgain3} works for estimates on $|I_t^{(2)}(u_1,\bar{u}_2,u_3,P_{N_4}\bar{\psi})|$ and $|I_t^{(2)}(u_1,\bar{u}_2,\bar{u}_3,P_{N_4}\psi)|$.
  On $|I_t^{(2)}(\bar{u}_1,\bar{u}_2,u_3,P_{N_4}\psi)|$, a stronger resonance relation \eqref{eq_res2} is available since $N_1\sim N_2\gts N_3\gg N_4$.
  The proof of Lemma \ref{lem_Bourgain3} works with replacing $R=N_1^{1/3}N_3^{4/3}$ and $L=N_1N_3$ by $R=N_1^{5/3}$ and $L=N_1^2$, respectively.
\end{proof}

\begin{rem}\label{rem_Bourgain4.1}
  The right hand side of \eqref{eq_I4.2} is controlled by that of \eqref{eq_I4.1}.
  This is thanks to the strong nonresonant relation.
\end{rem}

The following proposition is one of main estimates in the present paper.

\begin{prop}[A priori estimate]\label{prop_apri}
  Let $0<T<1$, $3/4\le s\le 2$.
  Let $\e>0$.
  Suppose that the dyadic sequence $\{\om_N\}$ of positive numbers satisfies $\om_N\le \om_{2N}\le 2^{\e/2}\om_N$ for $N\ge1$.
  Let $u_0\in H_\om^s(\R)$ and let $u\in L^\infty([0,T];H_\om^s(\R)) $ be a solution to \eqref{eq2} on $[0,T]$.
  Assume that there exists $K>0$ such that
  \EQQS{
    \|u\|_{Z_T^{3/4}}+\|\p_t \psi\|_{L_{t,x}^\I}
     +\|J_x^{s+1+\e}\psi\|_{L_{t,x}^\I}
      +\|\Psi\|_{L_t^\I H_x^{s+\e}}\le K,
  }
  where $\Psi$ is defined in \eqref{def_psi}.
  Then there exists $C=C(K)>0$ such that
  \EQS{\label{P}
    \|u\|_{L_T^\I H_\om^s}^2
    \le \|u_0\|_{H_\om^s}^2 + C T^{1/4}\|u\|_{Z_{\om,T}^{s}}
    \|u\|_{L_T^\I H_\om^s} +T\|\Psi\|_{L_t^\I H_x^{s+\e}}^2.
  }
\end{prop}

\begin{proof}
   First, we note that by Lemma \ref{lem47} it holds that
$ u\in Z_{\om,T}^{s}$.
  By using \eqref{eq2}, we have
  \EQQS{
    \frac{d}{dt}\|P_N u(t,\cdot)\|_{L_x^2}^2
    &=2\la \sum_{j=1}^{5} \RE\int_\R P_N f_j(u) P_N \bar{u}dx\\
    &\quad+2\mu \sum_{j=6}^{10} \RE\int_\R P_N f_j(u) P_N \bar{u}dx
     -2\IM \int_\R P_N \Psi P_N \bar{u}dx,
  }
 where $\Psi=i\p_t \psi+\p_x^2 \psi-i\la|\psi|^2\p_x \psi-i\mu\psi^2 \p_x \bar{\psi}$ and
 \EQS{\label{def_fu}
   \begin{aligned}
     &f_1(u)= |u|^2\p_x u,\
     f_2(u)=  |u|^2\p_x \psi,\
     f_3(u)=2 \p_x u \RE u\bar{\psi},\\
     &f_4(u)=2  \p_x \psi \RE u\bar{\psi},\
     f_5(u)=  |\psi|^2 \p_x u,\
     f_6(u)=  u^2\p_x \bar{u},\\
     &f_7(u)=  u^2\p_x \bar{\psi},\
     f_8(u)=2  u\psi \p_x\bar{u},\
     f_9(u)=2  u\psi \p_x\bar{\psi},\
     f_{10}(u)= \psi^2\p_x \bar{u}.
   \end{aligned}
 }
 Since $\la$ and $\mu$ do not play any role, we assume that $\la=\mu=1$ for the sake of simplicity.
 Fixing $ t\in [0,T] $,  integration in time between $0$ and $t$, multiplication by $ \om_N^2 N^{2s} $ and summation over $N$ yield
  \EQS{\label{PP}
   \begin{aligned}
     \|u(t)\|_{H_\om^s}^2
     &\le \|u_0\|_{H_\om^s}^2
       +\sum_{j=1}^{10}|J_t^{(j)}|\\
     &\quad +2\sum_N \om_N^2 N^{2s} \bigg|\IM \int_0^t \int_\R P_N \Psi P_N \bar{u}dx dt'\bigg|,
   \end{aligned}
  }
  where $\Psi$ is defined in \eqref{def_psi} and
  \EQQS{
    J_t^{(j)}:=2\sum_{N}  \om_N^2 N^{2s}
    \RE \int_0^t\int_\R f_j(u) P_N^2 \bar{u} dxdt'.
  }
  For the last term in the right hand side of \eqref{PP}, the Cauchy-Schwarz inequality shows that
  \EQQS{
    &\sum_N \om_N^2 N^{2s} \bigg|\IM \int_0^t \int_\R P_N \Psi P_N \bar{u}dx dt'\bigg|\\
    &\les \int_0^t \sum_{N} \om_N^2 N^{2s}
      \|P_N \Psi\|_{L_x^2} \|P_N \bar{u}\|_{L_x^2} dx dt'
    \les T\|\Psi\|_{L_T^\I H_x^{s+\e}}
      \|u\|_{L_T^\I H_\om^{s}}.
  }
  Here, $\om_N$ is absorbed by $\|\Psi\|_{L_T^\I H_x^{s+\e}}$.
  In what follows, we treat each $J_t^{(j)}$ for $j=1,\dots,10$ in  \eqref{PP}.
  We notice that \eqref{eq2.2.1} immediately implies that
  \EQQS{
    |J_t^{(2)}|
    \les T\|f_2(u)\|_{L_T^\I H_\om^s}\|u\|_{L_T^\I H_\om^s}
    \les TK^2 \|u\|_{L_T^\I H_\om^s}^2
  }
  since a derivative lands on $\psi$.
  Similarly, $J_t^{(4)},J_t^{(7)}$ and $J_t^{(9)}$ is bounded by the same way as above.
  Except for $J_t^{(1)}, J_t^{(3)}$ and $J_t^{(5)}$, we first simply employ the Littlewood-Paley decomposition to each $u$ and $\psi$, for example $u=\sum_{N_1}P_{N_1}u$.
  On the other hand, for $J_t^{(1)}, J_t^{(3)}$ and $J_t^{(5)}$, we make a complex conjugate copy of its integrand before being decomposed so as to use integration by parts (to be precise, Lemma \ref{lem_comm1}).
  For the sake of simplicity, we set
  \EQQS{
    K:=1+\|u\|_{Z_T^{3/4}}+\|\p_t \psi\|_{L_{t,x}^\I}
     +\|J_x^{s+1+\e}\psi\|_{L_{t,x}^\I}
      +\|\Psi\|_{L_T^\I H_x^{s+\e}}.
  }

  \noindent
  \textbf{Estimate for $J_t^{(1)}$.}
  As stated above, we use Lemma \ref{lem_comm1} in order to close the estimate. For that purpose, we decompose functions in the following way:
  \EQS{\label{eq_symm}
    \begin{aligned}
      2\RE\int_\R |u|^2 \p_x u P_N^2 \bar{u}dx
      &=\RE\int_\R (|u|^2 \p_x u P_N^2 \bar{u}
        + |u|^2 \p_x \bar{u} P_N^2 u)dx\\
      &=\RE\int_\R \Pi_N(u,\bar{u})|u|^2 dx\\
      &=\sum_{N_1,\dots,N_4}\RE\int_\R
        \Pi_N(P_{N_1}u,P_{N_2}\bar{u})
        P_{N_3}u P_{N_4}\bar{u} dx,
    \end{aligned}
  }
  where $\Pi_N(\cdot,\cdot)$ is defined by \eqref{def_pi}.
  By taking the complex conjugate, we may assume that $N_1\ge N_2$.
  We consider $J_t^{(1)}$ by the case-by-case analysis.
  In the following, we frequently use the following estimate: by Proposition \ref{prop_stri1} and
  \EQS{\label{eq_stri1}
    \begin{aligned}
      &\sum_N\om_N^2 N^{2s-1/2} \|P_N u\|_{L_T^2 L_x^\I}^2\\
      &\les \sum_N \om_N^2 N^{2s}
       (\|P_N u\|_{L_{T,x}^2}
        +N^{-1}\|P_N (F(u,\psi)-\Psi)\|_{L_{T,x}^2})^2\\
      &\les TC(K)\|u\|_{L_T^\I H_\om^s}^2
      +T\|\Psi\|_{L_T^\I H_x^{s+\e}}^2.
    \end{aligned}
  }

  \noindent
  \underline{Case 1: $N_1\sim N_2\gts N_3\vee N_4$.}
  First note that we have $N\sim N_1\sim N_2$.
  So it suffices to consider $J_t^{(1)}$ for $N\gg 1$, which we still denote by $J_t^{(1)}$.
  First we consider the case $N_3\sim N_4$.
  Without loss of generality, we may assume $N_3\ge 1$ since it is easy to check that $\|P_{N_3}u\|_{L_T^2 L_x^\I}\les \|P_{N_3}u\|_{L_{T,x}^2}$ when $N_3=0$.
  Lemma \ref{lem_comm1}, Proposition \ref{prop_stri1} with \eqref{eq_stri1} and \eqref{eq_2.3apl} imply that
  \EQQS{
    |J_t^{(1)}|
    &\les
    \sum_{N_1\sim N_2\gts N_3\sim N_4\ge 1}\om_N^2 N_1^{2s}N_3
    \int_0^t \|P_{N_1}u\|_{L_x^2}\|P_{N_2}u\|_{L_x^2}
      \|P_{N_3}u\|_{L_x^\I}\|P_{N_4}u\|_{L_x^\I}dt'\\
    &\les \|u\|_{L_T^\I H_\om^s}^2
     \sum_{N_3\ge 1}N_3\|P_{N_3}u\|_{L_T^2 L_x^\I}^2\\
    &\les \|u\|_{L_T^\I H_\om^s}^2
     \sum_{N_3\ge 1}(N_3^{3/2}\|P_{N_3}u\|_{L_{T,x}^2}^2
      +N^{-1/2}\|P_{N_3}(F(u,\psi)-\Psi)\|_{L_{T,x}^2}^2)\\
    &\les T C(K) \|u\|_{L_T^\I H_\om^s}^2.
  }
  On the other hand, when $N_3\gg N_4$, we use Bourgain type estimates (Lemma \ref{lem_Bourgain1}).
  For that purpose, we take the extensions $\check{u}=\rho_T(u)$ of $u$ defined in \eqref{def_ext}.
  Then, we notice that
  \EQQS{
    J_t^{(1)}
    =2\sum_{N}\sum_{N_1,\dots,N_4}\om_N^2 N^{2s}
     I_t(P_{N_1}\check{u},P_{N_2}\bar{\check{u}},
      P_{N_3}\check{u},P_{N_4}\bar{\check{u}}).
  }
  Lemma \ref{lem_Bourgain1}, \eqref{eq2.1} and the Young inequality show that
  \EQQS{
    |J_t^{(1)}|
    \les T^{1/4}K^2
     (\|\check{u}\|_{X_\om^{s-1,1}}+\|\check{u}\|_{L_t^\I H_\om^{s}})
     \|\check{u}\|_{L_t^\I H_\om^{s}}
    \le T^{1/4}C(K) \|u\|_{Z_{\om,T}^s}
     \|u\|_{L_T^\I H_\om^{s}}.
  }
  For the case $N_4\gg N_3$, according to Remark \ref{rem_Bourgain1}, we can estimate $J_t^{(1)}$ by the same bound as above.

  \noindent
  \underline{Case 2: $N_1\sim N_3\gts N_2\vee N_4$.}
  When $N_4\gts N_2$, the H\"older inequality, Proposition \ref{prop_stri1} with \eqref{eq_stri1}, the Young inequality, \eqref{eq_2.3apl} and \eqref{eq_2.3apl1} show that
  \EQQS{
    |J_t^{(1)}|
    &\les \sum_N\sum_{N_1\sim N_3\gts N_4\gts N_2} \int_0^t \om_N^2 N^{2s}
      \|P_{N_3}u\|_{L_x^2}\|P_{N_4}u\|_{L_x^\I}\\
    &\quad\times(N_1\|P_{N_1}u\|_{L_x^2}\|P_N^2 P_{N_2}u\|_{L_x^\I}
     +N_2\|P_N^2P_{N_1}u\|_{L_x^2}\|P_{N_2}u\|_{L_x^\I}) dt'\\
    &\les \|u\|_{L_T^\I H_x^{3/4}}^2 \sum_{N_4\gts N_2}
     \om_{N_2}^2 N_2^{2s} N_4^{-1/2}
     \|P_{N_2}u\|_{L_T^2 L_x^\I}
     \|P_{N_4}u\|_{L_T^2 L_x^\I}\\
    &\quad+\|u\|_{L_T^\I H_\om^s}^2
     \sum_{N_4\gts N_2}N_2\|P_{N_2}u\|_{L_T^2 L_x^\I}
     \|P_{N_4}u\|_{L_T^2 L_x^\I}\\
    &\les K^2 \sum_{N_4\gts N_2}
     \bigg(\frac{N_2}{N_4}\bigg)^{s+1/4}
     \om_{N_2}N_2^{s-1/4}\|P_{N_2}u\|_{L_T^2 L_x^\I}
     \om_{N_4}N_4^{s-1/4}\|P_{N_4}u\|_{L_T^2 L_x^\I}\\
    &\quad+\|u\|_{L_T^\I H_\om^s}^2
     \sum_{N_4\gts N_2}\bigg(\frac{N_2}{N_4}\bigg)^{1/2}
     N_2^{1/2}\|P_{N_2}u\|_{L_T^2 L_x^\I}
     N_4^{1/2}\|P_{N_4}u\|_{L_T^2 L_x^\I}\\
    &\les T C(K)\|u\|_{L_T^\I H_\om^s}^2
     +T\|\Psi\|_{L_T^\I H_x^{s+\e}}^2.
  }
  On the other hand, when $N_2\gg N_4$, as in Case 1, we see from \eqref{eq_I2.1} and \eqref{eq2.1} that
  \EQQS{
    |J_t^{(1)}|
    \les T^{1/4}C(K) \|u\|_{Z_{\om,T}^s}
     \|u\|_{L_T^\I H_\om^{s}}.
  }

  \noindent
  \underline{Case 3: $N_3\sim N_4\gts N_1\ge N_2$.}
  This case is easy to treat since derivatives are already distributed.
  The H\"older inequality, Proposition \ref{prop_stri1} and the Young inequality show that
  \EQQS{
    |J_t^{(1)}|
    &\les \sum_{N_3\sim N_4\gts N_1\ge N_2}
    \om_{N_2}^2 N_1^{2s}N_2
    \int_0^t\|P_{N_3}u\|_{L_x^2}\|P_{N_4}u\|_{L_x^2}
     \|P_{N_1}u\|_{L_x^\I}\|P_{N_2}u\|_{L_x^\I} dt'\\
    &\les \|u\|_{L_T^\I H_{\om}^s}^2
     \sum_{N_1\ge N_2}\bigg(\frac{N_2}{N_1}\bigg)^{1/2}
     \prod_{j=1}^2 (\|P_{N_j}u\|_{L_{T}^2 H_x^{3/4}}
      +\|P_{N_j}(F(u,\psi)-\Psi)\|_{L_{T}^2 H_x^{-1/4}})\\
    &\les TC(K)\|u\|_{L_T^\I H_{\om}^s}^2.
  }
  Here, we used $N_1^{2s}\le N_3^{2s}$ and $\om_{N_2}\les \om_{N_3}$ in the first inequality.
  This completes the estimate the contribution of $f_1(u)$.

  \noindent
  \textbf{Estimate for $J_t^{(3)}$.}
  In this case, we use Lemma \ref{lem_comm1}. As in \eqref{eq_symm}, we rewrite the integral on $\R$:
  \EQQS{
    J_t^{(3)}
    &=\la\sum_{N}\om_N^2 N^{2s}
     \RE\int_0^t\int_\R \Pi_N(u,\bar{u}) \RE(u\bar{\psi})dxdt'\\
    &=\la\sum_{N}\sum_{N_1,\dots, N_4}\om_N^2 N^{2s}
     \RE\int_0^t\int_\R \Pi_N(P_{N_1}u,P_{N_2} \bar{u}) \RE(P_{N_3} uP_{N_4}\bar{\psi})dxdt',
  }
  where $\Pi_N(\cdot,\cdot)$ is defined in \eqref{def_pi}.
  By symmetry, we may assume that $N_1\ge N_2$.

  \noindent
  \underline{Case 1: $N_4\gts N_3$.}
  In this case, we can distribute one derivative to $\psi$ after using Lemma \ref{lem_comm1}:
  \EQQS{
    \|\p_x (\RE(P_{N_3}uP_{N_4}\bar{\psi}))\|_{L_x^\I}
    \les N_4\|P_{N_3}u\|_{L_x^\I}\|P_{N_4}\psi\|_{L_x^\I}.
  }
  So, it is easy to see that $|J_t^{(3)}|\les T K^2 \|u\|_{L_T^\I H_\om^s}^2$.

  \noindent
  \underline{Case 2: $N_2\ge N_3\gg N_4$.}
  We see from Lemma \ref{lem_Bourgain3}, \eqref{eq2.1} and the Young inequality that $|J_t^{(3)}|
  \les T^{1/4}C(K) \|u\|_{Z_{\om,T}^s} \|u\|_{L_T^\I H_\om^{s}}$.

  \noindent
  \underline{Case 3: $N_3\ge N_2\gg N_4$.}
  Similarly to Case 2, Lemma \ref{lem_Bourgain4}, \eqref{estXregular} and the Young inequality show that
  $|J_t^{(3)}|
  \les T^{1/4}C(K)\|u\|_{Z_{\om,T}^s} \|u\|_{L_T^\I H_\om^{s}}$.

  \noindent
  \underline{Case 4: $N_3\ge N_4\gts N_2$.}
  In this case, we have $N_1\sim N_3$, and we can distribute derivatives on $\psi$.
  It is easy to see that
  \EQQS{
    |J_t^{(3)}|
    &\les \sum_{N_1\sim N_3\ge N_4\gts N_2}
     \om_{N_1}^2 N_1^{2s}N_2 \int_0^t
     \|P_{N_1}u\|_{L_x^2}\|P_{N_2}u\|_{L_x^\I}
     \|P_{N_3}u\|_{L_x^2}\|P_{N_4}\psi\|_{L_x^\I}dt'\\
    &\les T K^2 \|u\|_{L_T^\I H_\om^s}^2
  }
  since $N_1 N_2^{2s}\le N_1^{2s}N_2$ and $\om_{N_2}\les \om_{N_1}$.

  \noindent
  \textbf{Estimate for $J_t^{(5)}$.}
  As in \eqref{eq_symm}, we rewrite $J_t^{(5)}$:
  \EQQS{
    J_t^{(5)}
    =2\la \sum_{N}\sum_{N_1,\dots, N_4}
     \om_N^2 N^{2s}\RE
     \int_0^t \int_\R \Pi_N(P_{N_1}u, P_{N_2}\bar{u})
     P_{N_3}\psi P_{N_4} \bar{\psi} dxdt'.
  }
  By symmetry, we may assume that $N_1\ge N_2$. Without loss of generality, we also assume that $N_3\ge N_4$ since where the complex conjugate is has no rule in the following argument.

  \noindent
  \underline{Case 1: $N_1\sim N_2\gts N_3\ge N_4$.}
  In this case, we note that $N_1\sim N_2\sim N$. We only consider the case $N\gg 1$.
  By Lemma \ref{lem_comm1}, we have
  \EQQS{
    |J_t^{(5)}|
    &\les \sum_{N_1\sim N_2\gts N_3\ge N_4}\om_{N_1}^2 N_1^{2s} \|P_{N_1}u\|_{L_{T,x}^2}\|P_{N_2}u\|_{L_{T,x}^2}
    \|P_{N_3}\p_x \psi\|_{L_{T,x}^\I}
    \|P_{N_4}\psi\|_{L_{T,x}^\I}\\
    &\les T K^2 \|u\|_{L_T^\I H_\om^s}^2
  }
  since $J_x^{s+1+\e}\psi\in L^\I(\R^2)$.

  \noindent
  \underline{Case 2: $N_1\sim N_3\gts N_2\vee N_4$.}
  The H\"older inequality, we have
  \EQQS{
    |J_t^{(5)}|
    &\les \sum_{N_1\sim N_3\gts N_2\vee N_4}\om_{N_1}^2 N_1^{2s+1}\|P_{N_1}u\|_{L_{T,x}^2}\|P_{N_2}u\|_{L_{T,x}^2}
    \|P_{N_3}\psi\|_{L_{T,x}^\I}
    \|P_{N_4}\psi\|_{L_{T,x}^\I}\\
    &\les T K \|u\|_{L_T^\I H_x^{0+}} \sum_{N_1\sim N_3}
     \om_{N_1}^2 N_1^{2s+1}\|P_{N_1}u\|_{L_{T}^\I L_x^2}
     \|P_{N_3}\psi\|_{L_{T,x}^\I}
    \les TK^2 \|u\|_{L_T^\I H_\om^s}^2.
  }
  In the first inequality, we used
  $\sum_{N\les N_1}N^{2s}\les N_1^{2s}$.

  \noindent
  \underline{Case 3: $N_3\sim N_4\gts N_1\ge N_2$.}
  This case can be treated by the same way as Case 1 to get
  $|J_t^{(5)}|\les TK^2 \|u\|_{L_T^\I H_\om^s}^2$.

   \noindent
   \textbf{Estimate for $J_t^{(6)}$.}
   First we take the Littlewood-Paley decomposition of each function:
   \EQQS{
    J_t^{(6)}
    =\mu \sum_N \sum_{N_1,\dots,N_4}
     \om_N^2 N^{2s}\RE
     \int_0^t \int_\R
     \Pi_N(P_{N_1}\bar{u},P_{N_2}\bar{u})
     P_{N_3}u P_{N_4} u dxdt'.
   }
   By the symmetrization argument, we may assume that $N_1\ge N_2$ and $N_3\ge N_4$.

   \noindent
   \underline{Case 1: $N_1\sim N_2\gts N_3\ge N_4$.}
   This case can be treated as in Case 1 in the estimates for $J_t^{(1)}$.
   When $N_3\sim N_4$, we first use Lemma \ref{lem_comm1}, and then employ Proposition \ref{prop_stri1} to $P_{N_3}u$ and $P_{N_4}u$.
   On the other hand, when $N_3\gg N_4$, we use Lemma \ref{lem_Bourgain1} to get
   $|J_t^{(6)}|\les T^{1/4}C(K)\|u\|_{L_T^\I H_\om^s}^2+T^{1/4}\|\Psi\|_{L_T^\I H_x^s}^2$.

  \noindent
  \underline{Case 2: $N_1\sim N_3\gts N_2\vee N_4$.}
  This case can be treated by the same way as Case 2 in the estimates for $J_t^{(1)}$ with \eqref{eq_I2.1}.

  \noindent
  \underline{Case 3: $N_3\sim N_4\gts N_1\ge N_2$.}
  This case is identical to Case 3 in the estimates for $J_t^{(1)}$. We can obtain $|J_t^{(6)}|\les TC(K)\|u\|_{L_T^\I H_{\om}^s}^2$.

  \noindent
  \textbf{Estimate for $J_t^{(8)}$.}
  By decomposing each function, we have
  \EQQS{
    J_t^{(8)}
    =\mu \sum_N \sum_{N_1,\dots,N_4}\om_N^2 N^{2s}\RE
     \int_0^t \int_\R \Pi_N(P_{N_1}\bar{u},P_{N_2}\bar{u})P_{N_3}u
     P_{N_4}\psi dxdt'.
  }
  By the symmetrizations argument, we may assume that $N_1\ge N_2$.
  Then, we can use the same argument as the estimate for $J_t^{(3)}$ with Lemmas \ref{lem_Bourgain3} and \ref{lem_Bourgain4}.

  \noindent
  \textbf{Estimate for $J_t^{(10)}$.}
  This case is identical with the estimate for $J_t^{(5)}$.
  We can use the same argument if we replace $|\psi|^2, \p_x u$ by $\psi^2,\p_x \bar{u}$, respectively.

  Collecting above estimates, we obtain \eqref{P} by Lemma \ref{extensionlem}, which completes the proof.
\end{proof}

\section{Estimate for the Difference in the Regular Case}
\label{regular}

In this section, we deduce the estimate for the difference in the regular case.
Since the difference equation has fewer symmetries than the original equation, we have to derive such estimates at a lower regularity than those of the solution.
We refer to the regular case as the situation where we can estimate the difference of solutions at least one regularity level below that of the solutions.
This will hold in two different settings: when $ s\ge 1 $ without any assumption on $ (\lambda,\mu) $, and when $ s>3/4 $ under the assumption $ \lambda=2\mu $ (the divergence form case).
In this section, we prove a Lipschitz bound for the difference of solutions in $ H^{s-1}$ for $ s\in (3/4,1]  $ under the assumption $ \lambda=2\mu $, and in $ L^2$ for $ s\ge 1 $, without any restriction on $ (\lambda,\mu) $.
\begin{lem}\label{LemZ1}
  Let $0<T<1$.
  Let $u_1$ and $u_2$ be two solutions of \eqref{eq2} on $[0,T]$ belonging to $ L^\infty_T H_x^s $ with $ s>0 $.
  Assume that there exists $K_s>0$ such that
  \EQQS{
    \|u_1\|_{L_T^\I H_x^s}+\|u_2\|_{L_T^\I H_x^s}+\|J_x\psi\|_{L_{T,x}^\I}
    \le K_s.
  }
  If $ s\in (3/4,1]$ and $ \lambda=2\mu $, then $ w=u_1-u_2 $ satisfies
  \begin{equation}\label{estZ1}
    \|w\|_{Z^{s-1}_T}
    \le C(K_s) \|w\|_{L^\infty_T H_x^{s-1}}
  \end{equation}
  and for any $ N\ge 1 $,
  \EQS{\label{estwLinfi1}
    \|P_N w\|_{L^4_T L^\infty_x}
    \le C(K_s) N^{\frac{5}{4}-s}
    \|w\|_{L_T^\I H_x^{s-1}},
  }
  whereas if $ s\ge 1 $, then for any $(\lambda,\mu) $, $w$ satisfies
  \begin{equation}\label{estZ2}
    \|w\|_{Z^{0}_T}
    \le C(K_1) \|w\|_{L^\infty_T L_x^2}
  \end{equation}
  and for any $ N\ge 1 $,
  \EQS{\label{estwLinfi2}
    \|P_N w\|_{L^4_T L^\infty_x}
    \le C(K_1) N^{\frac{1}{4}}
      \|w\|_{L_T^\I L^2_x}.
  }
\end{lem}

\begin{proof}
We notice that \eqref{eq1} may be rewritten as
\EQQS{
  i \p_t v+\p_x^2 v
  =\frac{i\la}{2}\partial_x (|v|^2 v)
   + \frac{i(2\mu-\la)}{2} v^2\p_x \ov{v},
}
and thus, setting $ z=u_1+u_2$,  $w=u_1-u_2$ satisfies the following equation:
\EQS{\label{eq_w}
  \begin{aligned}
    \p_t w
    &=i\p_x^2 w
     +\frac{\la}{2} \p_x \Bigl( |u_1|^2 w+u_2\RE(z\bar{w})
     + \psi \RE(z\bar{w})
     +2u_1 \RE(\psi \bar{w})\\
    &\quad\quad\quad\quad\quad\quad\quad
     +2w \RE(\psi \bar{u}_2)
     +2 \psi\RE(\psi \bar{w})
     +|\psi|^2w\Bigr)\\
    &\quad+\frac{2\mu-\la}{2}\Bigl(u_1^2\p_x \bar{w}
     +zw\p_x \bar{u}_2
     +wz\p_x \bar{\psi}
     +2w\psi \p_x \bar{u}_2\\
    &\quad\quad\quad\quad\quad\quad\quad
     +2u_1\psi\p_x \bar{w}
     +2w\psi\p_x \bar{\psi}
     +\psi^2 \p_x \bar{w}\Bigr)\\
    &=:i\p_x^2 w
     +\frac{\la}{2}\p_x \Bigl(  \sum_{j=1}^7 \tilde{g}_j(u_1,u_2)\Bigr)
     +\frac{2\mu-\la}{2}\sum_{j=8}^{14} g_j(u_1,u_2) \; .
    \end{aligned}
  }
By Lemmas \ref{Lem24} and \ref{Lem25}, we can check that for $ 1/2<s\le 1  $,
\EQS{\label{tr1}
  \bigg\| \sum_{j=1}^7 \tilde{g}_j(u_1,u_2)\bigg\|_{H_x^{s-1}}
  \le C(K_s)\|w\|_{H_x^{s-1}}.
}
Here, we applied \eqref{eq_2.3apl4} to $\tilde{g}_7$ with $s_1=1/2$ and $s_2=s_3=s-1$.
This proves \eqref{estZ1} by standard linear estimates in Bourgain spaces (see, e.g., \cite{GTV97}).
To prove \eqref{estZ2}, it now remains to evaluate the contributions of $ g_j$ for $ j\in \{8,\dots,14\}$.
For that purpose, we first notice that
$\|g_{13}(u_1,u_2)\|_{H^{-1}}\les \|g_{13}(u_1,u_2)\|_{L^2}\les  \|J_x\psi\|_{L^{\infty}}^2 \|w\|_{L^2} $ and
\EQQS{
  \bigg\|\sum_{j=9}^{11} g_j(u_1,u_2)\bigg\|_{H_x^{-1}}
  \les \bigg\|\sum_{j=9}^{11} g_j(u_1,u_2)\bigg\|_{L_x^1}
  \le C(K_1)\|w\|_{L_x^2}.
}
On the other hand, rewriting $g_8(u_1,u_2)$ as $ g_8(u_1,u_2)= \partial_x(u_1^2  \bar{w})-\partial_x(u_1^2) \bar{w} $ we get
\EQQS{
  \|g_8(u_1,u_2)\|_{H_x^{-1}}
  \les \|u_1^2 \bar{w}\|_{L_x^2}
   +\|\partial_x(u_1^2) \bar{w} \|_{L_x^1}
  \les \|u_1\|_{H_x^1}^2 \|w\|_{L_x^2}.
}
Treating $ g_{12}(u_1,u_2) $ and $ g_{14}(u_1,u_2) $ in the same way we eventually get
\EQS{\label{tr2}
  \bigg\|\sum_{j=8}^{14} g_j(u_1,u_2)\bigg\|_{H_x^{-1}}
  \le C(K_1)\|w\|_{L_x^2},
}
which, together with \eqref{tr1}, leads to \eqref{estZ2} by standard linear estimates in Bourgain spaces.

Next, to prove \eqref{estwLinfi1}, we rely on Proposition \ref{prop_stri1}.
By taking $\tilde{q}=q=4$ and $p=\I$, we infer that
\EQQS{
  \|P_N w\|_{L^4_T L^\infty_x}
  &\les N^{\frac{1}{4}}\|P_N w \|_{L^\infty_T L^2_x}
    + N^{-\frac{3}{4}}
    \bigg\| \partial_x P_N \sum_{j=1}^6 \tilde{g}_j(u_1,u_2)\bigg\|_{L^\infty_T L^2_x} \\
  &\les N^{\frac{5}{4}-s} \bigg( \|P_N w \|_{L^\infty_T H^{s-1}}
   +\bigg\|\sum_{j=1}^6 \tilde{g}_j(u_1,u_2)\bigg\|_{L^\infty_T H_x^{s-1}}\bigg)
}
that leads to \eqref{estwLinfi1} thanks to \eqref{tr1}.
Finally, \eqref{estwLinfi2} follows in the same way by making use of \eqref{tr2}.
This completes the proof.
\end{proof}

\begin{prop}\label{prop_dif1}
  Let $0<T<1$ and $ s>0$.
  Let $u_1$ and $u_2$ be two solutions of \eqref{eq2} on $[0,T]$ belonging to $Z_T^s$ and emanating from $u_{0,1}\in H^s(\R)$ and $u_{0,2}\in H^s(\R)$, respectively.
  Let $K_s\ge 1$ such that
  \EQQS{
    \|u_1\|_{Z_T^{s\wedge 1}}+\|u_2\|_{Z_T^{s\wedge 1}}+\|\p_t \psi\|_{L_{t,x}^\I}+\|J_x^{1+\e}\psi\|_{L_{t,x}^\I}
      + \|\Psi\|_{L_t^\I H^{\e}_x} \le K_s,
  }
  where $\Psi$ is defined in \eqref{def_psi}.
  If $ s\in (3/4,1]$ and $ \lambda=2\mu $, then $w=u_1-u_2$ satisfies
   \EQS{\label{eq_dif1}
    \|w\|_{L_T^\I  H_x^{s-1}}^2
    \le \|u_{0,1}-u_{0,2}\|_{H_x^{s-1}}^2 \
    + C(K_s) T^{1/4}\|w\|_{Z_{T}^{s-1}}
    \|w\|_{L_T^\I H_x^{s-1}},
  }
  whereas if $ s\ge 1 $, then for any $(\lambda,\mu)\in \R^2$,  $w$ satisfies
  \EQS{\label{eq_dif2}
    \|w\|_{L_T^\I L_x^2}^2
    \le \|u_{0,1}-u_{0,2}\|_{L_x^2}^2 + C(K_1) T^{1/4}\|w\|_{Z_{T}^{0}}
    \|w\|_{L_T^\I L_x^2}.
  }
\end{prop}

\begin{proof}
We will make use of the equation \eqref{eq_w} satisfied by $ w$.
As in the preceding lemma, we first prove \eqref{eq_dif1} that enables at the same time to estimate the contribution of $ \tilde{g}_j $, $j\in \{1,\dots,7\} $ in  \eqref{eq_dif2}.
By using \eqref{eq_w}, when $\lambda=2\mu $, we have
\EQQS{
  \frac{d}{dt}\|P_N w(t,\cdot)\|_{L_x^2}^2
  =\la\sum_{j=1}^{7} \RE\int_\R \p_x P_N \tilde{g}_j(u_1,u_2) P_N \bar{w}dx.
}
In the sequel, we will frequently use the fact that Bernstein inequalities together with \eqref{eq_2.3apl1} and Proposition \ref{prop_stri1} with $\tilde{q}=q=4 $ and $ p=\infty $ leads for any $ N\ge 1 $ to
\EQS{\label{estL1}
  \|P_N u_j\|_{L^4_T L^\infty_x}
  \le C(K_s) N^{\frac{1}{4}-s},
}
where $j=1,2$.
Recall that we fix $ s\in (3/4,1] $.
Fixing $ t\in [0,T] $, multiplying by $N^{2(s-1)}$, integrating in time between $0$ and $t$ and summing over $N$ yield
\EQS{\label{PP2}
  \begin{aligned}
    \|w(t)\|_{H_x^{s-1}}^2
    &\le  \|u_{0,1}-u_{0,2}\|_{H_x^{s-1}}^2
    +|\la|\sum_{j=1}^{7}|J_t^{(j)}|,
  \end{aligned}
  }
where
\EQQS{
  J_t^{(j)}
  := \sum_{N\ge 1} N^{2(s-1)}
    \RE \int_0^t\int_\R \tilde{g}_j(u_1,u_2) P_N^2 \p_x \bar{w} dxdt'.
}
First thanks to \eqref{tr1} we see that the contribution of the sum over $ 1\le N\les 1 $ of $\sum_{j=1}^{7}|J_t^{(j)}| $ can be easily estimated by
\EQS{\label{estlowN}
  \sum_{1\le N\les 1} N  \bigg\|\sum_{j=1}^7 \tilde{g}_j(u_1,u_2) \bigg\|_{H_x^{s-1}} \|w\|_{H_x^{s-1}}
  \le C(K_s) \|w\|_{H_x^{s-1}}^2.
}
Now, to control the contribution of the sum over $ N\gg 1$, as in the proof of Proposition \ref{prop_apri}, we proceed by treating each $J_t^{(j)}$ for $j=1,\dots,7$ in \eqref{PP}.
We employ the non-homogeneous Littlewood-Paley decomposition to each $u_1$, $u_2$ and $\psi$, for example $u_1=\sum_{N_1\ge 1}P_{N_1}u_1$.
In the sequel, $ (\delta_{2^j})_{j\in\N} $ denotes any sequence in $ \ell^2(\N) $ with $\|(\delta_{2^j})\|_{\ell^2(\N)}\le 1$.

\noindent
\textbf{Estimate for $J_t^{(1)}$.}
We eventually get
\EQQS{
  J_t^{(1)}
  \les\sum_{N\gg 1} N^{2(s-1)}\Bigl|\sum_{\substack{N_1,N_2,N_3\ge 1,\\N_4\sim N}}
   \RE\int_0^t \int_{\R} \RE(P_{N_1} u_1P_{N_2} \bar{u}_1) P_{N_3}w  \partial_x P_N^2 P_{N_4} \bar{w}dxdt' \Bigr|.
}
\underline{Case 1: $N_1\vee N_2\gts N$.}
By symmetry we may assume $N_1 \ge N_2$.
Then we must have $ N_1 \gts N_3 $.
We notice that the contribution of the sum over $1\le N_3\les 1$ can be easily estimated by making use of discrete Young's inequality and Sobolev embeddings
\EQS{
  \begin{aligned}\label{tou1}
    |J_t^{(1)}|
    &\les \sum_{N_1\gts N\gg 1} N^{2s-1}
    \int_0^t \|P_{N_1}u_1\|_{L_x^2}
     \|P_{\les N_1}u_1\|_{L_x^\I}
     \|P_{\les 1}w\|_{L_x^\I}
     \|P_{N}w\|_{L_x^2}dt'\\
    &\les\sum_{N_1\gts N\gg 1}
     \Bigl(\frac{N}{N_1}\Bigr)^{s}\delta_N \delta_{N_1}
     \int_0^t \|u_1\|_{H_x^s}
     \|u_1\|_{H_x^{\frac{1}{2}+}}\|w\|_{H_x^{s-1}}
     \|w\|_{H_x^{s-1}}dt' \\
    &\les T \|u_1\|_{L^\I_T H_x^s}^2 \|w\|_{L^\I_T H_x^{s-1}}^2.
  \end{aligned}
}

\noindent
\underline{Subcase 1.1: $N_2\gts N_3\gg 1$.}
Then we make use again of discrete Young's inequality but this time together with \eqref{estwLinfi1} and \eqref{estL1} to get for $ 3/4<s\le 1 $,
\EQQS{
  |J_t^{(1)}|
  &\les \sum_{\substack{N_1\gts N\gg 1,\\1\ll N_3 \les N_2\le N_1}}
   \Bigl(\frac{N}{N_1}\Bigr)^{s}\delta_N\delta_{N_1}
   \int_0^t \|u_1\|_{H_x^s}
   \|P_{N_2}u_1\|_{L_x^\infty}\|P_{N_3}w\|_{L_x^\infty}
   \|w\|_{H_x^{s-1}}dt'\\
  &\les T^{1/2} K_s \|w\|_{L^\infty_T H_x^{s-1}}
   \sum_{\substack{N_1\gts N\gg 1,\\1\ll N_3\les N_2\le N_1}}
   \Bigl(\frac{N}{N_1}\Bigr)^{s} \delta_N \delta_{N_1}
   \| P_{N_2}u_1\|_{L^4_T L_x^\I}
   \| P_{N_3}w\|_{L^4_T L^\infty_x}\\
  &\le C(K_s) T^{1/2} \|w\|_{L^\infty_T H_x^{s-1}}^2
   \sum_{1\ll N_3\les N_2}
   N_2^{\frac{1}{4}-s} N_3^{\frac{5}{4}-s}
  \le C(K_s) T^{1/2} \|w\|_{L_T^\I H_x^{s-1}}^2.
  }

\noindent
\underline{Subcase 1.2: $1\le N_2\ll N_3$.}
When $N\gts N_3$ or $N_3\gg N\gg N_2$, Lemma \ref{lem_res1} ensures that
$ |\Om |\sim  N_1 (N\wedge N_3) $, and \eqref{eq_I2.1} of Lemma \ref{lem_Bourgain2} leads to
\EQQS{
  |J_t^{(1)}|
  \les T^{1/4}\|u_1\|_{Z^s_T}^2 \|w\|_{Z^{s-1}_T}
  \|w\|_{L^\I_T H_x^{s-1}}.
}
On the other hand, when $N\les N_2$, as in Subcase 1.1 we get
\EQQS{
  |J_t^{(1)}|
  \les T^{1/2} K_s^2  \|w\|_{L^\infty_T H_x^{s-1}}^2.
}

\noindent
\underline{Case 2: $N_1\vee N_2\ll N$.}
Then $N_4\sim N_3 \sim N $.
Similarly to \eqref{eq_symm}, we notice that this contribution may be rewritten as
\EQQS{
  J_t^{(1)}
  :=\sum_{N\gg 1} N^{2(s-1)}
   \bigg|\sum_{\substack{1\le N_1,N_2\ll N,\\N_3\sim N_4\sim N}}
   \RE\int_0^t \int_{\R} \RE(P_{N_1} u_1 P_{N_2} \bar{u}_1)
   \ti{\Pi}_N(P_{N_3}w, P_{N_4}\bar{w})dxdt'\bigg|
}
with $\ti{\Pi}_N(\cdot,\cdot)$ defined in \eqref{def_pi2}.

\noindent
\underline{Subcase 2.1: $N_1\sim N_2 \ge  1$.}
Then by Lemma \ref{lem_comm1} and \eqref{estL1}, we get
\EQQS{
  |J_t^{(1)}|
  &\les T^{1/2} \|w\|_{L^\infty_T H_x^{s-1}}^2
   \sum_{ N_1\sim N_2\ge 1}
   N_1 \|P_{N_1}u_1\|_{L^4_T L^\infty_x}
   \|P_{N_2}u_1\|_{L^4_T L^\infty_x}\\
  &\les T^{1/2} K^2 \|w\|_{L^\infty_T H_x^{s-1}}^2
   \sum_{N_1\ge 1} N_1^{\frac{3}{2}-2s}
  \les T^{1/2} K_s^2 \|w\|_{L^\infty_T H_x^{s-1}}^2,
}
which is acceptable for $ s>3/4$.

\noindent
\underline{Subcase 2.2: $N_1\gg N_2 $ or $N_2\gg N_1$.}
Then Lemma \ref{lem_res1} ensures that
$ |\Om |\sim  (N_1\vee N_2)N  $.
So, we see from Lemma \ref{lem_comm1} and \eqref{eq_I1} of Lemma \ref{lem_Bourgain1} that
\EQQS{
  |J_t^{(1)}|
  \les T^{1/4} \|u_1\|_{Z^s_T}^2
   \|w\|_{Z^{s-1}_T} \|w\|_{L^\infty_T H_x^{s-1}}.
}
This completes the estimate of $J_t^{(1)}$.

\noindent
\textbf{Estimate for $J_t^{(2)}$.}
We eventually get
\EQQS{
  J_t^{(2)}
  \les \sum_{N\gg 1} N^{2(s-1)}
   \Bigl|\sum_{\substack{N_1,N_2,N_3\ge 1,\\ N_4\sim N}}
   \RE\int_0^t \int_{\R} \RE(P_{N_1} z P_{N_2} \bar{w}) P_{N_3}u_2
   \p_x P_N^2 P_{N_4} \bar{w}dxdt' \Bigr|.
}

\noindent
\underline{Case 1: $N_1\vee N_3\gts N$.}
Notice that $N_1\gts N_2$.
We can assume $ N_1\ge N_3 $.
As in \eqref{tou1}, the contribution of the sum over $ 1<N_2\les 1 $ can be easily estimated as
\EQQS{
  J_t^{(2)}
  &\les \sum_{N\gg 1} N^{2(s-1)}
   \Bigl|\sum_{\substack{N_1\ge N_3\ge 1,\\N_4\sim N}}
   \RE\int_0^t \int_{\R} \RE(P_{N_1} z P_{\les 1} \bar{w})
   P_{N_3}u_2 \p_x P_N^2 P_{N_4} \bar{w} dxdt' \Bigr|\\
  &\les T \|u_2\|_{L^\infty_T H_x^s}\|z\|_{L^\infty_T H_x^s} \|w\|_{L^\infty_T H_x^{s-1}}^2.
}

\noindent
\underline{Subcase 1.1: $N_3\gts N_2\gg 1$.}
Then we can proceed exactly as in the Subcase 1.1 in the estimates for
$J_t^{(1)}$ to get for $ 3/4<s\le 1 $,
\EQQS{
  |J_t^{(2)}|
  &\les \sum_{\substack{N_1\gts N\gg 1,\\1\ll N_2 \les N_3\le N_1}} N^{2s-1}
   \int_0^t \|P_{N_1}z\|_{L_x^2}
   \|P_{N_2}w\|_{L_x^\infty}\|P_{N_3}u_2\|_{L_x^\infty}
   \|P_{N}w\|_{L_x^2}dt'\\
  &\les T^{1/2} K_s^2
   \|w\|_{L^\infty_T H_x^{s-1}}^2.
}

\noindent
\underline{Subcase 1.2: $N_3\ll N_2 $.}
We can argue in an almost parallel way to Subcase 1.2 in the estimates for $J_t^{(1)}$.
When $N\gts N_2$ or $N_2\gg N\gg N_3$, Lemma \ref{lem_res1} gives $ |\Om |\sim N_1 (N\wedge N_2)$.
On the other hand, when $N_3\gts N$, we apply \eqref{estwLinfi1} to $P_{N_3}w$ and $P_N w$.
In both cases, we obtain
\EQQS{
  |J_t^{(2)}|
  \le C(K_s) T^{1/4}\|w\|_{L^\infty_T H_x^{s-1}}^2.
}

\noindent
\underline{Case 2: $N_1\vee N_3\ll N $.}
We have $ N_2\sim N_4\sim N $.
By the same argument as in Case 2 of the estimates for $J_t^{(1)}$, we obtain
\EQQS{
  |J_t^{(2)}|
  \les T^{1/4} \|u_1\|_{Z^s_T} \|z\|_{Z^s_T}
   \|w\|_{Z^{s-1}_T} \|w\|_{L^\infty_T H_x^{s-1}},
}
which completes the estimates for $J_t^{(2)}$.

\noindent
\textbf{Estimate for $J_t^{(3)}$.}
We eventually get
\EQQS{
  J_t^{(3)}
  \les \sum_{N\gg 1} N^{2(s-1)}
   \Bigl|\sum_{\substack{N_1,N_2,N_4>1,\\N_3\sim N}}
   \RE\int_0^t \int_{\R} \RE(P_{N_1} z P_{N_2} \bar{w})
   P_{N_4}\psi \p_x P_N^2 P_{N_3}
   \bar{w}dxdt' \Bigr|.
}

\noindent
\underline{Case 1: $N_1\vee N_4\gts N$.}
The contribution of the sum over $ 1\le N_2\les 1 $ can be easily estimated by
\EQQS{
  |J_t^{(3)}|
  \les T \|z\|_{L^\infty_T H_x^s}
   \|J_x^{1+\e}\psi\|_{L_{T,x}^\I}
   \|w\|_{L^\infty_T H_x^{s-1}}^2.
}

\noindent
\underline{Subcase 1.1: $N_1\ge N_4$}.
If $ N_4\gts N_2\wedge N\gg 1 $, then we get
\EQQS{
  |J_t^{(3)}|
  \les T \|z\|_{L^\infty_T H_x^s}
   \|J_x\psi\|_{L_{T,x}^\I}
   \|w\|_{L^\infty_T H_x^{s-1}}^2
}
whereas if $ N_4\ll N_2 \wedge N $ then \eqref{eq_res1} of Lemma   \ref{lem_res1}  leads to
$ |\Om |\sim N_1(N_2 \wedge N)$.
Applying \eqref{eq_I2.1} of Lemma \ref{lem_Bourgain2}, we infer that
\EQQS{
  |J_t^{(3)}|
  \les T^{1/4} \|z\|_{Z^s_T}
   (\|\p_t \psi\|_{L^\infty_{t,x}} +\|J_x\psi\|_{L_{t,x}^\I})
   \|w\|_{Z^{s-1}_T}\|w\|_{L_T^\I H_x^{s-1}}.
}

\noindent
\underline{Subcase 1.2: $N_4\ge N_1$}.
Then, we easily get
\EQQS{
  |J_t^{(3)}|
  \les T \|z\|_{L^\I_T L^\infty_x }
   \|J_x^{1+\e}\psi\|_{L_{T,x}^\I}
   \|w\|_{L^\infty_T H_x^{s-1}}^2.
}

\noindent
\underline{Case 2: $N_1\vee N_4\ll N$.}
We have $ N_3\sim N_2\sim N $ in this case.
By almost the same argument as in Case 1 in the estimate of $J_t^{(1)}$, we obtain
\EQQS{
  |J_t^{(3)}|
  \les T^{1/4} \|z\|_{Z^s_T}
   (\|\p_t \psi\|_{L^\infty_{t,x}}+\|J_x\psi\|_{L_{t,x}^\I})
   \|w\|_{Z^{s-1}_T}
   \|w\|_{L^\infty_T H_x^{s-1}}.
}
Here, we used \eqref{eq_I3} instead of \eqref{eq_I1}.
This completes the estimates for $ J_t^{(3)}$.

\noindent
\textbf{Estimates for $J_t^{(4)}$ and$J_t^{(5)}$.}
These terms can be estimated in exactly the same way as $J_t^{(3)}$.

\noindent
\textbf{Estimate for  $J_t^{(6)}$.}
As in \eqref{eq_symm}, we rewrite $J_t^{(6)}$ as follows:
\EQQS{
  J_t^{(6)}
  &= \frac{1}{4}\sum_{N\gg 1} N^{2(s-1)}
   \Bigl|\sum_{N_1,N_2,N_3\ge1}
   \RE\int_0^t \int_{\R}  P_{N_1} (\psi^2)
   \ti{\Pi}_N(P_{N_2} \bar{w}, P_{N_3} \bar{w}) dxdt' \Bigr|\\
  &\quad+\frac{1}{4}\sum_{N\gg 1} N^{2(s-1)}
   \Bigl|\sum_{N_1,N_2,N_3\ge1}
   \RE\int_0^t \int_{\R}
   P_{N_1}(|\psi|^2)
   \ti{\Pi}_N(P_{N_2} w, P_{N_3} \bar{w}) dxdt' \Bigr|,
}
where $\ti{\Pi}_N(\cdot,\cdot)$ is defined in \eqref{def_pi2}.
By symmetry, we may assume that $N_2\ge N_3$.
We only consider the first term (which we still denote by $J_t^{(6)}$), since the second term can be handled in exactly the same way.
We can handle both terms simultaneously since we do not have to use Bourgain type estimates.

\noindent
\underline{Case 1: $ N_2\gg N_3$.}
We have $N_1\sim N_2\sim N$.
By the impossible interaction, the integrand of $J_t^{(6)}$ becomes $P_{N_1}(\psi^2)\p_x P_N^2 P_{N_2}\bar{w} P_{N_3}\bar{w}$, ensuring that
\EQQS{
  |J_t^{(6)}|
  &\les \sum_{N\sim N_1\sim N_2\gg N_3\ge1}N^{2s-1}
   \|P_{N_1}(\psi^2)\|_{L_{T,x}^\I}
   \|P_N^2 P_{N_2}w\|_{L_{T,x}^2}
   \|P_{N_3}w\|_{L_{T,x}^2}\\
  &\les T \|J_x^{1+\e}\psi\|_{L_{T,x}^\I}^2
   \|w\|_{L^\infty_T H_x^{s-1}}^2
  \les T\|J_x^{1+\e}\psi\|_{L_{T,x}^\I}^2
   \|w\|_{L^\infty_T H_x^{s-1}}^2.
}

\noindent
\underline{Case 2: $ N_2\sim N_3$.}
Note that $N\sim N_2\sim N_3\gts N_1$.
By Lemma \ref{lem_comm1}, we have
\EQQS{
  |J_t^{(6)}|
  &\les \sum_{N\sim N_2\sim N_3\gts N_1\ge1}N^{2(s-1)}
   \|\p_x P_{N_1}(\psi^2)\|_{L_{T,x}^\I}
   \|P_{N_2}w\|_{L_{T,x}^2}
   \|P_{N_3}w\|_{L_{T,x}^2}\\
  &\les T\|J_x^{1+\e}\psi\|_{L_{T,x}^\I}
   \|w\|_{L^\infty_T H_x^{s-1}}^2.
}
Here, we used \eqref{eq2.7} in the last inequality.
This completes the estimates for $J_t^{(6)}$.

\noindent
\textbf{Estimate for $J_t^{(7)}$.}
This term is identified with the second term of $J_t^{(6)}$.
So, this can be estimated in exactly the same way as above.
Gathering the above estimates for $ J_t^{(j)} $ for $ j=1,\dots,7$, and \eqref{PP2}, we obtain \eqref{eq_dif1}.

Let us now tackle the proof of \eqref{eq_dif2}.
As in \eqref{PP2}, we get
\EQS{\label{PP22}
  \begin{aligned}
    \|w(t)\|_{L^2_x}^2
    &\le \|u_0-v_0\|_{L^2_x}^2
     +|\la|\sum_{j=1}^{7}|J_t^{(j)}|
     + |2\mu-\la| \sum_{j=8}^{14}|J_t^{(j)}|,
  \end{aligned}
}
where for $j=1,\dots,7$,
\EQQS{
  J_t^{(j)}
  := \sum_{N\ge 1}
   \RE \int_0^t\int_\R \tilde{g}_j(u_1,u_2) P_N^2 \p_x \bar{w} dxdt'
}
and for $j=8,\dots,14$,
\EQQS{
  J_t^{(j)}
  := \sum_{N\ge 1}
   \RE \int_0^t\int_\R g_j(u_1,u_2) P_N^2 \bar{w} dxdt'.
}
It thus remains to control $ J_t^{(j)} $ for $ j=8,\dots,14$.
First we notice that since
$\|w z \partial_x \bar{\psi}\|_{L_x^2}$
$+$
$\|w \psi \p_x\bar{\psi}\|_{L_x^2}\les K_1^2 \|w\|_{L_x^2}$,
we directly get (recall that $ s=1$)
\EQQS{
  |J_t^{(10)}|+|J_t^{(13)}|
  \les T K_1^2 \|w\|_{L_T^\I L_x^2}^2.
}

\noindent
\textbf{Estimate for $J_t^{(8)}$.}
We rewrite $ u_1^2 \p_x \bar{w} $ as $u_1^2 \p_x \bar{w}= \p_x( u_1^2  \bar{w})-\p_x (u_1^2) \bar{w}$.
Then the contribution of $\p_x(u_1^2 \bar{w})$ can be estimated in exactly the same way as $ J_t^{(2)} $, and it thus suffices to estimate the contribution of $ \p_x (u_1^2) \bar{w}$:
\EQQS{
  J_t^{(8,1)}
  = \sum_{N\gg 1}
   \Bigl| \sum_{\substack{N_1,N_2,N_3 \ge 1,\\N_4\sim N}}
   \int_0^t \int_{\R} \p_x (P_{N_1} u_1 P_{N_2} u_1)
   P_{N_3} \bar{w} P^2_N P_{N_4} \bar{w} dxdt' \Bigr|.
}
By symmetry, we may assume $N_1\ge N_2 $.

\noindent
\underline{Case 1: $N_3\les 1$.}
We have $N_1\gts N$.
Then, integration by parts and the discrete Young inequality lead to
\EQQS{
  J_t^{(8,1)}
  & \les \int_0^t
   \sum_{N_1\gts N\gg 1} \Bigl(\frac{N}{N_1}  \Bigr)
   \|P_{N_1} u_1(t')\|_{H_x^1} \|P_N w(t')\|_{L^2_x}
   \|w(t')\|_{L^2_x}
   \|u_1(t')\|_{L^\infty_x} dt' \\
  & \les T \|u_1\|_{L^\infty_T H_x^1}^2
   \|w\|_{L^\infty_T L^2_x}^2
  \les T K_1^2
   \|w\|_{L^\infty_T L^2_x}^2.
}

\noindent
\underline{Case 2: $N_1\ll N_3$.}
In this case, we have $ N_3\sim N_4\sim N\gg 1 $.
Then, \eqref{eq_res2} of Lemma \ref{lem_res1} ensures that
$ |\Om | \sim N^2 $.
We see from \eqref{eq_I2.2} of Lemma \ref{lem_Bourgain2} that
\EQQS{
  |J_t^{(8,1)}|
  \les T^{1/4} K_1^2 \|w\|_{Z^{0}_T}
   \|w\|_{L^\infty_T L^2_x}.
}
\underline{Case 3: $N_1\gts N_3\gg 1$.}
Notice that $ N_1\gts N$.

\noindent
\underline{Subcase 3.1: $N_2\gts N_3$.}
Then we make use of discrete Young's inequality together with
\eqref{estwLinfi2} and \eqref{estL1} to get
\EQQS{
  J_t^{(8,1)}
  &\les \sum_{N_1\gts N\gg 1} \Bigl(\frac{N}{N_1}\Bigr)
   \int_0^t
   \|P_{N_1}u_1\|_{H_x^1} \|P_N w\|_{L^2_x}
   \sum_{1\ll N_3\les N_2\le N_1}
   \| P_{N_2}u_1\|_{L_x^\infty}
   \| P_{N_3}w\|_{L^\infty_x}dt'\\
  &\les T^{1/2} \|u_1\|_{L^\infty_T H_x^1}
   \|w\|_{L^\infty_T L^2_x}
   \sum_{1\ll N_3\les N_2}
   \| P_{N_2}u_1\|_{L^4_T L_x^\infty}
   \| P_{N_3}w\|_{L^4_T L^\infty_x}\\
  &\le C(K_1) T^{1/2} \|w\|_{L^\infty_T L_x^2}^2
   \sum_{1\ll N_3\les N_2} N_2^{-\frac{3}{4}} N_3^{\frac{1}{4}}
  \le C(K_1) T^{1/2} \|w\|_{L^\infty_T L^2_x}^2.
}

\noindent
\underline{Subcase 3.2: $1\le N_2\ll N_3$.}
When $N\gtrsim N_3$ or $N_3\gg N\gg N_2$, Lemma \ref{lem_res1} gives $ |\Om |\sim N_1(N\wedge N_3) $ and \eqref{eq_I2.1} of Lemma \ref{lem_Bourgain2} yields
\EQQS{
  |J_t^{(8,1)}|
  \les T^{1/4} K_1^2 \|w\|_{Z^{0}_T} \|w\|_{L^\infty_T L^2_x}.
}
On the other hand, when $N\les N_2$, we can apply the argument of Subcase 1.1 in the estimate of $J_t^{(1)}$ with $s=1$ so that
\EQQS{
  |J_t^{(8,1)}|
  \les T^{1/2} K_1^2 \|w\|_{L^\infty_T L^2_x}^2.
}

\noindent
\textbf{Estimate for $J_t^{(9)}$.}
We eventually get
\EQQS{
  | J_t^{(9)}|
  \les \sum_{N\gg 1}
   \Bigl| \sum_{\substack{N_1,N_2,N_3 \ge 1,\\N_4\sim N}}
   \int_0^t \int_{\R} P_{N_1} z \partial_x P_{N_2} \bar{u}_2 P_{N_3}w
    P^2_N P_{N_4} \bar{w} dxdt'\Bigr|.
}

\noindent
\underline{Case 1: $N_3\les 1$.}
Notice that $N_1\vee N_2 \gts N $.
We treat only the most diffcult case that is $ N_2\gg N_1 $ since for $ N_1\gts N_2 $ we may share the derivative on $z$ and $ v$ and close the estimate.
When $N_2\gg N_1$, it holds that $N_2\sim N$.
The Bernstein and the discrete Young inequality lead to
\EQQS{
  |J_t^{(9)}|
  & \les \sum_{N\gg 1}\int_0^t
  \|P_{\sim N} u_2\|_{H_x^1} \|P_N w\|_{L^2_x}
  \|w\|_{L^2_x} \|z\|_{H_x^1} dt'\\
  & \les T  \|u_2\|_{L^\infty_T H_x^1} \|z\|_{L^\infty_T H_x^1} \|w\|_{L^\infty_T L^2_x}^2
  \les T K_1^2 \|w\|_{L^\infty_T L^2_x}^2.
}

\noindent
\underline{Case 2: $N_1\vee N_2\ll N_3$.}
We have $ N_3\sim N_4\sim N\gg 1 $ in this case.

\noindent
\underline{Subcase 2.1: $N_1\gts N_2$}.
Then, we share the derivative on $z$ and $ v$ and conclude the estimate by using \eqref{estL1}.

\noindent
\underline{Subcase 2.2: $N_1\ll N_2$}.
Then Lemma \ref{lem_res1} gives
$ |\Om |\sim N_2 N $, and \eqref{eq_I2.1} of Lemma \ref{lem_Bourgain2} yields
\EQQS{
  |J_t^{(9)}|
  \les T^{1/4} K_1^2 \|w\|_{Z^{0}_T}
  \|w\|_{L^\infty_T L^2_x}.
}

\noindent
\underline{Case 3: $N_1\vee N_2\gts N_3\gg 1$.}

\noindent
\underline{Subcase 3.1: $N_1\gts N_2$.}
Since $ N_1\gts N $, by distributing the derivative and making use of \eqref{estwLinfi2}, we eventually obtain
\EQQS{
  |J_t^{(9)}|
  &\les T^{1/2}
   \sum_{N_1\gts N_2\vee N_3\vee N_4}
   N_2 \|P_{N_1}z\|_{L^\I_T L_x^2}
   \|P_{N_2}u_2\|_{L^\I_T L^2_x}
   \|P_{N_3}w\|_{L^4_T L^\infty_x}
   \|P_{N_4}w\|_{L^4_T L^\infty_x}\\
  &\le C(K_1) T^{1/2} \|z\|_{L^\I_T H_x^1}
   \|u_2\|_{L^\I_T H_x^1}
   \|w\|_{L^\I_T L^2_x}^2
   \sum_{N_1\gts N_2\vee N_3\vee N_4}
   N_1^{-1}N_3^{\frac{1}{4}}N_4^{\frac{1}{4}}\\
  &\le C(K_1) T^{1/2} \|w\|_{L^\infty_T L^2_x}^2.
}

\noindent
\underline{Subcase 3.2: $N_2\gg N_1\gts  N_3$}.
Notice that $N_2\sim N \sim N_4 $.
By \eqref{eq_res2} of Lemma \ref{lem_res1}, we have $|\Om|\sim N_2^2$.
Then, we can get the desired result easily by making use of \eqref{eq_I2.2} of Lemma \ref{lem_Bourgain2}.

\noindent
\underline{Subcase 3.3: $1\le N_1\ll N_3\les N_2$.}
If one of following configurations holds (a) $N_2\gg N_3$, (b) $N_2\sim N_3$ and $N_1\gg N$, or (c) $N_2\sim N_3$ and $N_1\ll N$, Lemma \ref{lem_res1} gives
$ |\Om |\gts N_3 (N\vee N_1) $.
In particular, we have the strong resonant relation $|\Om |\sim N_2^2$ when (a) holds.
Then, we see from \eqref{eq_I2.1} of Lemma \ref{lem_Bourgain2} that
\EQQS{
  |J_t^{(9)}|
  \les T^{1/4} K_1^2 \|w\|_{Z^{0}_T}
   \|w\|_{L^\infty_T L^2_x}.
}
On the other hand, when $N_2\sim N_3$ and $N\sim N_1$, it is easy to see that
\EQQS{
  |J_t^{(9)}|
  &\les \sum_{N_2\sim N_3\gg N_1\sim N_4}
   N_1^{-1}N_4^{\frac{1}{2}}
   \|P_{N_2}u_2\|_{L_T^2 H_x^1}
   \|P_{N_3}w\|_{L_{T,x}^2}
   \|J_x\psi\|_{L_{T}^2 L_x^\I}
   \|w\|_{L_T^\I L_x^2}\\
  &\les T K_1^2 \|w\|_{L_T^\I L_x^2}^2,
}
which completes the estimates for $ J_t^{(9)}$.

\noindent
\textbf{Estimate for $J_t^{(11)}$.}
We apply an argument almost parallel to the estimates for $J_t^{(9)}$.
We use the fact that $J_x^{1+\e}\psi\in L^\I(\R^2)$ instead of \eqref{estL1}.

Finally, as for $ J_t^{(12)} $, we notice that $g_{12}=\p_x(u_1\psi\bar{w})-\p_x u_1\psi\bar{w}-u_1\p_x\psi\bar{w}$.
The first and third terms are essentially handled by the estimates for $J_t^{(3)}$ and $J_t^{(10)}$, respectively.
On the other hand, the second term can be treated in a similar way as $J_t^{(11)}$.
As for $J_t^{(14)} $, we observe that $g_{14}=\p_x(\psi^2 \bar{w})-2\psi\p_x \psi \bar{w}$.
The first and second terms can be estimated by using the argument applied to $J_t^{(6)}$ and $J_t^{(13)}$, respectively.
By collecting the above estimates, we obtain \eqref{eq_dif2}, which completes the proof.
\end{proof}

To end this section, we adapt the above proposition on the difference  of two solutions $u_1$ and $u_2$ to the case where they are associated with possibly two different functions $ \psi_1$ and $\psi_2$ satisfying \eqref{hyp_psi}.
This result is needed to prove the continuity with respect to initial data in Zhidkov spaces.

\begin{prop}\label{prop_dif1Z1}
  Let $0<T<1$ and $s>0$.
  Let $u_1$ and $u_2$ be two solutions of \eqref{eq2} on $[0,T]$, associated with possibly two different functions $ \psi_1$ and $\psi_2 $, both belonging to $L^\infty_T H^s $, and emanating from $u_{0,1}\in H^s(\R)$ and $u_{0,2}\in H^s(\R)$, respectively.
  Let $K_s\ge 1$ such that
  \EQQS{
   \sum_{j=1}^2 \Big(\|u_j\|_{Z_T^{s\wedge 1}}
    +\|\p_t \psi_j\|_{L_{t,x}^\I}
    +\|J_x^{1+\e}\psi_j\|_{L_{T,x}^\I}
    + \|\Psi_j\|_{L_t^\I H_x^{\e}}\Big)
    \le K_s,
  }
  where $\Psi_i$ is defined in \eqref{def_psi}.
  Setting $w=u_1-u_2$ then,
  for $s\in (3/4,1]$ whenever $ \lambda=2\mu $ and $s =1 $ otherwise,  it holds
  \EQS{\label{estwLinfi1Z1}
    \begin{aligned}
      &\|P_N w\|_{L^4_T L^\infty_x}\\
      &\le C(K_s) N^{\frac{5}{4}-s}
      \Bigl(\|w\|_{L_T^\I H_x^{s-1}}
      + \|J_x(\psi_1-\psi_2)\|_{L^\infty_{T,x}}
      +\|\Psi_1-\Psi_2\|_{L^\infty_T L^2_x}\Bigr),
    \end{aligned}
  }
  \EQS{\label{estZ1Z1}
    \|w\|_{Z^{s-1}_T}
    \le C(K_s)\Bigl(  \|w\|_{L^\infty_T H_x^{s-1}}
    +\|J_x(\psi_1-\psi_2)\|_{L^\infty_{T,x}}
    + \|\Psi_1-\Psi_2\|_{L^\infty_T L^2_x}\Bigr),
  }
  and
  \EQS{\label{eq_dif1Z1}
    \begin{aligned}
      \|w\|_{L_T^\I  H_x^{s-1}}^2
      &\le \|u_{0,1}-u_{0,2}\|_{H_x^{s-1}}^2
       +C(K_s) T^{1/4}\|w\|_{Z_{T}^{s-1}}
      \Bigl(\|w\|_{L_T^\I H_x^{s-1}}\\
      &\quad\quad\quad\quad\quad\quad\quad
      + \|J_x(\psi_1-\psi_2)\|_{L^\infty_{t,x}}
      + \|\Psi_1-\Psi_2\|_{L^\infty_T L^2_x}\Bigr).
    \end{aligned}
  }
\end{prop}

\begin{proof}
As in \eqref{eq_w}, setting $ z=u_1+u_2$ and $\theta=\psi_1-\psi_2 $,  we notice that $w=u_1-u_2$ satisfies the following equation:
\EQQS{
  \p_t w
  &=i\p_x^2 w
   +\frac{\la}{2} \p_x \Bigl( |u_1|^2 w+u_2\RE(z\bar{w})
   + \psi_1\RE(z\bar{w})
   +2u_1 \RE(\psi_1 \bar{w})\\
  &\quad\quad\quad\quad\quad\quad\quad
   +2w \RE(\psi_2 \bar{u}_2)
   +2 \psi_1\RE(\psi_1 \bar{w})
   +|\psi_1|^2w\Bigr)\\
  &\quad+\frac{2\mu-\la}{2}\Bigl(u_1^2\p_x \bar{w}
   +zw\p_x \bar{u}_2
   +wz\p_x \bar{\psi}_1
   +2w\psi_1 \p_x \bar{u}_1\\
  &\quad\quad\quad\quad\quad\quad\quad
   +2u_2\psi_2\p_x \bar{w}
   +2w\psi_1\p_x \bar{\psi}_1
   +\psi^2_1 \p_x \bar{w}\Bigr)+ \Lambda(\theta)+\Psi_1-\Psi_2\\
  &=:i\p_x^2 w
   +\frac{\la}{2}\p_x \Bigl(  \sum_{j=1}^7 \tilde{\mathfrak{g}}_j(u_1,u_2)\Bigr)
   +\frac{2\mu-\la}{2}\sum_{j=8}^{14} \mathfrak{g}_j(u_1,u_2) + \Lambda(\theta) +\Psi_1-\Psi_2,
}
where
\EQQS{
  \Lambda(\theta)
  & =\la \p_x \Bigl( \frac{|u_2|^2}{2} \theta
   +u_1 \RE(\theta \bar{u}_2)
   + \frac{u_2}{2} \RE(\bar{\theta}(\psi_1+\psi_2))
   +\psi_1  \RE(\bar{u}_2 \theta)
   + \theta \RE (\bar{u}_2 \psi_2)\Bigr)\\
  &\quad+\frac{2\mu-\la}{2}
   \Bigl(\theta (\psi_1+\psi_2) \partial_x \bar{u}_2
   +u_2^2 \partial_x \bar{\te}
   + 2 u_2 \theta \partial_x \bar{u}_1
   +2u_2\theta \partial_x \bar{\psi}_1
   + 2 u_2 \psi_2 \partial_x \bar{\theta}\Bigr)\\
  &:= \la \partial_x\Lambda_1(\theta)
   +\frac{2\mu-\la}{2}\Lambda_2(\theta)
}
It is straightforward to verify that $\tilde{\mathfrak{g}}_j $ for $j=1,\dots,7$ and $\mathfrak{g}_j$ for $j=8,\dots,14$, are of exactly the same form as $ \tilde{g}_j$ for $j=1,\dots,7$ and $g_j$ for $j=8,\dots,14$ in \eqref{eq_w}, respectively.
Thus, these terms can be treated in the same way as in Lemma \ref{LemZ1} and Proposition \ref{prop_dif1}.
On the other hand, for $ 3/4<s\le 1$, it follows directly from Lemma \ref{Lem24} that
\EQS{\label{New}
  \|\Lambda_1\|_{H_x^s} \le C(K_s) \|J_x\theta\|_{L_x^\I}
  \quad \text{and} \quad
  \|\Lambda_2\|_{L_x^2} \le C(K_1) \|J_x\theta\|_{L_x^\I}.
}
Therefore, the proof of Lemma \ref{LemZ1} leads to  \eqref{estwLinfi1Z1}-\eqref{estZ1Z1}.
Moreover, we deduce from \eqref{New} that
\EQQS{
  &\sum_{N\ge 1} N^{2(s-1)} \Bigl|\int_0^t\int_{\R}
   \partial_x P_N\Lambda_1(\theta) P_N w dxdt' \Bigr|\\
  &\les T \|\Lambda_1(\theta)\|_{L_T^\I H_x^s}
   \|w\|_{L_T^\I H_x^{s-1}}
  \le C(K_s)T \|J_x(\psi_1-\psi_2)\|_{L_{T,x}^\I}
   \|w\|_{L_T^\I H_x^{s-1}}
}
and
\EQQS{
  &\sum_{N\ge 1} N^{2(s-1)} \Bigl|\int_0^t\int_{\R}
   P_N\Lambda_2(\theta) P_N w dxdt' \Bigr|\\
  &\lesssim T \|\Lambda_2(\theta)\|_{L_T^\I L^2_x}\|w\|_{L_T^\I L^2_x}
  \le C(K_1)T\|J_x(\psi_1-\psi_2)\|_{L_{T,x}^\I}
   \|w\|_{L_T^\I L^2_x}.
}
The proof of Proposition \ref{prop_dif1}, combined with the above estimates and using \eqref{estwLinfi1Z1} in place of \eqref{estwLinfi1} or \eqref{estwLinfi2}, as well as \eqref{estZ1Z1} instead of \eqref{estZ1} or \eqref{estZ2}, yields \eqref{eq_dif1Z1}, which completes the proof.
\end{proof}

\section{Estimate for the Difference in the Nonregular Case}
\label{nonregular}

In this section, we discuss estimates for the difference when $\la=2\mu$ does not hold.
Recall that we assume $\la\neq0$.
This corresponds to the case where we cannot place the difference $w \in Z^{s-1}_T $ for $s<1$ since the equation is no longer in divergence form.
Nonlinear terms such as $ u\bar{u} \p_x w $ exhibit unfavorable interactions, specifically of the type {\it low$\times$high$\times$high$\to$low}, which are resonant and not well defined for $ u\in Z^s_T $ and $ w\in Z^{s-1}_T $.
At the same time, this issue makes it difficult to apply refined Strichartz estimates (see Proposition \ref{prop_stri1}), since we would like to use estimates of the form $\||u_1|^2\p_x w\|_{H^{s-2}}\les \|u_1\|_{H^s}^2\|w\|_{H^{s-1}}$, which holds when $s\ge 1$.

On the other hand, in the context of the Benjamin-Ono equation (that has the same dispersion order as the Schr\"odinger equation), it was shown in \cite{MPV18} that the difference of two solutions can be estimated in $ Z^{s-1/2}$ by introducing a modified energy that cancels the most problematic nonresonant interactions.
It is worth noting that the corresponding term, such as $u^2\p_x\bar{w}$, exhibits very strong nonresonance (see \eqref{eq_I2.2} in Lemma \ref{lem_Bourgain2}), so it is not necessary to use the modified energy; instead, we can directly apply the argument from Proposition \ref{prop_dif1}.
In \cite{MPV18}, this approach allows one to prove local well-posedness of the Benjamin-Ono equation in $ H^s(\R) $ for $ s>1/4 $.
This threshold corresponds to the minimal regularity required for the product $ u w $ to be well-defined when $ u_1\in H^s(\R) $ and $ w\in H^{s-1/2}(\R)$.
In this section, we follow this strategy, but we require that $ |u_1|^2 \p_x w $ be well-defined for $ u_1\in H^s(\R) $ and $ w\in H^{s-1/2}(\R) $, which forces $ s+(s-3/2) >0 \Leftrightarrow s>3/4 $.
Although the contribution from the time derivative of the correction terms arising from the \textit{linear} part of the equation eliminates the worst terms involving derivative loss, the contribution from \textit{nonlinear} terms may still contain some unfavorable configurations.
One such problematic case occurs when the derivative loss falls again on the difference $ w$ in the nonlinear term, particularly when $w$ has the highest frequency.
However, it was observed in \cite{EW19,MPV18} that a fundamental cancellation occurs in such terms for quadratic KdV-like equations, and we recover a similar cancellation in our setting.

Another difficulty results from the fact that the auxiliary function $\psi $ does not belong to $ L^2(\R) $.
More precisely, the modified energy method produces terms with anti-derivative (defined in \eqref{def_anti}) of the form $ \partial_x^{-1}(\dot{P}_M( w u_1 )) $ and $ \partial_x^{-1}(\dot{P}_M( w \psi )) $ that can be problematic when $ M>0 $ is very small.
Here, $\dP_M$ denotes the homogeneous Littlewood-Paley operator.
Actually, the first one is easy to handle: taking the $ L^\I $-norm and using the Bernstein inequality, the anti-derivative allows us to recover the $ L^1$-norm of the product $w u $, which is convenient since both functions belong to $ L^2(\R) $.
On the other hand, estimating the second term is more delicate.
First, control of $\psi$ in $ L^\I_x $ alone is insufficient, since $ w$ does not belong to $ L^1(\R) $.
Second, the anti-derivative causes a derivative loss in the highest frequency if we attempt to recover the $L^2$-norm (instead of $L^1$-norm) via the Bernstein inequality, which is not acceptable in our setting.
The key idea here is to assume, in addition to \eqref{hyp_psi}, that $\p_x \psi\in L^1(\R) $.
As discussed in Subsection \ref{subs_set}, this is a reasonable hypothesis and allows us to control the frequency projections of $ \psi $ in $ L^1$.
Note that, in order to exploit this hypothesis, it is also necessary to place the difference $ w$ in function spaces equipped with a weight on the low frequencies.
Such weighted spaces for the difference are frequently used in the context of low regularity results (see, for instance, \cite{IKT, MPV18}).
Our low frequency weight will be close to, but strictly larger than, $ -1/2 $, since the nonlinear term does not take the divergence form.
More precisely, we take the weight to be $\LR{N^{-1/2+\de}} $ for some $ 0<\delta \ll 1 $.
In conclusion, for the rest of this paper, we impose the following hypotheses on the given function $\psi$:
\EQS{\label{hyp_psi2}
  \p_x \psi\in L^\I(\R;L^1(\R))\  \textrm{with}\ \eqref{hyp_psi}.
}

\begin{rem}\label{rem_psi}
  It follows from \eqref{hyp_psi2} that $\p_x\psi\in L^\I(\R;L^2(\R))$, which implies that $|\psi|^2\p_x \psi,\psi^2\p_x\bar{\psi}\in L^\I(\R;L^2(\R))$.
  This together with $\Psi\in L^\I(\R;H^{s+\e}(\R))$ implies that $i\p_t\psi+\p_x^2\psi\in L^\I(\R;L^2(\R))$.
  This is a key observation in the estimates for the modified energy $\E_N^3$ in Proposition \ref{prop_dif2}.
\end{rem}

The a priori estimate for the difference $w=u_1-u_2$ is based on the energy estimate with correction terms, also known as the modified energy.
We use the anti-derivative \eqref{def_anti} to define correction terms \eqref{def_E2}.
This construction requires the use of homogeneous decompositions to ensure that $\p_x^{-1}\dP_M f$ is well-defined.
As a consequence, we need to consider nonlinear interactions involving very low frequencies, such as those appearing on the left hand side of \eqref{trf2}.
To address this, we work in a function space equipped with a weight at very low frequencies, which is defined below.

\begin{defn}
  For $s,\delta\in\R$, we define the Banach space
  \EQQS{
    \ov{H}^{s,\de}(\R)
    =\{\varphi\in H^s(\R): \|\varphi\|_{\overline{H}^{s,\de}}<\I\},
  }
  with
  \EQQS{
    \|\varphi\|_{\overline{H}^{s,\de}}
    :=\bigg(\sum_{N\in 2^\Z} (N^{-1+2\de}\vee N^{2s})
      \|\dot{P}_N\varphi\|_{L^2}^2\bigg)^{1/2}.
  }
\end{defn}

Remark that $\ov{H}^{s,\de}(\R) \hookrightarrow H^s(\R)$.
The aim of this article is to show the local well-posedness in $H^s(\R)$ rather than in $\ov{H}^{s,\de}(\R)$.
However, without loss of generality, we can assume that the two initial data coincide at very low frequencies $0<N<1$ when using the norm $\|\cdot\|_{\ov{H}^{s,\de}}$.
See the proof of Theorem \ref{thm2} in Section \ref{sec_proof}.
In what follows, we will take $\de>0$ sufficiently small.

\begin{lem}\label{lem_weight}
  Let $0<T<1$, $3/4< s < 1$, $N_0\gg 1$ and  $\de>0$ be sufficiently small.
  Let $u_1$ and $u_2$ be two solutions of \eqref{eq2} on $[0,T]$ belonging to $Z_T^s$ and emanating from $u_{0,1}\in H^s(\R)$ and $u_{0,2}\in H^s(\R)$, respectively.
  Assume that there exists $K>0$ such that
  \EQQS{
    \|u_1\|_{Z_T^{s}}+\|u_2\|_{Z_T^{s}}
    +\|\p_x\psi\|_{L_T^\I L_x^1}
    +\|J_x^{s+1+\e}\psi\|_{L_{T,x}^\I}
    \le K,
  }
  Then, there exists $C=C(K,N_0)>0$ such that
  \EQQS{
    \|P_{\le N_0} w\|_{L_T^\I \overline{H}^{0,\de}}
    \les \| P_{\le N_0}(u_{0,1}-u_{0,2})\|_{\overline{H}^{0,\de}}
     +C(K,N_0) \|w\|_{L_T^\I H_x^{1/4+}},
  }
  where $w=u_1-u_2$.
\end{lem}

\begin{proof}
  First, notice that $w$ satisfies the following:
  \EQQS{
    w(t)
    =e^{it\p_x^2}(u_{0,1}-u_{0,2})
     -i\int_0^t e^{i(t-t')\p_x^2}(F(u_1,\psi)-F(u_2,\psi))dt',
  }
  where $F(u,\psi)$ is defined by \eqref{def_F}.
  As in \eqref{eq_w}, we have
  \EQS{\label{eq_w3}
    \begin{aligned}
     &F(u_1,\psi)-F(u_2,\psi)\\
     &=i\la(|u_1|^2\p_x w+u_1\bar{w}\p_x u_2+w\bar{u}_2\p_x u_2
      +w\bar{z}\p_x \psi+2\RE(u_1\bar{\psi})\p_x w\\
     &\quad\quad\quad+2\RE(w\bar{\psi})\p_x u_2
      +2\RE(w\bar{\psi})\p_x \psi+|\psi|^2\p_x w)\\
     &\quad+i\mu(u_1^2\p_x \bar{w}+zw\p_x \bar{u}_2
      +wz\p_x \bar{\psi}
      +2u_1\psi\p_x \bar{w}
      +2w\psi \p_x \bar{u}_2\\
     &\quad\quad\quad
      +2w\psi\p_x \bar{\psi}+\psi^2 \p_x \bar{w})\\
      &=i\la \sum_{j=1}^8 h_j(u_1,u_2) + i\mu \sum_{j=9}^{15} h_j(u_1,u_2).
    \end{aligned}
  }
  Since the Schr\"odinger group is unitary in
  $ \overline{H}^{0,\delta} $, it suffices to bound $F(u_1,\psi)-F(u_2,\psi)$ in $ L^\infty_T \overline{H}^{0,\delta}$.
  As we will see below, we do not have to use time integrability.
  By the Minkowski inequality, we mainly consider $\|F(u_1,\psi)(t')-F(u_2,\psi)(t')\|_{\ov{H}^{0,\de}}$ for each $t'\in [0,t]$.
  In what follows, we drop $t'$ for simplicity.
  Notice that we have a different formulation compared to \eqref{eq_w}.
  So, we need to estimate 15 terms.
  For $h_1(u_1,u_2)=|u_1|^2 \p_x w$, note that $|u_1|^2 \p_x w=\p_x (|u_1|^2 w)-w\p_x (|u_1|^2)$.
  Then, it is easy to see that
  \EQQS{
  \|P_{\le N_0} \p_x (|u_1|^2 w)\|_{\overline{H}^{0,\de}}
   & \le \sum_{0<N\le N_0}   \|\dot{P}_N \p_x(|u_1|^2 w)\|_{\overline{H}^{0,\de}} \\
   & \les \Bigl( \sum_{0<N< 1 }  N^{\frac{1}{2}+\de}
    +\sum_{1\le N\le N_0} N \Bigr) \||u_1|^2w\|_{L_x^2} \\
   & \les N_0 \|u_1\|_{L_x^\I}^2 \|w\|_{L_x^2}\, .
  }
  Now, we estimate $w\p_x (|u_1|^2)$.
  This term is cumbersome since it is not a divergence form.
  We decompose
  \EQQS{
  \dot{P}_{N}  (w\p_x (|u_1|^2))=\sum_{N_1,N_2>0}\dot{P}_{N}
   \Bigr( \dot{P}_{N_1}w\p_x \dot{P}_{N_2}(|u_1|^2)\Bigl).
  }
  When $N_1\les N$, we have $N_2\les N$.
  Therefore, we can bound this contribution as follows
  \EQQS{
    &\sum_{0<N\le N_0} \sum_{0<N_1,N_2\les N}
    \Big\|\dot{P}_N
     \Big(\dot{P}_{N_1}w\p_x \dot{P}_{N_2}(|u_1|^2)\Big) \Big\|_{\overline{H}^{0,\de}}\\
    &\les  \Bigl( \sum_{0<N< 1 }  N^{-\frac{1}{2}+\de}
     +\sum_{1\le N\le N_0} \Bigr) \sum_{0<N_1,N_2\les N}
     \|\dot{P}_{N_1}w\|_{L_x^\infty}\| \p_x\dot{P}_{N_2}(|u_1|^2)\|_{L_x^2}\\
    &\les  \Bigl( \sum_{0<N< 1 }  N^{\de}
     +\sum_{1\le N\le N_0} N^{\frac{1}{2}} \Bigr) \sum_{0<N_1,N_2\les N} N_1^{1/2} N_2^{1/2}
     \|w\|_{L_x^2}\| |u_1|^2\|_{L_x^2}\\
    &\les N_0^{\frac{3}{2}}\|u_1\|_{H_x^{\frac{1}{2}+}}^2 \|w\|_{L_x^2}.
  }
  On the other hand, when $N_1\gg N$ we have $N_1\sim N_2$.
  By the Bernstein inequality and \eqref{eq2.2}, we have
  \EQQS{
     &\sum_{0<N\le N_0} \sum_{N_1\sim N_2\gg N}
     \Big\| \dot{P}_N
     \Big(\dot{P}_{N_1}w
     \p_x \dot{P}_{N_2}(|u_1|^2)\Big)
     \Big\|_{\overline{H}^{0,\de}}\\
    &\les\Bigl( \sum_{0<N<1 }  N^{\de}+\sum_{1\le N\le N_0} N^{\frac{1}{2}} \Bigr) \sum_{N_1\sim N_2\gg N} N_2
     \|\dot{P}_{N_1} w\|_{L_x^2}\|\dot{P}_{N_2}(|u_1|^2)\|_{L_x^2}\\
   &
    \les N_0^{\frac{1}{2}}
    \|u_1\|_{H_x^{\frac{3}{4}+}}^2 \|w\|_{H_x^{\frac{1}{4}}}.
  }
  In a similar manner, we can show that
  \EQQS{
    \|P_{\le N_0} h_j\|_{\overline{H}^{0,\de}}
    \les  C(K,N_0) \|w\|_{H_x^{\frac{1}{4}}}
  }
  for $j=2,3,9,10$.
  It is easy to see that
  \EQQS{
    \|P_{\le N_0} h_4\|_{\overline{H}^{0,\de}}
    \le   \sum_{0<N\le N_0} \|\dot{P}_N(w\bar{z}\p_x \psi)\|_{\overline{H}^{0,\de}}
    \les N_0^{\de} \|w\bar{z}\p_x \psi\|_{L_x^1}
    \les C(K,N_0) \|w\|_{L_x^2}.
  }
  By exactly the same argument, we also obtain $\|P_{\le N_0} h_{11}\|_{\overline{H}^{0,\de}}\le C(K,N_0) \|w\|_{L_x^2}$.
  For $h_6$, we proceed as for $ h_1 $ to estimate $P_{\le N_0}(w\bar{\psi}\p_x u_2)$.
  We decompose
  \EQQS{
    P_{\le N_0} (w\bar{\psi}\p_x u_2)
    =\sum_{0<N\le N_0} \sum_{N_1,N_2>0} \dot{P}_{N_0}\Bigl( \dot{P}_{N_1}(w\bar{\psi})\p_x \dot{P}_{N_2}u_2 \Bigr) .
  }
  When $N_1\les N$, we have $N_1,N_2\les N$.
  Then, we obtain
   \EQQS{
    &\sum_{0<N\le N_0}
     \sum_{0<N_1,N_2\les N}
     \Big\|\dot{P}_N\Big(\dot{P}_{N_1}(w\bar{\psi})
     \p_x \dot{P}_{N_2}u_2\Big) \Big\|_{\overline{H}^{0,\de}}\\
    &\les  \Bigl( \sum_{0<N< 1 }  N^{\de}
     +\sum_{1\le N\le N_0} N^{\frac{1}{2}} \Bigr) \sum_{0<N_1,N_2\les N} N_1^{1/2} N_2^{1/2}
     \|w\bar{\psi} \|_{L_x^2}\| u_2\|_{L_x^2}\\
    &
    \les N_0^{\frac{3}{2}} \|\psi\|_{L^\infty_{x}} \|u_2\|_{ L^2_x} \|w\|_{L_x^2}.
  }
  On the other hand, when $N_1\gg N$, we have $N_1\sim N_2$.
  Then, we get
  \EQQS{
    &\sum_{0<N\le N_0} \sum_{N_1\sim N_2\gg N}
     \Big\| \dot{P}_N
     \Big(\dot{P}_{N_1}(w\bar{\psi})\p_x \dot{P}_{N_2}u_2\Big) \Big\|_{\overline{H}^{0,\de}}\\
    &\les\Bigl( \sum_{0<N< 1 }  N^{\de}
     +\sum_{1\le N\le N_0} N^{\frac{1}{2}} \Bigr) \sum_{N_1\sim N_2\gg N} N_2
     \|\dot{P}_{N_1}(w\bar{\psi})\|_{L_x^2}\|\dot{P}_{N_2}u_2\|_{L_x^2}\\
    &\les N_0^{\frac{1}{2}}
     \|w\bar{\psi}\|_{H_x^\frac{1}{4}}
     \|u_2\|_{H_x^{\frac{3}{4}+}}
    \les  N_0^{\frac{1}{2}}
     \|J_x^{\frac{1}{4}}\psi\|_{L_{x}^\I}
     \|u_2\|_{H_x^{\frac{3}{4}+}}
     \|w\|_{H_x^\frac{1}{4}},
  }
  where we used \eqref{eq2.6} in the last step.
 This shows the estimate for $h_6$.
  By the same way, we can estimate $h_5,h_{12},h_{13}$.
  Finally, we treat $h_7$.
  It remains to estimate $w\bar{\psi}\p_x \psi$ since ``$\RE$" does not play any role in the following argument.
  Once we can close the estimate for $h_7$, we can close estimates for $h_8,h_{14},h_{15}$ by noticing $|\psi|^2\p_x w=\p_x(|\psi|^2w)-w\psi\p_x\bar{\psi}-w\bar{\psi}\p_x \psi$.
 Recall that, by Remark \ref{rem_psi}, $\|\partial_x \psi \|_{L^\infty_t L^2_x}\le K  $, this term can be directly bounded as follows
  \EQQS{
  \|P_{\le N_0} (w\bar{\psi}\p_x \psi)\|_{\overline{H}^{0,\de}}
  & \le \sum_{0<N\le N_0}
   \|\dot{P}_N \p_x(w\bar{\psi}
   \p_x \psi)\|_{\overline{H}^{0,\de}} \\
  & \les \Bigl( \sum_{0<N< 1 }  N^{\de}
   +\sum_{1\le N\le N_0} N^{\frac{1}{2}} \Bigr)
   \|w\|_{L^2_x} \|\partial_x \psi\|_{L^2_x} \|\psi\|_{L^\infty_x}\\
  & \les N_0^{\frac{1}{2}}
   \|\partial_x \psi\|_{L^2_x} \|\psi\|_{L^\infty_x} \|w\|_{L_x^2},
  }
  which completes the proof.
\end{proof}

\begin{lem}\label{LemZ22}
  Let $0<T<1$.
  Let $u_1$ and $u_2$ be two solutions of \eqref{eq2} on $[0,T]$ belonging to $ L^\infty_T H^s $ with $ s\in (3/4,1] $.
  For any $\te\in(1/4,s-1/2)$,  $ w=u_1-u_2 $ satisfies
  \EQS{\label{estZ22}
    \|w\|_{Z^{\theta}_T}
    \le
     C(\|u_1\|_{L_T^\I H_x^s}, \|u_2\|_{L_T^\I H_x^s},
     \|J_x\psi\|_{L_{T,x}^\I})
     \|w\|_{L^\infty_T H^{\theta}}.
  }
\end{lem}

\begin{proof}
We proceed as in the proof of Lemma \ref{LemZ1}.
It suffices to estimate $\|w\|_{X^{\theta-1,1}_T} $.
By a similar argument to \eqref{tr1}, Lemma \ref{Lem24} shows that
\EQQS{
  \bigg\| \sum_{j=1}^7 \tilde{g}_j(u_1,u_2)\bigg\|_{H_x^{\te}}
  \le C(\|u_1\|_{L_T^\I H_x^s}, \|u_2\|_{L_T^\I H_x^s},
  \|J_x\psi\|_{L_{T,x}^\I})\|w\|_{H_x^{\te}},
}
where $\tilde{g}_j$ is defined in \eqref{eq_w}.
It remains to evaluate the contributions of $ h_j$ for $ j\in \{8,\dots,14\}$ which are also defined in \eqref{eq_w}.
We first notice that
$\|h_{13}(u_1,u_2)\|_{H_x^{\theta-1}} \lesssim \|h_{13}(u_1,u_2)\|_{L_x^2}\lesssim  \|J_x\psi\|_{L_x^{\infty}}^2\|w\|_{L_x^2} $ and
since $ \theta+(s-1)>0 $, we can directly check that Lemmas \ref{Lem24} and \ref{Lem25}, and Corollary \ref{cor26} lead to
\EQQS{
  \bigg\|\sum_{j=8}^{14} h_j(u_1,u_2)\bigg\|_{H_x^{\theta-1}}
  \le C(\|u_1\|_{L_T^\I H_x^s}, \|u_2\|_{L_T^\I H_x^s},
  \|J_x\psi\|_{L_{T,x}^\I})\|w\|_{H_x^\theta}.
}
For instance, we rewrite $h_{14}(u_1,u_2) $ as $ h_{14}(u_1,u_2)= \p_x(\psi^2  \bar{w})-2\psi\p_x\psi \bar{w} $.
The first term is divergence free and easy to treat.
We use Corollary \ref{cor26} with $(s_1,s_2,s_3)=(s,\te,\te-1)$ and apply Lemma \ref{Lem24} to the second term so that
\EQQS{
  \|\psi\p_x\psi \bar{w}\|_{H_x^{\te-1}}
  \les \|\p_x \psi\|_{L_x^\I} \|\psi w\|_{H_x^\te}
  \les \|J_x\psi\|_{L_x^\I}^2 \|w \|_{H_x^\te},
}
which completes the proof.
\end{proof}

We now introduce the modified energy for the difference $w$ of two solutions $u_1$ and $u_2$.
As a preliminary step, we define the anti-derivative operator $\p_x^{-1}$.
For a function $f\in \Sp(\R)$ such that $ \xi^{-1} \hat{f} \in \Sp(\R) $, we set
\EQS{\label{def_anti}
  \p_x^{-1}f=\F_x^{-1}((i\xi)^{-1}\hat{f}(\xi)).
}
Note in particular that, for such a function, we have $f=\p_x \p_x^{-1}f=\p_x^{-1} \p_x f$.
For $N_0\ge 1$, we define the modified energy for the difference by

\EQS{\label{def_E2}
  \E_N(u_1,u_2,N_0)=
  \begin{cases}
    \|\dot{P}_N w\|_{L_x^2}^2&\mathrm{for} \quad N\le N_0,\\
    \|\dot{P}_N w\|_{L_x^2}^2 +\sum_{j=1}^3 c_j \E_N^j(u_1,u_2)&\mathrm{for} \quad N> N_0,
  \end{cases}
}
where $w:=u_1-u_2$ and
\EQQS{
  \E_N^1(u_1,u_2)
  &:=\sum_{\substack{N_1\sim N_3\gg N_2\vee N_4,\\
     N_1^{-1}\le M\le N_1^{2/3}} }
    \RE i \int_\R \p_x^{-1}P_N^2 P_{N_1}\bar{w}\p_x P_{N_3}u_2
    \p_x^{-1}\dot{P}_M (P_{N_2}\bar{w}P_{N_4}u_1)dx,\\
  \E_N^2(u_1,u_2)
  &:=\sum_{\substack{N_1\sim N_3\gg N_2\vee N_4,\\
     N_1^{-1}\le M\le N_1^{2/3}}}
    \RE i \int_\R \p_x^{-1}P_N^2 P_{N_1}\bar{w}\p_x P_{N_3}u_2
    \p_x^{-1}\dot{P}_M (P_{N_2}w P_{N_4}\bar{u}_2)dx,\\
  \E_N^3(u_1,u_2)
  &:=\sum_{\substack{N_1\sim N_3\gg N_2\vee N_4,\\
     N_1^{-1}\le M\le N_1^{2/3}}}
   \RE  i \int_\R \p_x^{-1}P_N^2 P_{N_1} \bar{w}\p_x P_{N_3}u_2
    \RE \p_x^{-1}\dot{P}_M(P_{N_2}w P_{N_4}\bar{\psi})dx
}
and $c_1,c_2,c_3$ are real constants that will be fixed later in the proof of Proposition \ref{prop_dif2}.
Recall that $P_N$ is the non-homogeneous Littlewood-Paley operator, whereas $\dot{P}_N $ is the homogeneous one (see Subsection \ref{subs_notation}).

Let us explain the choice of this modified energy.
The obstruction to going below $ H^1(\R) $ without modifying the energy is due to the contribution of three bad terms, each involving a product of the form $P_N^2 P_{N_1}\bar{w}\p_x P_{N_3}u_2 $, multiplied respectively by $ \dP_M (P_{N_2}\bar{w} P_{N_4} u_1)$, $ \dP_M (P_{N_2}w P_{N_4} \bar{u}_2)$, and
$ \dP_M (P_{N_2}w P_{N_4} \bar{\psi})$ with $N_1\sim N_3\gg N_2\vee N_4$ and $  N_1^{-1}\le M\le N_1^{2/3}$.
For these interactions, the resonance function is given by
\EQQS{
  \Omega(\xi_1,\xi_2,\xi_3,\xi_4)
  = \xi_1^2+\xi_2^2-\xi_3^2-\xi_4^2=-2(\xi_1+\xi_4)(\xi_2+\xi_4).
}
Note in particular that $ |\Omega|\sim N_1 M $.
Instead of applying a Fourier multiplier operator whose symbol is the inverse of $(\xi_1+\xi_4)(\xi_2+\xi_4) $ to the problematic terms, we take a slightly different approach.
While such a multiplier would yield complete cancellation, it is not conveniently formulated in physical space.
Therefore, we choose to apply a Fourier multiplier operator involving only the inverse of the main contribution in $ \Omega $, namely $ \xi_1 (\xi_2+\xi_4)$.
This choice has the advantage that it can be easily expressed in physical space using the anti-derivative operator.
Naturally, this approach leaves a residual term in the cancellation of the bad terms.
However, such a term consists of a Fourier multiplier whose symbol is $\xi_4 \xi_1^{-1} $, which effectively allows us to exchange a high-frequency derivative for a lower-order one.
See the second term in the right hand side of \eqref{eq_exchange}.
This exchange is sufficient to obtain the desired estimates.
We also define the modified energy at the $H^\theta$-regularity associated with the difference of two solutions:
\EQS{\label{def_E}
  E^\theta(u_1,u_2,N_0)
  =\sum_{N\in 2^\Z} (N^{-1+2\de}\vee N^{2\theta}) |\E_N(u_1,u_2,N_0)|.
}

The following Proposition ensures that $E^\theta(u_1,u_2,N_0)$ can be treated as $\|u_1(t)-u_2(t)\|_{\overline{H}^{\theta,\de}}^2$ if we choose $N_0>1$ sufficiently large.
Thanks to this estimate with Lemma \ref{lem_weight}, we can deduce the a priori estimate for $\|u_1-u_2\|_{L_T^\I H_x^\theta}^2$ once we have the corresponding estimate for $E^\theta(u_1,u_2,N_0)$.

\begin{prop}[Coercivity of the modified energy]\label{prop_coer}
  Let $0<T\le 1$, $3/4< s < 1$ and $1-s< \theta< s-1/2$.
  Let $u_1,u_2\in  L_T^\I H^s$ and $J_x\psi\in L^\I(\R^2)$.
  Then, for $N_0\gg(\|u_1\|_{L_T^\I H_x^s}+\|u_2\|_{L_T^\I H_x^s}+\|J_x\psi\|_{L_{T,x}^\I})^\frac{2}{s-\theta-1/2}$ it holds
  \EQS{\label{eq_coercivity}
    \begin{aligned}
      \big|E^\theta(u_1(t),u_2(t),N_0)-\|w(t)\|_{\overline{H}^{\theta,\de}}^2\big|
      \le \frac14 \|w(t)\|_{\overline{H}^{\theta,\de}}^2.
    \end{aligned}
  }
  for $t\in[0,T]$.
  In particular, for any $t\in [0,T]$,
  \EQS{\label{eq_coercivity2}
    E^\te(u_1(t),u_2(t),N_0)
    \le \frac{5}{4} \|w(t)\|_{\ov{H}^{\te,\de}}^2
    \le \frac{5}{3} E^\te(u_1(t),u_2(t),N_0).
  }
\end{prop}

\begin{proof}
  For simplicity, we drop $t$.
  We see from \eqref{def_E} and the triangle inequality that
  \EQQS{
    \big|E^\theta(u_1,u_2,N_0)-\|w\|_{\overline{H}^{s,\de}}^2\big|
    \le \sum_{j=1}^3\sum_{N>N_0}N^{2\theta}|\E_N^j(u_1,u_2)|.
  }
  The Bernstein inequality shows
  \EQQS{
    &N^{2\theta}|\E_N^1(u_1,u_2)|\\
    &\les \|P_N w\|_{H_x^\theta}\|P_N u_2\|_{H_x^\theta}
     \sum_{N^{-1}\les M\les N^{2/3}}\sum_{N_2\vee N_4\ll N}
     M^{-1/2}\|P_{N_2} w\|_{L_x^2}N_4^{1/2}\|P_{N_4} u_1\|_{L_x^2}\\
    &\les N^{-s+\theta+1/2} \|P_N w\|_{H_x^\theta}\|P_N u_2\|_{H_x^s}
     \|w\|_{L_x^2}\|u_1\|_{H_x^s}
  }
  since $1-s<\theta<s-1/2$ ad we take the non-homogeneous decomposition to $P_{N_2}w$.
  Similarly, we have
  \EQQS{
    N^{2\theta}|\E_N^2(u_1,u_2)|
    \les N^{-s+\theta+1/2} \|P_N w\|_{H_x^\theta}\|P_N u_2\|_{H_x^s}
     \|w\|_{L_x^2}\|u_2\|_{H_x^s}.
  }
  Finally, we estimate $\E_N^3(u_1,u_2)$.
  In this case, we need to look at frequency interactions more carefully.
  We have
  \EQQS{
    N^{2\theta}|\E_N^3(u_1,u_2)|
    &\les \sum_{N^{-1}\les M\les N^{2/3}}
     \sum_{N_2\vee N_4\ll N}
     \|P_N w\|_{H_x^\theta}\|P_N u_2\|_{H_x^\theta}
     M^{-1/2}\|\dP_M(P_{N_2}w P_{N_4}\bar{\psi})\|_{L_x^2}\\
    &\les N^{-s+\theta+1/2} \|P_N w\|_{H_x^\theta}\|P_N u_2\|_{H_x^s}
    \|w\|_{L_x^2} \|J_x\psi\|_{L_{x}^\I}.
  }
  Therefore, by choosing $N_0$ sufficiently large (i.e., $N_0^{-s+\theta+1/2}(\|u_1\|_{L_T^\I H_x^s}+\|u_2\|_{L_T^\I H_x^s}+\|J_x\psi\|_{L_{T,x}^\I})^2\ll 1$)
  with a trivial embedding $\|w\|_{H_x^\theta}\les \|w\|_{\ov{H}^{\theta,\de}}$, we conclude the proof of \eqref{eq_coercivity}.
  Finally, \eqref{eq_coercivity2} follows directly from \eqref{eq_coercivity} and the triangle inequality.
\end{proof}

\begin{prop}\label{prop_dif2}
  Let $N_0>1$ be a dyadic number, $0<T<1$, $3/4< s < 1$, and $1/4< \theta< s/2-1/8$.
  Let $u_1$ and $u_2$ be two solutions of \eqref{eq2} on $[0,T]$ belonging to $Z_T^s$ and emanating from $u_{0,1}\in H^s(\R)$ and $u_{0,2}\in H^s(\R)$, respectively.
  Setting
  \EQS{
   \begin{aligned}\label{eq77}
     K_s= &\|u_1\|_{Z_T^{s}}+\|u_2\|_{Z_T^{s}}
      +\|J_x^{s+1+\e}\psi\|_{L_{t,x}^\I}\\
      &+\|\p_t \psi\|_{L_{t,x}^\I}+\|\p_x\psi\|_{L_{t}^\I L_x^1}
      +\|(i\p_t+\p_x^2)\psi\|_{L_t^\I L_x^2}
      +\|\Psi\|_{L_t^\I H_x^{s+1+\e}},
   \end{aligned}
  }
  where $\Psi$ is defined in \eqref{def_psi}.
  Then, there exists $C=C(K_s)>0$ such that
  \EQS{\label{eq_dif_nonreg}
    \sup_{t\in [0,T]}E^\theta(u_1(t),u_2(t),N_0)
    \le E^\theta(u_{0,1},u_{0,2},N_0)
    +CT^{1/4}\|w\|_{Z_T^\te}\|w\|_{L_T^\I \ov{H}^{\theta,\de}},
  }
  where we set $w=u_1-u_2$.
\end{prop}

\begin{proof}
  For simplicity, we put $z:=u_1+u_2$.
  Similarly to Propositions \ref{prop_apri} and \ref{prop_dif1}, we consider the time derivative of $\E_N(u_1(t),u_2(t),N_0)$ for each $N\in 2^\Z$.
  Once we have finished proving the desired estimates, we sum over $N$ based on \eqref{def_E}.
  We notice that the low frequency part $ N\le N_0 $ is already estimated by Lemma \ref{lem_weight}, it remains to treat the high frequency part $N>N_0$, which is the most delicate partof this article.
  By the fundamental theorem of calculus, we have
  \EQS{\label{eq_6.16}
    \begin{aligned}
       N^{2\theta}
      \E_N(u_1(t),u_2(t),N_0)
      &=N^{2\theta}
      \E_N(u_1(0),u_2(0),N_0)\\
      &\quad +\sum_{j=1}^{15}2\la_j N^{2\theta}
       \RE\int_0^t \int_\R  h_j(u_1,u_2) P^2_N \bar{w}dxdt'\\
      &\quad +\sum_{j=1}^{3} c_j  N^{2\theta}
       \int_0^t \frac{d}{dt'}\E_N^j (u_1(t'),u_2(t')) dt',
    \end{aligned}
  }
  where $h_j$ is defined in \eqref{eq_w3}, $\la_j=\la$ for $j=1,\dots,8$ and $\la_j=\mu$ for $j=9,\dots,15$.
  Now, for $j=1,\dots, 15$, we set
  \EQQS{
    J_t^{(j)}:=  N^{2\theta}
    \RE \int_0^t\int_\R h_j(u_1,u_2) P_N^2 \bar{w} dxdt'.
  }
  In what follows, we treat each $J_t^{(j)}$ for $j=1,\dots,15$.
  First it immediately follows from \eqref{eq2.2.1} that
  \EQQS{
  \sum_{N> N_0}|J_t^{(4)}|
  \les T\|h_4(u_1,u_2)\|_{L_T^\I H_x^\theta}\|w\|_{L_T^\I H_x^\theta}
  \les TK^2 \|w\|_{L_T^\I H_x^\theta}^2.
  }
  Similarly, we can estimate $J_t^{(7)}, J_t^{(11)}$ and $J_t^{(14)}$ by the same bound as above.
  Except for $J_t^{(2)}, J_t^{(3)}$ and $J_t^{(6)}$,
  the strategy to obtain the desired bound is the same as Proposition \ref{prop_dif1}, i.e., we use Bourgain type estimates \eqref{eq4.3} for nonresonant interactions whereas the refined Strichartz estimate (Proposition \ref{prop_stri1}) for resonant interactions after employing the non-homogeneous Littlewood-Paley decomposition to each $u_1,u_2$ and $\psi$, for example $u_1=\sum_{N_1\in 2^\N}P_{N_1}u_1$ and $\psi=\sum_{N_2\in 2^\N}P_{N_2}\psi$.
  Moreover, for $J_t^{(1)}, J_t^{(5)}$ and $J_t^{(8)}$, we take the complex conjugate of the integrand before decomposition so as to use integration by parts (to be precise, Lemma \ref{lem_comm1}).
  On the other hand, for $J_t^{(2)}, J_t^{(3)}$ and $J_t^{(6)}$, correction terms play an essential role in canceling out derivative losses.

  \noindent
  \textbf{Estimate for $J_t^{(1)}$.}
  Following \eqref{eq_symm} from the proof of Proposition \ref{prop_apri}, it suffices to estimate the following:
  \EQQS{
    J_t^{(1)}
    =N^{2\theta}
      \sum_{N_1,\dots,N_4\in 2^\N}
      \RE\int_0^t \int_\R \Pi_N(P_{N_1}w,P_{N_2}\bar{w})
      P_{N_3} u_1 P_{N_4}\bar{u}_1dxdt'
  }
  where $\Pi_N(\cdot,\cdot)$ is defined by \eqref{def_pi}.
  By taking the complex conjugate, we may assume that $N_1\ge N_2$.
  We consider $J_t^{(1)}$ by the case-by-case analysis.

  \noindent
  \underline{Case 1: $N_1\sim N_2\gts N_3\vee N_4$.}
  When $N_3\sim N_4$, we can use Proposition \ref{prop_stri1} on $P_{N_3}u_1$ and $P_{N_4}u_1$ after using Lemma \ref{lem_comm1}.
  On the other hand, when $N_3\gg N_4$ or $N_4\gg N_3$, we employ \eqref{eq_I2.1}.

  \noindent
  \underline{Case 2: $N_1\sim N_3\gts N_2\vee N_4$.}
  When $N_2\les N_4$, we can use Proposition \ref{prop_stri1} on $P_{N_2}u_1$ and $P_{N_4}u_1$.
  On the other hand, when $N_2\gg N_4$, we employ \eqref{eq_I2.2}.

  \noindent
  \underline{Case 3: $N_1\sim N_4\gts N_2\vee N_3$.}
  When $N_2\les N_3$, we can use Proposition \ref{prop_stri1} on $P_{N_2}u_1$ and $P_{N_3}u_1$.
  On the other hand, when $N_2\gg N_3$, we employ \eqref{eq_I2.1}.

  \noindent
  \underline{Case 4: $N_3\sim N_4\gts N_1\ge N_2$.}
  We use Proposition \ref{prop_stri1} on $P_{N_1}u_1$ and $P_{N_2}u_1$.

  \noindent
  \textbf{Estimate for $J_t^{(5)}$.}
  As in \eqref{eq_symm}, we have
  \EQQS{
    J_t^{(5)}
    =N^{2\theta}
     \sum_{N_1,\dots, N_4\in 2^{\N}} \RE
     \int_0^t \int_\R \Pi_N(P_{N_1}w, P_{N_2}\bar{w})\RE(P_{N_3}u_1 P_{N_4}\bar{\psi})dxdt'.
  }
  By the symmetry, we may assume that $N_1\ge N_2$.

  \noindent
  \underline{Case 1: $N_1\sim N_2\gts N_3\vee N_4$.}
  When $N_3\les N_4$, we can put one derivative on $P_{N_4}\psi$ after using Lemma \ref{lem_comm1}.
  On the other hand, when $N_3\gg N_4$, we see from Lemma \ref{lem_Bourgain3} that $|J_t^{(5)}|\les T^{1/4}K^2 \|w\|_{L_T^\I H_x^\theta}^2$.

  \noindent
  \underline{Case 2: $N_1\sim N_3\gts N_2\vee N_4$.}
  When $N_2\les N_4$, we can put $N_2^{2\theta}$ on $P_{N_4}\psi$, since $N_1^{2\theta}N_2\le N_1 N_2^{2\theta}$.
  When $N_2\gg N_4$, we can use \eqref{eq_I4.1} (resp. \eqref{eq_I4.2}) for $P_{N_3}\bar{u}_1P_{N_4}\psi$ (resp. $P_{N_3}u_1 P_{N_4}\bar{\psi}$) with
  \EQQS{
    (u_1,u_2,u_3)
    =(\p_x P_{N_1}w,P_{N_3}u_1, P_N^2 P_{N_2}w)\quad \mathrm{or}\quad (P_N^2 P_{N_1}w,P_{N_3}u_1, \p_x P_{N_2}w).
  }

  \noindent
  \underline{Case 3: $N_1\sim N_4\gts N_2\vee N_3$ or $N_3\sim N_4\gts N_1\ge N_2$.}
  Since $N_4$ is one of the highest frequencies, we can put derivatives on $P_{N_4}\psi$.

  \noindent
  \textbf{Estimate for $J_t^{(8)}$.}
  By the same way as the estimate for $J_t^{(5)}$ in the proof of Proposition \ref{prop_apri}, we obtain $|J_t^{(8)}|
  \les TK^2 \|w\|_{L_T^\I H_x^\theta}^2$.

  \noindent
  \textbf{Estimate for $J_t^{(9)}$.}
  Recall that $h_9=u_1^2 \p_x \bar{w}$.
  Since the estimate of $h_1=|u_1|^2\p_x w$ is closed as in the estimate for $J_t^{(1)}$, we can apply this argument to $h_9=u_1^2 \p_x \bar{w}$.
  Recall that \eqref{eq_I2.2} (which is better than \eqref{eq_I2.1}) is available to the most difficult interaction whereas \eqref{eq_I2.1} is strong enough.

  \noindent
  \textbf{Estimate for $J_t^{(10)}$.}
  We estimate the following:
  \EQQS{
    J_t^{(10)}
    =N^{2\theta}
      \sum_{N_1,\dots,N_4\in 2^\N}
      \RE\int_0^t \int_\R P_N^2 P_{N_1}\bar{w}P_{N_2}w
      \p_x P_{N_3}\bar{u}_2 P_{N_4}z dxdt'.
  }
  It suffices to consider the case $N_1\ge N_2$ since the case $N_2\ge N_1$ can be controlled by this case.

  \noindent
  \underline{Case 1: $N_1\sim N_2\gts N_3\vee N_4$.}
  This case can be treated by the same way as Case 1 in the estimates for $J_t^{(1)}$.

  \noindent
  \underline{Case 2: $N_1\sim N_3\gts N_2\vee N_4$.}
  When $N_1\sim N_3\sim N_2\sim N_4\gg 1$, we use Proposition \ref{prop_stri1}:
  \EQQS{
    \sum_{N>N_0}|J_t^{(10)}|
    &\les \sum_{N_1\sim N_3\sim N_2\sim N_4} N_1^{2\theta+1}
     \int_0^t \|P_{N_1}w\|_{L_x^2}\|P_{N_2}w\|_{L_x^2}
     \|P_{N_3}u_2\|_{L_x^\I}\|P_{N_4}z\|_{L_x^\I}dt'\\
    &\les \|w\|_{L_T^\I H_x^\theta}^2
     \sum_{N_3\sim N_4} N_3 \|P_{N_3}u_2\|_{L_T^2 L_x^\I}
     \|P_{N_4}z\|_{L_T^2 L_x^\I}
    \les TC(K)\|w\|_{L_T^\I H_x^\theta}^2.
  }
  On the other hand, when $N_1\sim N_3\gg N_2\vee N_4$, $N_1\sim N_3\gts N_2\gg N_4$ or $N_1\sim N_3\gts N_4\gg N_2$, we can use \eqref{eq_I2.2} after taking extensions defined in \eqref{def_ext}.

  \noindent
  \underline{Case 3: $N_3\sim N_4\gts N_1\ge N_2$.}
  When $N_1\sim N_2$, we use Proposition \ref{prop_stri1} on $P_{N_3}u_2$ and $P_{N_4}z$.
  When $N_1\gg N_2$, we can use \eqref{eq_I2.1} and after taking extensions defined in \eqref{def_ext}.

  \noindent
  \textbf{Estimate for $J_t^{(13)}$.}
  We estimate the following:
  \EQQS{
    J_t^{(13)}
    =N^{2\theta}
      \sum_{N_1,\dots,N_4\in 2^\N}
      \RE\int_0^t \int_\R P_N^2 P_{N_1}\bar{w} P_{N_2}w
      \p_x P_{N_3}\bar{u}_2 P_{N_4}\psi dxdt'.
  }
  By the same reason as in the estimate for $J_t^{(10)}$, it suffices to consider the case $N_1\ge N_2$.

  \noindent
  \underline{Case 1: $N_1\sim N_2\gts N_3\vee N_4$.}
  When $N_4\gts N_3$, we can put one derivative on $P_{N_4}\psi$, so it is easy to close this case.
  On the other hand, when $N_3\gg N_4$, we can use \eqref{eq_I4.1} after taking extensions.

  \noindent
  \underline{Case 2: $N_1\sim N_3\gts N_2\vee N_4$.}
  When $N_1\sim N_3\sim N_2\sim N_4$, we are allowed to put some derivatives to $P_{N_4}\psi$, which enables us to close this case.
  On the other hand, when $N_1\sim N_3\gg N_2\vee N_4$, $N_1\sim N_3\gts N_2\gg N_4$ or $N_1\sim N_3\gts N_4\gg N_2$, we can use \eqref{eq_I4.2} as in Case 2 in the estimates for $J_t^{(10)}$.

  \noindent
  \underline{Case 3: $N_3\sim N_4\gts N_1\ge N_2$.}
  Since $N_4$ is one of the highest frequencies, we can easily distribute some derivatives to $P_{N_4}\psi$.

  \noindent
  \textbf{Estimates for $J_t^{(12)}$ and $J_t^{(15)}$.}
  The contributions of $J_t^{(12)}$ and $J_t^{(15)}$ can be treated by the same way as the estimates for $J_t^{(5)}$ and $J_t^{(8)}$, respectively.

  Now, we have to take the contribution of correction terms into account so as to close the estimate.
  To be precise, some interactions in $J_t^{(2)}, J_t^{(3)}$ and $J_t^{(6)}$ cannot be closed by Bougain type estimates (such as Lemma \ref{lem_Bourgain1}) nor refined Strichartz estimates (Proposition \ref{prop_stri1}), and are cancelled by the contribution of correction terms.

  \noindent
  \textbf{Estimate for $J_t^{(2)}$.}
  We estimate the following:
  \EQQS{
    J_t^{(2)}
    &=N^{2\theta}\sum_{N_1,\dots,N_4\in 2^\N}
     \RE \int_0^t \int_\R
     P_N^2 P_{N_1}\bar{w} P_{N_2} \bar{w}
     \p_x P_{N_3}u_2 P_{N_4}u_1 dxdt'
  }
   In what follows, we only consider the case $N_1\ge N_2$ since we can easily close the estimate when $N_2\ge N_1$.
  For example, when $N_2\sim N_3\gg N_1\vee N_4$, we can distribute derivatives $N^{2\theta}$ to at least 3 functions, whereas we have to introduce the correction term $\E_N^1(u_1,u_2)$ when $N_1\sim N_3\gg N_2\vee N_4$.
  First, we consider cases where we can close the argument using Bourgain-type estimates and refined Strichartz estimates.
  We identify the most difficult case, and we eliminate it by choosing $c_1$ in the corresponding correction term.

  \noindent
  \underline{Case 1: $N_1\sim N_2\gts N_3\vee N_4$.}
  We can follow the argument of the estimate for $J_t^{(10)}$ in this case.
  When $N_3\les N_4$, we can use Proposition \ref{prop_stri1} on $P_{N_3}u_2$ and $P_{N_4}u_1$.
  On the other hand, we use \eqref{eq_I2.2} after taking extensions defined in \eqref{def_ext} when $N_3\gg N_4$.

  \noindent
  \underline{Case 2: $N_1\sim N_4\gts N_2\vee N_3$.}
  This case is similar to Case 1. We use Proposition \ref{prop_stri1} when $N_2\gts N_3$.
  When $N_3\gg N_2$, we use \eqref{eq_I2.1} after taking extensions defined in \eqref{def_ext}.

  \noindent
  \underline{Case 3: $N_3\sim N_4\gts N_1\ge N_2$.}
  When $N_1\sim N_2$, we use Proposition \ref{prop_stri1} and the Young inequality on $P_{N_1}w$ and $P_{N_2}w$ since $\theta> 1/4$.
  When $N_1\gg N_2$, we use \eqref{eq_I2.2} after taking extensions defined in \eqref{def_ext}.

  \noindent
  \underline{Case 4: $N_1\sim N_3\gg N_2\vee N_4$.}
  First, we take a homogeneous dyadic decomposition for $P_{N_2}\bar{w}P_{N_4}u_1=\sum_{M\in 2^\Z}\dP_M(P_{N_2}\bar{w}P_{N_4}u_1)$:
  \EQQS{
   &N^{2\theta}
     \sum_{\substack{N_1 \sim N_3\gg N_2\vee N_4,\\ M>0}}
     \RE \int_0^t \int_\R
     P_N^2 P_{N_1}\bar{w} \p_x P_{N_3}u_2
     \dP_M( P_{N_2} \bar{w} P_{N_4}u_1) dxdt'\\
   &=N^{2\theta}
     \sum_{\substack{N_1 \sim N_3\gg N_2\vee N_4,\\ 0<M< N_1^{-1}}}
     \RE \int_0^t \int_\R
     P_N^2 P_{N_1}\bar{w} \p_x P_{N_3}u_2
     \dP_M( P_{N_2} \bar{w} P_{N_4}u_1) dxdt'\\
   &\quad+N^{2\theta}\sum_{\substack{N_1 \sim N_3\gg N_2\vee N_4,\\ M\ge N_1^{-1}}}
     \RE \int_0^t \int_\R
     P_N^2 P_{N_1}\bar{w} \p_x P_{N_3}u_2
     \dP_M( P_{N_2} \bar{w} P_{N_4}u_1) dxdt'\\
   &=:J_t^{(2,1)}+J_t^{(2,2)}.
  }
  We notice that $N_1\gg 1$ since $N_1\sim N> N_0\gg 1$.
  When $M< N_1^{-1}$, the Bernstein inequality shows that
    \EQQS{
      &\sum_{N>N_0}|J_t^{(2,1)}|\\
      &\les \sum_{\substack{N_1 \sim N_3\gg N_2\vee N_4,\\ 0<M< N_1^{-1}}}
        N_1^{2\theta+1}N_4^{1/4}M^{3/4}
        \|P_{N_1}w\|_{L_{T,x}^2}
        \|P_{N_3}u_2\|_{L_{T,x}^2}
        \|P_{N_2}w\|_{L_T^\I L_x^2}
        \|P_{N_4}u_1\|_{L^\infty_T L_x^2}\\
      &\les \|w\|_{L^\infty_T \ov{H}^{\theta,\de}}
        \|u_1\|_{L^\infty_T H_x^{s}}
        \sum_{N_1 \sim N_3\gg 1} N_1^{-1/4}
        \|D_x^{\theta} P_{N_1}w\|_{L_{T,x}^2}
        \|D_x^{\theta+1/2} P_{N_3}u_2\|_{L_{T,x}^2}\\
      &\les TK^2 \|w\|_{L_T^\I H_x^\theta}^2.
    }
  Now we consider $J_t^{(2,2)}$, and further decompose this into two cases $M> N_1^{2/3}$ and $N_1^{-1}\le M\le N_1^{2/3}$:
  \EQQS{
    J_t^{(2,2)}
    &=N^{2\theta}
      \sum_{\substack{N_1 \sim N_3\gg N_2\vee N_4,\\ N_1^{2/3}< M}}
      \RE \int_0^t \int_\R
      P_N^2 P_{N_1}\bar{w} \p_x P_{N_3}u_2
      \dP_M( P_{N_2} \bar{w} P_{N_4}u_1) dxdt'\\
    &\quad+N^{2\theta}\sum_{\substack{N_1 \sim N_3\gg N_2\vee N_4,\\ N_1^{-1}\le M\le N_1^{2/3}}}
      \RE \int_0^t \int_\R
      P_N^2 P_{N_1}\bar{w} \p_x P_{N_3}u_2
      \dP_M( P_{N_2} \bar{w} P_{N_4}u_1) dxdt'\\
    &=:J_t^{(2,2,1)}+J_t^{(2,2,2)}.
  }
  When $M> N_1^{2/3}$, we see from the proof of Lemma \ref{lem_res1} that we have $|\xi_1^2+\xi_2^2-\xi_3^2-\xi_4^2|\gts N_1M\gts  N_1^{5/3}$, which is strong enough to close.
  To see this, we first observe
  \EQQS{
    \sum_{N> N_0}|J_t^{(2,2,1)}|
    \le \sum_{N> N_0}N^{2\theta}
     \sum_{\substack{N_1 \sim N_3\gg N_2\vee N_4,\\ N_1^{2/3}< M}}
     |I_{t,M}(P_N^2 P_{N_1}\bar{w},
      \p_x P_{N_3}u_2,P_{N_2}\bar{w},P_{N_4}u_1)|,
  }
  where we put
  \EQS{\label{def_I5}
    I_{t,M}(\al_1,\al_2,\al_3,\al_4)
    :=\RE\int_0^t \int_\R \al_1 \al_2 \dP_M (\al_3 \al_4)dxdt'.
  }
  Following the proof of Lemma \ref{lem_Bourgain1}, we have
  \EQQS{
    &|I_{t,M}(P_N^2 P_{N_1}\bar{w}, \p_x P_{N_3}u_2,P_{N_2}\bar{w},P_{N_4}u_1)|\\
    &\le |I_{\I,M}(P_N^2 P_{N_1}\1_{t,R}^{\text{high}}\bar{\check{w}}, \p_x P_{N_3}\1_t\check{u}_2,P_{N_2}\bar{\check{w}},P_{N_4}\check{u}_1)|\\
    &\quad+|I_{\I,M}(P_N^2 P_{N_1}\1_{t,R}^{\text{low}}\bar{\check{w}}, \p_x P_{N_3}\1_{t,R}^{\text{high}}\check{u}_2,P_{N_2}\bar{\check{w}},P_{N_4}\check{u}_1)|\\
    &\quad+|I_{\I,M}(P_N^2 P_{N_1}\1_{t,R}^{\text{low}}\bar{\check{w}}, \p_x P_{N_3}\1_{t,R}^{\text{low}}\check{u}_2,P_{N_2}\bar{\check{w}},P_{N_4}\check{u}_1)|
    =:I_{\I,M,1}+I_{\I,M,2}+I_{\I,M,3},
  }
  where $R=N_1^{1/3}M^{4/3}$ and $\check{w}=\rho_T(w)$ is the extension of $w$.
  See \eqref{def_ext} for its definition.
  Remark that $|\Om|\gts N_1M\gts N_1^{1/3} M^{5/3} =M^{1/3}R\gg R$ since $N_1^{2/3}<M\les N_1$, which implies that we can apply \eqref{eq4.6} to $I_{\I,M,3}$.
  We see from the proof of \eqref{eq_time} and the Bernstein inequality that
  \EQQS{
    I_{\I,M,1}
    &\les N_3 M^{1/4} \|\1_{t,R}^{\text{high}}\|_{L_t^1}
     \|P_N^2 P_{N_1}\check{w}\|_{L_t^\I L_x^2}
     \|P_{N_3}\check{u}_2\|_{L_t^\I L_x^2}
     \|P_{N_2}\check{w}\|_{L_t^\I L_x^4}
     \|P_{N_4}\check{u}_1\|_{L_{t,x}^\I}\\
    &\les T^{1/4} N_1^{3/4}M^{-3/4}
     \|P_N^2 P_{N_1}\check{w}\|_{L_t^\I L_x^2}
     \|P_{N_3}\check{u}_2\|_{L_t^\I L_x^2}
     \|P_{N_2}\check{w}\|_{L_t^\I L_x^4}
     \|P_{N_4}\check{u}_1\|_{L_{t,x}^\I},
  }
  which implies that
  \EQQS{
    &\sum_{N> N_0}
     \sum_{\substack{N_1 \sim N_3\gg N_2\vee N_4,\\ N_1^{2/3}< M}}
     N^{2\theta}
     I_{\I,M,1}\\
    &\les T^{1/4}
     \sum_{\substack{N> N_0,\\ N_1 \sim N_3\gg N_2\vee N_4}}
     N_1^{2\theta+1/4}
     \|P_N^2 P_{N_1}\check{w}\|_{L_t^\I L_x^2}
     \|P_{N_3}\check{u}_2\|_{L_t^\I L_x^2}
     \|P_{N_2}\check{w}\|_{L_t^\I L_x^4}
     \|P_{N_4}\check{u}_1\|_{L_{t,x}^\I}\\
    &\les T^{1/4} \|\check{u}_1\|_{L_t^\I H_x^s}
    \|\check{u}_2\|_{L_t^\I H_x^s} \|\check{w}\|_{L_t^\I H_x^\theta}^2
    \les T^{1/4}K^2 \|w\|_{L_T^\I H_x^\theta}^2.
  }
  Here, in the last inequality, we used \eqref{eq2.1single}.
  Notice that we can estimate $I_{\I,M,2}$ by the same way.
  As for $I_{\I,M,3}$, Lemmas \ref{eq_res1} and \ref{resonance} show that
  \EQQS{
    I_{\I,M,3}
    &\le |I_{\I,M}(P_N^2 P_{N_1}\ov{Q_{\gts L}(\1_{t,R}^{\text{low}}\check{w})},
    \p_x P_{N_3}\1_{t,R}^{\text{low}}\check{u}_2,
    P_{N_2}\bar{\check{w}},P_{N_4}\check{u}_1)|\\
    &\quad +|I_{\I,M}(P_N^2 P_{N_1}\ov{Q_{\ll L}(\1_{t,R}^{\text{low}}\check{w})},
    \p_x P_{N_3}Q_{\gts L}(\1_{t,R}^{\text{low}}\check{u}_2),
    P_{N_2}\bar{\check{w}},P_{N_4}\check{u}_1)|\\
    &\quad +|I_{\I,M}(P_N^2 P_{N_1}\ov{Q_{\ll L}(\1_{t,R}^{\text{low}}\check{w})},
    \p_x P_{N_3}Q_{\ll L}(\1_{t,R}^{\text{low}}\check{u}_2),
    P_{N_2}\ov{Q_{\gts L}\check{w}},P_{N_4}\check{u}_1)|\\
    &\quad +|I_{\I,M}(P_N^2 P_{N_1}\ov{Q_{\ll L}(\1_{t,R}^{\text{low}}\check{w})},
    \p_x P_{N_3}Q_{\ll L}(\1_{t,R}^{\text{low}}\check{u}_2),
    P_{N_2}\ov{Q_{\ll L}\check{w}},P_{N_4}Q_{\gts L}\check{u}_1)|\\
    &=:I_{\I,M,3,1}+I_{\I,M,3,2}+I_{\I,M,3,3}+I_{\I,M,3,4}
  }
  with $L=N_1M$.
  It then follows that
  \EQQS{
    &\sum_{N> N_0}
     \sum_{\substack{N_1 \sim N_3\gg N_2\vee N_4,\\ N_1^{2/3}< M}}
     N^{2\theta}
     |I_{\I,M,3,1}|\\
    &\les
     \sum_{\substack{N_1 \sim N_3\gg N_2\vee N_4,\\ N_1^{2/3}< M}}
     N_1^{2\theta+1} M^{1/4}
     \|P_{N_1}
      Q_{\gts L}(\1_{t,R}^{\text{low}}\check{w})\|_{L_{t,x}^2}
     \|P_{N_3}\1_{t,R}^{\text{low}}\check{u}_2\|_{L_{t,x}^2}
     \|P_{N_2}\check{w}P_{N_4}\check{u}_1\|_{L_t^\I L_x^4}\\
    &\les T^{1/2} \sum_{\substack{N_1 \sim N_3\gg N_2\vee N_4,\\ N_1^{2/3}< M}}
    N_1^{2\theta} M^{-3/4}
    \|P_{N_1}\check{w}\|_{X^{0,1}}
    \|P_{N_3}\check{u}_2\|_{L_t^\I L_x^2}
    \|P_{N_2}\check{w}\|_{L_t^\I L_x^4}
    \|P_{N_4}\check{u}_1\|_{L_{t,x}^\I}\\
    &\les T^{1/2}K^2\|\check{w}\|_{X^{\theta-1,1}}
     \|\check{w}\|_{L_t^\I H_x^\theta}
    \les T^{1/2}K^2 \|w\|_{Z_T^\theta}
     \|w\|_{L_T^\I H_x^\theta}.
  }
  Similarly, by \eqref{eq2.1}, we obtain
  \EQQS{
    &\sum_{N> N_0}
     \sum_{\substack{N_1 \sim N_3\gg N_2\vee N_4,\\ N_1^{2/3}< M}}
     N^{2\theta}
     (|I_{\I,M,3,2}|+|I_{\I,M,3,3}|+|I_{\I,M,3,4}|)\\
    &\les T^{1/2}K^2(\|\check{w}\|_{X^{\theta-1,1}}
      +\|\check{w}\|_{L_t^\I H_x^\theta})
      \|\check{w}\|_{L_t^\I H_x^\theta}
    \les T^{1/2}K^2 \|w\|_{Z_T^\theta}
     \|w\|_{L_T^\I H_x^\theta}.
  }
  Finally, it is left to treat the contribution of $J_t^{(2,2,2)}$.
  For that purpose, we take the modified energy into account.
  First we will observe that the highest terms coming from the time derivative of a correction term indeed cancels out $J_t^{(2,2,2)}$.
  Using the equation \eqref{eq2}, we have
  \EQQS{
    &\int_0^t\frac{d}{dt'}\E_N^1(u_1(t'),u_2(t'))dt'\\
    &=\sum\RE i\int_0^t\int_\R \p_x^{-1}P_N^2 P_{N_1}\p_t\bar{w}\p_x P_{N_3}u_2
      \p_x^{-1}\dP_M (P_{N_2}\bar{w}P_{N_4}u_1)dxdt'\\
    &\quad+\sum\RE i\int_0^t\int_\R \p_x^{-1}P_N^2 P_{N_1}\bar{w}\p_x P_{N_3}\p_t u_2
      \p_x^{-1}\dP_M (P_{N_2}\bar{w}P_{N_4}u_1)dxdt'\\
    &\quad+\sum\RE i\int_0^t\int_\R \p_x^{-1}P_N^2 P_{N_1}\bar{w}\p_x P_{N_3}u_2
      \p_x^{-1}\dP_M (P_{N_2}\p_t \bar{w}P_{N_4}u_1)dxdt'\\
    &\quad+\sum\RE i\int_0^t\int_\R \p_x^{-1}P_N^2 P_{N_1}\bar{w}\p_x P_{N_3}u_2
      \p_x^{-1}\dP_M (P_{N_2}\bar{w}P_{N_4}\p_t u_1)dxdt'\\
    &=\sum\RE \int_0^t\int_\R \p_x P_N^2 P_{N_1}\p_x^2 \bar{w}\p_x P_{N_3}u_2
      \p_x^{-1}\dP_M (P_{N_2}\bar{w}P_{N_4}u_1)dxdt'\\
    &\quad-\sum\RE\int_0^t\int_\R \p_x^{-1}P_N^2 P_{N_1}\bar{w}\p_x^3 P_{N_3}u_2
      \p_x^{-1}\dP_M (P_{N_2}\bar{w}P_{N_4}u_1)dxdt'\\
    &\quad+\sum\RE\int_0^t\int_\R \p_x^{-1}P_N^2 P_{N_1}\bar{w}\p_x P_{N_3}u_2
      \p_x^{-1}\dP_M (P_{N_2}\p_x^2 \bar{w}P_{N_4}u_1)dxdt'\\
    &\quad-\sum\RE\int_0^t\int_\R \p_x^{-1}P_N^2 P_{N_1}\bar{w}\p_x P_{N_3}u_2
      \p_x^{-1}\dP_M (P_{N_2}\bar{w}P_{N_4}\p_x^2 u_1)dxdt'\\
    &\quad-\sum\RE i \int_0^t\int_\R \p_x^{-1}P_N^2 P_{N_1}\ov{W}\p_x P_{N_3}u_2
      \p_x^{-1}\dP_M (P_{N_2}\bar{w}P_{N_4}u_1)dxdt'\\
    &\quad+\sum\RE i \int_0^t\int_\R \p_x^{-1}P_N^2 P_{N_1}\bar{w}\p_x P_{N_3} V
      \p_x^{-1}\dP_M (P_{N_2}\bar{w}P_{N_4}u_1)dxdt'\\
    &\quad-\sum\RE i \int_0^t\int_\R \p_x^{-1}P_N^2 P_{N_1}\bar{w}\p_x P_{N_3}u_2
      \p_x^{-1}\dP_M (P_{N_2}\ov{W}P_{N_4}u_1)dxdt'\\
    &\quad+\sum\RE i\int_0^t \int_\R \p_x^{-1}P_N^2 P_{N_1}\bar{w}\p_x P_{N_3}u_2
      \p_x^{-1}\dP_M (P_{N_2}\bar{w}P_{N_4} U)dxdt'\\
    &=:\sum_{j=1}^8 A_j,
  }
  where
  \EQS{\label{def_WUV}
    W=F(u_1,\psi)-F(u_2,\psi),\quad
    U=F(u_1,\psi)-\Psi,\quad
    V=F(u_2,\psi)-\Psi.
  }
  Here, we denoted $\sum_{N_1\sim N_3\gg N_2\vee N_4}\sum_{N_1^{-1}\le  M\le N_1^{2/3}}$ by $\sum$ for simplicity.
  As for $W$, we use the following re-formulation of \eqref{eq_w}:
  \EQS{\label{eq_w2}
    \begin{aligned}
      \p_t w
      &=i\p_x^2 w
       +\la(w\bar{u}_1\p_x u_1 + u_2\bar{w}\p_x u_1 + |u_2|^2 \p_x w
       +w\bar{z}\p_x \psi+2\RE(w\bar{\psi})\p_x u_1\\
      &\quad\quad\quad+2\RE(u_2\bar{\psi})\p_x w
       +2\RE(w\bar{\psi})\p_x \psi+|\psi|^2\p_x w)\\
      &\quad+\mu(wz\p_x \bar{u}_1 + u_2^2 \p_x \bar{w}
       +wz\p_x \bar{\psi}
       +2w\psi\p_x \bar{u}_1
       +2u_2\psi \p_x \bar{w}\\
      &\quad\quad\quad
       +2w\psi\p_x \bar{\psi}+\psi^2 \p_x \bar{w})\\
      &=:i\p_x^2 w + \la \sum_{j=1}^8 \tilde{h}_j(u_1,u_2)
       + \mu \sum_{j=9}^{15} \tilde{h}_j(u_1,u_2).
    \end{aligned}
  }
  This formulation plays an important role in the cancellation between $A_{5,3}$ and $A_{6,1}$ which are defined below.
  By integration by parts, we obtain
  \EQQS{
    A_2
    &=-\sum\RE\int_0^t \int_\R \p_x P_{N_3}u_2 \p_x^2(\p_x^{-1}P_N^2 P_{N_1}\bar{w}
      \p_x^{-1}\dP_M (P_{N_2}\bar{w}P_{N_4}u_1))dxdt'\\
    &=-\sum\RE\int_0^t \int_\R \p_x P_N^2 P_{N_1}\bar{w}\p_x P_{N_3}u_2
      \p_x^{-1}\dP_M (P_{N_2}\bar{w}P_{N_4}u_1)dxdt'\\
    &\quad-2 \sum\RE\int_0^t \int_\R P_N^2 P_{N_1}\bar{w}\p_x P_{N_3}u_2
      \dP_M (P_{N_2}\bar{w}P_{N_4}u_1)dxdt'\\
    &\quad-\sum\RE\int_0^t \int_\R \p_x^{-1} P_N^2 P_{N_1}\bar{w}\p_x P_{N_3}u_2
      \p_x \dP_M (P_{N_2}\bar{w}P_{N_4}u_1)dxdt'.
  }
  We also notice that
  \EQS{\label{eq_anti}
    \p_x^{-1}(fg)=f \p_x^{-1}g -\p_x^{-1}(\p_x f \p_x^{-1}g),
  }
  which implies that
  \EQQS{
    \p_x^{-1}(P_{N_2}\p_x^2 \bar{w}P_{N_4}u_1)-\p_x^{-1}(P_{N_2}\bar{w}P_{N_4}\p_x^2 u_1)
    =\p_x P_{N_2}w P_{N_4}u_1 - P_{N_2}w \p_x P_{N_4}u_1.
  }
  This together with $A_2$ shows that
  \EQS{\label{eq_exchange}
    \begin{aligned}
      &A_1+A_2+A_3+A_4\\
      &=-2 \sum\RE\int_0^t\int_\R P_N^2 P_{N_1}\bar{w}\p_x P_{N_3}u_2
        \dP_M (P_{N_2}\bar{w}P_{N_4}u_1)dxdt'\\
      &\quad-\sum\RE\int_0^t\int_\R \p_x^{-1} P_N^2 P_{N_1}\bar{w}\p_x P_{N_3}u_2
        \p_x\dP_M(P_{N_2}\bar{w}P_{N_4}u_1)dxdt'\\
      &\quad+\sum\RE\int_0^t\int_\R \p_x^{-1} P_N^2 P_{N_1}\bar{w}\p_x P_{N_3}u_2
        \p_x^{-1}\dP_M(P_{N_2}\p_x^2 \bar{w}P_{N_4}u_1-P_{N_2}\bar{w}P_{N_4}\p_x^2 u_1)dxdt'\\
      &=-2 \sum\RE\int_0^t\int_\R P_N^2 P_{N_1}\bar{w}\p_x P_{N_3}u_2
        \dP_M (P_{N_2}\bar{w} P_{N_4}u_1)dxdt'\\
      &\quad-2\sum\RE\int_0^t\int_\R \p_x^{-1} P_N^2 P_{N_1}\bar{w}\p_x P_{N_3}u_2
        \dP_M(P_{N_2}\bar{w} \p_x P_{N_4} u_1)dxdt',
    \end{aligned}
  }
  where $\sum$ denotes $\sum_{N_1\sim N_3\gg N_2\vee N_4}\sum_{N_1^{-1}\le  M\le N_1^{2/3}}$.
  Note that this is indeed the decomposition announced after the definition of the modified energy: the contribution of  the linear part of the equation in the derivative of the modified energy consists of the bad term we aim to cancel, plus a remaining term that is obtained by applying the Fourier multiplier by $ \xi_4 \xi_1^{-1} $ to this bad term.
  Then, by choosing a constant $c_1$ in \eqref{def_E2} appropriately, the first term of $A_1+A_2+A_3+A_4$ coincides with $J_t^{(2,2,2)}$.
  On the other hand, the second term in $A_1+A_2+A_3+A_4$ can be treated in a similar manner to $J_t^{(2,2,1)}$ since we can write
  \EQQS{
    &\sum_{N>N_0}\sum_{\substack{N_1\sim N_3\gg N_2\vee N_4,\\
      N_1^{-1}\le M\le N_1^{2/3}}} N^{2\theta}
      \RE\int_0^t\int_\R \p_x^{-1} P_N^2 P_{N_1}\bar{w}\p_x P_{N_3}u_2
        \dP_M(P_{N_2}\bar{w}\p_xP_{N_4} u_1)dxdt'\\
    &\le \sum_{N>N_0}\sum_{\substack{N_1\sim N_3\gg N_2\vee N_4,\\
      N_1^{-1}\le M\le N_1^{2/3}}} N^{2\theta}
      |I_{t,M}(\p_x^{-1}P_N^2 P_{N_1}\bar{w}, \p_x P_{N_3}u_2,P_{N_2}\bar{w},\p_x P_{N_4}u_1)|\\
    &\les T^{1/4}K^2 \|w\|_{Z_T^\theta}\|w\|_{L_T^\I H_x^\theta}.
  }
  Next, we consider the nonlinear contributions of $A_5,A_6,A_7$ and $A_8$.
  For $A_5$, we use the equation \eqref{eq_w2}, i.e.,
  we set
  \EQQS{
    A_5
    &=\sum_{j=1}^{15}
     \sum_{\substack{N_1\sim N_3\gg N_2\vee N_4,\\
      N_1^{-1}\le M\le N_1^{2/3}}}
      \RE i\int_0^t\int_\R \p_x^{-1}P_N^2 P_{N_1}\bar{\tilde{h}}_j
       \p_x P_{N_3}u_2
      \p_x^{-1}\dP_M (P_{N_2}\bar{w}P_{N_4}u_1)dxdt'\\
    &=:\sum_{j=1}^{15}A_{5,j}.
  }
  We also set
  \EQQS{
    A_6
    &=\sum_{j=1}^{10}\sum_{\substack{N_1\sim N_3\gg N_2\vee N_4,\\
      N_1^{-1}\le M\le N_1^{2/3}}}
      \RE i \int_0^t\int_\R \p_x^{-1}P_N^2 P_{N_1}\bar{w}\p_x P_{N_3} f_j(u_2)
      \p_x^{-1}\dP_M (P_{N_2}\bar{w}P_{N_4}u_1)dxdt'\\
    &=:\sum_{j=1}^{10}A_{6,j},
  }
  where $f_j(u_2)$ is defined in \eqref{def_fu}.
  Then, we see the cancellation between  $A_{5,3}$ and $A_{6,1}$.
  First we decompose $|u_2|^2 \p_x w = P_{\ll N}(|u_2|^2) \p_x w +P_{\gts N}(|u_2|^2) \p_x w $ and $|u_2|^2 \p_x u_2 = P_{\ll N}(|u_2|^2) \p_x u_2 +P_{\gts N}(|u_2|^2) \p_x u_2 $.
  The dangerous contributions are the ones of $ P_{\ll N} (|u_2|^2) $ since for the other ones we can ``recuperate some derivatives" by making use of the regularity of $u_2$.
  The trick is that, up to terms that contain the commutator  $[P_N,P_{\ll N} (|u_2|^2) ] $ which enables to exchange one high derivative with a low one (that is sufficient for us!), the contribution of
  $ P_{\ll N} (|u_2|^2) $ is of the form:
  \EQQS{
    &\Big(P_N^2 P_{N_1} \bar{w} \partial_x P_{N_3} u_2
    + \p_x^{-1} P_N^2 P_{N_1} \bar{w} \partial_x^2 P_{N_3} u_2\Big) P_{\ll N} g(u_1,u_2,w)\\
    &= \partial_x\Big(\partial_x^{-1}P_N^2 P_{N_1} \bar{w} \partial_x P_{N_3}u_2\Big)P_{\ll N} g(u_1,u_2,w),
  }
  so that by integration by parts we may also exchange a high derivative with a low one.
  More precisely, we have
  \EQS{\label{eq_6.1}
    \begin{aligned}
      &A_{5,3}+A_{6,1}\\
      &=\sum\RE \int_0^t\int_\R \p_x^{-1}P_N^2 P_{N_1}(P_{\ll N}(|u_2|^2)\p_x \bar{w})\p_x P_{N_3}u_2
        \p_x^{-1}\dP_M (P_{N_2}\bar{w}P_{N_4}u_1)dxdt'\\
      &\quad+\sum\RE i\int_0^t\int_\R \p_x^{-1}
       P_N^2 P_{N_1}
       (P_{\gts N}(|u_2|^2)\p_x \bar{w})\p_x P_{N_3}u_2
        \p_x^{-1}\dP_M (P_{N_2}\bar{w}P_{N_4}u_1)dxdt'\\
      &\quad+\sum\RE i \int_0^t\int_\R \p_x^{-1}P_N^2 P_{N_1} \bar{w}\p_x P_{N_3}(P_{\ll N}(|u_2|^2)\p_x u_2)
        \p_x^{-1}\dP_M (P_{N_2}\bar{w}P_{N_4}u_1)dxdt'\\
      &\quad+\sum\RE i \int_0^t\int_\R \p_x^{-1}P_N^2 P_{N_1} \bar{w}\p_x P_{N_3}(P_{\gts N}(|u_2|^2)\p_x u_2)
        \p_x^{-1}\dP_M (P_{N_2}\bar{w}P_{N_4}u_1)dxdt'\\
      &=\sum\RE i \int_0^t\int_\R \p_x^{-1}P_N^2 P_{N_1}(P_{\gts N}(|u_2|^2)\p_x \bar{w})\p_x P_{N_3}u_2
        \p_x^{-1}\dP_M (P_{N_2}\bar{w}P_{N_4}u_1)dxdt'\\
      &\quad+\sum\RE i \int_0^t\int_\R \p_x^{-1}P_N^2 P_{N_1} \bar{w}\p_x P_{N_3}(P_{\gts N}(|u_2|^2)\p_x u_2)
        \p_x^{-1}\dP_M (P_{N_2}\bar{w}P_{N_4}u_1)dxdt'\\
      &\quad+\sum\RE i\int_0^t\int_\R \p_x^{-1}([P_N^2 P_{N_1},P_{\ll N}(|u_2|^2)]\p_x \bar{w})\p_x P_{N_3}u_2
        \p_x^{-1}\dP_M (P_{N_2}\bar{w}P_{N_4}u_1)dxdt'\\
      &\quad+\sum\RE i \int_0^t\int_\R \p_x^{-1}P_N^2 P_{N_1} \bar{w}\p_x ([P_{N_3},P_{\ll N}(|u_2|^2)]\p_x u_2)
        \p_x^{-1}\dP_M (P_{N_2}\bar{w}P_{N_4}u_1)dxdt'\\
      &\quad+\sum\RE i\int_0^t\int_\R \p_x^{-1}(P_{\ll N}(|u_2|^2)\p_x P_N^2 P_{N_1} \bar{w})\p_x P_{N_3}u_2
        \p_x^{-1}\dP_M (P_{N_2}\bar{w}P_{N_4}u_1)dxdt'\\
      &\quad+\sum\RE i \int_0^t\int_\R \p_x^{-1}P_N^2 P_{N_1} \bar{w}\p_x (P_{\ll N}(|u_2|^2)\p_x P_{N_3}u_2)
        \p_x^{-1}\dP_M (P_{N_2}\bar{w}P_{N_4}u_1)dxdt',\\
      &=:\sum_{l=1}^{6} A_{5,3,l}
    \end{aligned}
  }
  where we denoted $\sum_{N_1\sim N_3\gg N_2\vee N_4}\sum_{N_1^{-1}\le  M\le N_1^{2/3}}$ by $\sum$.
  In estimating $A_j$, we frequently use the following estimate:
  \EQS{\label{eq_6.6}
    \sum_{\substack{N_2,N_4\ll N_1,\\ N_1^{-1}\le M\le N_1^{2/3}}}
    \|\p_x^{-1}\dP_M (P_{N_2}\bar{w} P_{N_4}u_1)\|_{L_x^\I}
    \les N_1^{0+} \|u_1\|_{H_x^s}
    \|w\|_{H_x^\theta}.
  }
  Indeed, by the Bernstein inequality,
  \EQQS{
    \sum_{\substack{N_2,N_4\ll N_1,\\ N_1^{-1}\le M\le N_1^{2/3}}}
    \|\p_x^{-1}\dP_M (P_{N_2}\bar{w} P_{N_4}u_1)\|_{L_x^\I}
    &\les \sum_{\substack{N_2,N_4\ll N_1,\\ N_1^{-1}\le M\le N_1^{2/3}}}
     \|P_{N_2}w\|_{L_x^2}\|P_{N_4}u_1\|_{L_x^2}\\
    &\les (\log N_1) \|u_1\|_{H_x^s}
    \|w\|_{H_x^\theta},
  }
  which implies \eqref{eq_6.6}.
  Notice that $A_{5,3,1},A_{5,3,2},A_{5,3,3}$ and $A_{5,3,4}$ in the right hand side of \eqref{eq_6.1} are not dangerous.
  Indeed, we see from Proposition \ref{prop_stri1} and \eqref{eq_6.6} that
  \EQS{\label{eq_6.9}
    \begin{aligned}
      &\sum_{N> N_0}N^{2\theta}|A_{5,3,1}|\\
      &\les \sum_{N> N_0}\sum_{\substack{N_1\sim N_3\gg N_2\vee N_4,\\N_1^{-1}\le  M\le N_1^{2/3}}}\sum_{\ti{N}\gts N} N^{2\theta}
      \|P_{N_3}u_2\|_{L_T^1 L_x^\I}
      \|\p_x^{-1}\dP_M (P_{N_2}\bar{w} P_{N_4}u_1)\|_{L_{T,x}^\I}\\
      &\quad\quad \times(\|P_N^2 P_{N_1}(P_{\ti{N}}(|u_2|^2)\p_x P_{\sim \ti{N}} \bar{w})\|_{L_T^\I L_x^1}
      +\|P_N^2 P_{N_1}(P_{\ti{N}}(|u_2|^2)\p_x P_{\ll N} \bar{w})\|_{L_T^\I L_x^1})\\
      &\les K \|w\|_{L_T^\I \ov{H}^{\theta,\de}}\|w\|_{L_T^\I H_x^{1/4}}
       \sum_{N> N_0}\sum_{\ti{N}\gts N} N^{2\theta+\de}
       \ti{N}^{3/4}\|P_{\ti{N}}(|u_2|^2)\|_{L_T^\I L_x^2}
       \|P_{\sim N}u_2\|_{L_T^1 L_x^\I}\\
      &\les K T^{3/4} \|w\|_{L_T^\I \ov{H}^{\theta,\de}}
       \||u_2|^2\|_{L_T^\I H_x^s}
       \|w\|_{L_T^\I H_x^\theta}\|J_x^{2\te}u_2\|_{L_T^4 L_x^{\I}}\\
      &\les C(K)T^{3/4}\|w\|_{L_T^\I \ov{H}^{\theta,\de}}\|w\|_{L_T^\I H_x^\theta}.
    \end{aligned}
  }
  Here, we used that $ 2\theta +\frac{1}{4}<s$ and we took $ \delta>0 $ such that $ \frac{3}{4}+\delta<s $. We can get the same result for the term containing $P_{\gts N}(|u_2|^2)\p_x u_2$ in the right hand side of \eqref{eq_6.1} as above.
  On the other hand, for the term containing $[P_N^2 P_{N_1},P_{\ll N}(|u_2|^2)]\p_x \bar{w}$ in the right hand side of \eqref{eq_6.1}, we use Proposition \ref{prop_stri1}, \eqref{eq2.6} and \eqref{eq_comm3}:
  \EQS{
    \begin{aligned}\label{eq_6.7}
      &\|\p_x^{-1}([P_N^2 P_{N_1},P_{\ll N}(|u_2|^2)]\p_x \bar{w})\p_x P_{N_3}u_2\|_{L_{T,x}^1}\\
      &\les \|[P_N^2 P_{N_1},P_{\ll N}(|u_2|^2)]\p_x \bar{w}\|_{L_T^1 L_x^2}\|P_{N_3}u_2\|_{L_T^\I L_x^2}\\
      &\les \|\p_x P_{\ll N}(|u_2|^2)\|_{L_T^1 L_x^\I}
      \|P_{\sim N} w\|_{L_T^\I L_x^2}
      \|P_{N_3}u_2\|_{L_T^\I L_x^2}\\
      &\les N^{1/2}\|D_x^{1/2+1/p}(|u_2|^2)\|_{L_T^1 L_x^p}
      \|P_{\sim N} w\|_{L_T^\I L_x^2}
      \|P_{N_3}u_2\|_{L_T^\I L_x^2}\\
      &\les N^{1/2}\|u_2\|_{L_{T,x}^\I}\|D_x^{1/2+1/p}u_2\|_{L_T^1 L_x^p}
      \|P_{\sim N} w\|_{L_T^\I L_x^2}
      \|P_{N_3}u_2\|_{L_T^\I L_x^2}\\
      &\les C(K)T^{3/4}N^{1/2}\|P_{\sim N} w\|_{L_T^\I L_x^2}
      \|P_{N_3}u_2\|_{L_T^\I L_x^2},
    \end{aligned}
  }
  where we chose $1<p<\I$ so that $3/4+1/2p<s$.
  We can get the same result for the term containing $[P_{N_3},P_{\ll N}(|u_2|^2)]\p_x u_2$ in the right hand side of \eqref{eq_6.1} as above.
  For the last two terms in the right hand side of \eqref{eq_6.1}, they look dangerous, but they cancel out each other.
  We see from \eqref{eq_6.1} and the integration by parts that
  \EQS{\label{eq_6.2}
    \begin{aligned}
      &A_{5,3,5}+A_{5,3,6}\\
      &=\sum\RE i \int_0^t \int_\R \p_x^{-1}(P_{\ll N}(|u_2|^2)\p_x P_N^2 P_{N_1} \bar{w})\p_x P_{N_3}u_2
        \p_x^{-1}\dP_M (P_{N_2}\bar{w}P_{N_4}u_1)dxdt'\\
      &\quad-\sum\RE i \int_0^t \int_\R P_N^2 P_{N_1} \bar{w} P_{\ll N}(|u_2|^2)\p_x P_{N_3} u_2
        \p_x^{-1}\dP_M (P_{N_2}\bar{w}P_{N_4}u_1)dxdt'\\
      &\quad-\sum\RE i \int_0^t \int_\R \p_x^{-1} P_N^2 P_{N_1} \bar{w} P_{\ll N}(|u_2|^2)\p_x P_{N_3} u_2
        \dP_M (P_{N_2}\bar{w}P_{N_4}u_1)dxdt'\\
      &=-\sum\RE i\int_0^t\int_\R \p_x^{-1}(\p_x P_{\ll N}(|u_2|^2) P_N^2 P_{N_1} \bar{w})\p_x P_{N_3}u_2
        \p_x^{-1}\dP_M (P_{N_2}\bar{w}P_{N_4}u_1)dxdt'\\
      &\quad-\sum\RE i\int_0^t \int_\R \p_x^{-1} P_N^2 P_{N_1} \bar{w} P_{\ll N}(|u_2|^2)\p_x P_{N_3} u_2
        \dP_M (P_{N_2}\bar{w}P_{N_4}u_1)dxdt',
    \end{aligned}
  }
  where $\sum$ denotes $\sum_{N_1\sim N_3\gg N_2\vee N_4}\sum_{N_1^{-1}\le  M\le N_1^{2/3}}$.
  We see from \eqref{eq_6.7} that these two terms are indeed closed.
  Now we consider contributions of $A_{5,j}$ except for $A_{5,3}$.

  \noindent
  \underline{Estimate for $A_{5,1}$.}
  Thanks to \eqref{eq_6.6}, it suffices to show that
  \EQS{\label{eq_6.8}
    \|\p_x^{-1}P_N^2 P_{N_1}(w\bar{u}_1\p_x u_1)\p_x P_{N_3}u_2\|_{L_{T,x}^1}
    \les N^{-2\theta-\e}C(K)T^{3/4}
    \|w\|_{L_T^\I H_x^\theta}
  }
  for sufficiently small $\e>0$.
  By Proposition \ref{prop_stri1} and \eqref{eq2.6}, we have
  \EQS{
    \begin{aligned}\label{eq_6.14}
      &\|\p_x^{-1}P_N^2 P_{N_1}(\bar{w}u_1\p_x \bar{u}_1)\p_x P_{N_3}u_2\|_{L_{T,x}^1}\\
      &\le \|\p_x^{-1}P_N^2 P_{N_1}(P_{\ll N}(\bar{w}u_1)\p_x \bar{u}_1)\p_x P_{N_3}u_2\|_{L_{T,x}^1}\\
      &\quad +\|\p_x^{-1}P_N^2 P_{N_1}(P_{\gts N}(\bar{w}u_1)\p_x P_{\ll N} \bar{u}_1)\p_x P_{N_3}u_2\|_{L_{T,x}^1}\\
      &\quad+\|\p_x^{-1}P_N^2 P_{N_1}(P_{\gts N}(\bar{w}u_1)\p_x P_{\gts N} \bar{u}_1)\p_x P_{N_3}u_2\|_{L_{T,x}^1}\\
      &\les N\|P_{\ll N}(\bar{w}u_1)\|_{L_T^1 L_x^\I}
       \|P_{\sim N}u_1\|_{L_T^\I L_x^2}
       \|P_{N_3}u_2\|_{L_T^\I L_x^2}\\
      &\quad+N^{1-2s}\|P_{\gts N}(\bar{w}u_1)\|_{L_T^1 L_x^\I}
        \|P_{\ll N}u_1\|_{L_T^\I H_x^s}
        \|P_{N_3}u_2\|_{L_T^\I H_x^s}\\
      &\quad +\sum_{\ti{N}\gts N}\ti{N}^{3/4-s}
       \|P_{\ti{N}}(\bar{w}u_1)\|_{L_T^\I H_x^{1/4}}
       \|P_{\sim \ti{N}}u_1\|_{L_T^\I H_x^s}
      \|P_{N_3}u_2\|_{L_T^1 L_x^\I}\\
      &\les C(K) N^{1-2s}T^{3/4}\|w\|_{L_T^\I H_x^\theta}
      \les  N^{-2\theta-}C(K)T^{3/4} \|w\|_{L_T^\I H_x^\theta},
    \end{aligned}
  }
  where we used $s>2\theta+1/4$ and $\theta>1/4$.
  Here, we applied Proposition \ref{prop_stri1} to $u_2$:
  \EQQS{
    \|P_{N_3}u_2\|_{L_T^1 L_x^\I}
    \les N_3^{1/4-s}T^{3/4}C(K).
  }

  \noindent
  \underline{Estimate for $A_{5,2}$.}
  This is the same as $A_{5,1}$.

  \noindent
  \underline{Estimate for $A_{5,4}$.}
  This case is easy to treat since we have
  \EQQS{
    \|\p_x^{-1}P_N^2 P_{N_1}(\bar{w}z \p_x \bar{\psi})\p_x P_{N_3} u_2\|_{L_x^1}
    &\les \|\bar{w}z \p_x \bar{\psi}\|_{L_x^2}
     \|P_{N_3} u_2\|_{L_x^2}\\
    &\les \|w\|_{L_x^2}\|z\|_{L_x^2}\|\p_x \psi\|_{L_x^\I}
     \|P_{N_3} u_2\|_{L_x^2}.
  }

  \noindent
  \underline{Estimate for $A_{5,5}$.}
  This case is similar to \eqref{eq_6.8}.

  \noindent
  \underline{Estimate for $A_{5,6}$.}
  In order to close the estimate for $A_{5,6}$, we have to observe a cancellation with $A_{6,3}$ as in \eqref{eq_6.1}.
  We can follow the argument in $A_{5,3}$ with replacing $|u_2|^2$ with $2\RE(u_2\bar{\psi})$, and we place the $L^\I$-norm on $\psi$ instead of the $L^2$-norm.

  \noindent
  \underline{Estimate for $A_{5,7}$.}
  As in $A_{5,4}$, we obtain
  \EQQS{
    \|\p_x^{-1}P_N^2 P_{N_1}(\RE(w\bar{\psi})\p_x \bar{\psi})\p_x P_{N_3}u_2\|_{L_x^1}
    &\les \|P_N^2 P_{N_1}(\RE(w\bar{\psi})\p_x \bar{\psi})\|_{L_x^2}
    \|P_{N_3}u_2\|_{L_x^2}\\
    &\les \|w\|_{L_x^2}\|\psi\|_{L_x^\I}\|\p_x \psi\|_{L_x^\I}
    \|P_{N_3}u_2\|_{L_x^2},
  }
  which is enough to close.

  \noindent
  \underline{Estimate for $A_{5,8}$.}
  In order to close the estimate for $A_{5,8}$, we need to exploit a cancellation with $A_{6,5}$ as in \eqref{eq_6.1}.
  The estimates are similar to those of $A_{5,3}$.
  We use \eqref{eq2.7} instead of \eqref{eq2.6} to $|\psi|^2$ since $\psi$ does not belong to $L^2$.

  \noindent
  \underline{Estimate for $A_{5,9}$.}
  We see from the proof of \eqref{eq_6.8} that
  \EQQS{
    \|\p_x^{-1}P_N^2 P_{N_1}(\bar{w}z\p_x \bar{u}_1)
     \p_x P_{N_3}u_2\|_{L_{T,x}^1}
    \les N^{-2\theta-\e}C(K)T^{3/4}
    \|w\|_{L_T^\I H_x^\theta},
  }
  which is enough to close.

  \noindent
  \underline{Estimate for $A_{5,10}$.}
  As in $A_{5,3}$, the most dangerous part is $P_{\ll N}(\bar{u}_2^2)\p_x P_N^2 P_{N_1} w$.
  In particular, it suffices to consider the interaction $P_{\ll N}((P_{\ll N}\bar{u}_2)^2)\p_x P_N^2 P_{N_1} w$
  since the interaction $P_{\ll N}((P_{\gts N}\bar{u}_2)^2)\p_x P_N^2 P_{N_1} w$ can be estimated by the same way as \eqref{eq_6.9}.
  Other interactions are estimated by the same way as in $A_{5,3}$.
  On the one hand, we need to observe the cancellation in $A_{5,3}$ and $A_{5,8}$ in order to close estimates.
  On the other hand, we can use Bourgain type estimates for the interaction
  $P_{\ll N}((P_{\ll N} u_2)^2)\p_x P_N^2 P_{N_1} w$
  thanks to the strong resonance relation.
  More precisely, the resonance relation of
  \EQQS{
    \int_0^t\int_\R \p_x^{-1}(P_{\ll N}(P_{\ll N}\bar{u}_2)^2 \p_x P_N^2 P_{N_1}w)\p_x P_{N_3}u_2
    \p_x^{-1}\dP_M (P_{N_2}\bar{w}P_{N_4}u_1)dxdt'
  }
  is the following:
  \EQQS{
    |\Om^{(2)}(\xi_1,\dots, \xi_6)|
    =|-\xi_1^2-\xi_2^2+\xi_3^2+\xi_4^2\pm\xi_5^2\mp\xi_6^2|
    \gts N_1^2,
  }
  since $|\xi_1|,|\xi_2|,|\xi_5|,|\xi_6|\ll N$ and $|\xi_3|,|\xi_4|\sim N$.
  We take extensions of $\check{u}_1=\rho_T(u_1)$, $\check{u}_2=\rho_T(u_2)$ and $\check{w}=\rho_T(w)$ defined in \eqref{def_ext}.
  We drop the check in the sequel to simplify the notation.
  With a slight abuse of notation, we set
  \EQQS{
    &I^{(6)}_{t,M}(\al_1,\al_2,\al_3,\al_4,\al_5,\al_6)\\
    &:=\RE \int_0^t\int_\R \p_x^{-1}
     (P_{\ll N}(P_{\ll N}\al_1 P_{\ll N}\al_2)
      P_N^2 P_{N_1}\al_3)
      P_{N_3}\al_4
    \p_x^{-1}\dP_M (P_{N_2}\al_5 P_{N_4}\al_6)dxdt'.
  }
  So, as in Lemma \ref{lem_Bourgain1}, we have
  \EQQS{
    &\RE\int_0^t\int_\R \p_x^{-1}(P_{\ll N}(P_{\ll N}\bar{u}_2)^2 \p_x P_N^2 P_{N_1}w)\p_x P_{N_3}u_2
    \p_x^{-1}\dP_M (P_{N_2}\bar{w} P_{N_4}u_1)dxdt'\\
    &=I^{(6)}_{\I,M}(\bar{u}_2,\bar{u}_2,
    \1_{t,R}^{\text{high}}\p_xw,
      \1_t \p_x u_2,\bar{w},u_1)
    +I^{(6)}_{\I,M}(\bar{u}_2,\bar{u}_2,
    \1_{t,R}^{\text{low}}\p_x w,
    \1_{t,R}^{\text{high}}\p_x u_2,\bar{w},u_1)\\
    &\quad+I^{(6)}_{\I,M}(\bar{u}_2,\bar{u}_2,
    \1_{t,R}^{\text{low}}\p_x w,
    \1_{t,R}^{\text{low}}\p_x u_2,\bar{w},u_1),
  }
  where $R=N_1^{5/3}$.
  The first two terms in the right hand side can be estimated easily if we use \eqref{eq4.1}.
  So, we focus on the third term in the right hand side.
  Putting $L:=N_1^2\gg R$, by Lemma \ref{resonance}, we have six terms as follows:
  \EQQS{
    &I^{(6)}_{\I,M}(\bar{u}_2,\bar{u}_2,
      \1_{t,R}^{\text{low}}\p_x w,
      \1_{t,R}^{\text{low}}\p_x u_2,\bar{w},u_1)\\
    &=I^{(6)}_{\I,M}(\bar{u}_2,\bar{u}_2,
      Q_{\gts L}(\1_{t,R}^{\text{low}}\p_x w),
      \1_{t,R}^{\text{low}}\p_x u_2,\bar{w},u_1)\\
    &\quad+I^{(6)}_{\I,M}(\bar{u}_2,\bar{u}_2,
      Q_{\ll L}(\1_{t,R}^{\text{low}}\p_x w),
      Q_{\gts L}(\1_{t,R}^{\text{low}}\p_x u_2),\bar{w},u_1)+\cdots\\
    &\quad+I^{(6)}_{\I,M}(\ov{Q_{\ll L}u_2},\ov{Q_{\ll L}u_2},
      Q_{\ll L}(\1_{t,R}^{\text{low}}\p_x w),
      Q_{\ll L}(\1_{t,R}^{\text{low}}\p_x u_2),\ov{Q_{\ll L}w},Q_{\gts L}u_1).
  }
  It then follows from the Bernstein inequality that
  \EQS{\label{tutu}
    \begin{aligned}
      &\sum_{N>N_0}
       \sum_{\substack{N_1\sim N_3\gg N_2\vee N_4,\\ N_1^{-1}\le  M\le N_1^{2/3}}} N^{2\theta}
       |I^{(6)}_{\I,M}(\bar{u}_2,\bar{u}_2,
         Q_{\gts L}(\1_{t,R}^{\text{low}}\p_x w),
         \1_{t,R}^{\text{low}}\p_x u_2,\bar{w},u_1)|\\
      &\les \sum_{N>N_0}\sum_{\substack{N_1\sim N_3\gg N_2\vee N_4,\\ N_1^{-1}\le  M\le N_1^{2/3}}}
       N^{2\theta+1}
       \|P_{\ll N}u_2\|_{L_{t,x}^\I}^2
       \|P_N^2 P_{N_1}Q_{\gts L}(\1_{t,R}^{\text{low}} w)\|_{L_{t,x}^2}\\
      &\quad\times \|P_{N_3}\1_{t,R}^{\text{low}}u_2\|_{L_{t,x}^2}
      \|P_{N_2}w\|_{L_t^\I L_x^2}
        \|P_{N_4}u_1\|_{L_t^\I L_x^2}\\
      &\les T^{1/2}K^3 \|w\|_{L_t^\I H_x^\theta}
       \sum_{N_1\sim N_3\gts N_0}N_1^{2\theta-1+}\|P_{N_1}w\|_{X^{0,1}}
       \|P_{N_3}u_2\|_{L_t^\I L_x^2}\\
        &\les T^{1/2}K^3 \|w\|_{L_t^\I H_x^\theta}
       \sum_{N_1\sim N_3\gts N_0}N_1^{\theta-s+}\|P_{N_1}w\|_{X^{\theta-1,1}}
       \|P_{N_3}u_2\|_{L_t^\I H_x^s}\\
      &\les K^4 T^{1/2}\|w\|_{L_t^\I H_x^\theta}
       \|w\|_{X^{\theta-1,1}}.
    \end{aligned}
  }
  Other terms can be estimated in a similar way.

  \noindent
  \underline{Estimate for $A_{5,11}$.}
  It is easy to see that $A_{5,11}$ can be estimated by the same way as $A_{5,4}$.

  \noindent
  \underline{Estimate for $A_{5,12}$.}
  Notice that $A_{5,12}$ can be estimated by the same way as $A_{5,5}$.

  \noindent
  \underline{Estimate for $A_{5,13}$.}
  Recall that in $A_{5,6}$ we have to observe the cancellation for the interaction $(\RE P_{\ll N}(u_2\bar{\psi}))\p_x P_N^2 P_{N_1}\bar{w}$.
  So, the most dangerous part of $A_{5,13}$ is the interaction $(\RE P_{\ll N}(\bar{u}_2\bar{\psi}))\p_x P_N^2 P_{N_1}w$.
  Moreover, it suffices to consider the interaction $(\RE P_{\ll N}(P_{\ll N} \bar{u}_2 P_{\ll N}\bar{\psi}))\p_x P_N^2 P_{N_1}w$
  since $(\RE P_{\gts N}(P_{\gts N} \bar{u}_2 P_{\ll N}\bar{\psi}))\p_x P_N^2 P_{N_1}w$ can be estimated easily.
  And this interaction is estimated by a similar way to that of $A_{5,10}$ with using the argument of Lemma \ref{lem_Bourgain3} instead of Lemma \ref{lem_Bourgain1}.

  \noindent
  \underline{Estimate for $A_{5,14}$.}
  This is the same as $A_{5,7}$.

  \noindent
  \underline{Estimate for $A_{5,15}$.}
  Similarly to $A_{5,13}$, we only consider the interaction
  \EQQS{
    P_{\ll N}((P_{\ll N}\bar{\psi})^2)\p_x P_N^2 P_{N_1}w.
  }
  We can still use the argument of Lemma \ref{lem_Bourgain3} to this interaction.
  Indeed, if we have
  \EQQS{
    &\int_{\R^2} \p_x^{-1}(P_{\ll N}(P_{\ll N}R_{\ll L}
     \bar{\psi}P_{\ll N}\bar{\psi})) \p_x P_N^2 P_{N_1}Q_{\ll L}(\1_{t,R}^{\text{low}}w))
     \p_x P_{N_3}Q_{\ll L}(\1_{t,R}^{\text{low}}u_2)\\
    &\quad\quad\quad\times\RE\p_x^{-1}\dP_M
    (P_{N_2}\ov{Q_{\ll L}w}P_{N_4}Q_{\ll L}u_1)dxdt',
  }
  then it automatically holds that $P_{\ll N}\bar{\psi}= P_{\ll N}R_{\gts L}\bar{\psi}$ with $L=N_1^2$ since
  \EQQS{
    |\ta_1-\xi_1^2+\ta_2-\xi_2^2+\ta_3+\xi_3^2+\ta_4+\xi_4^2+\ta_5-\xi_5^2+\ta_6-\xi_6^2|\gts L
  }
  and $|\ta_1|, |\xi_1|^2, |\xi_2|^2, |\ta_3+\xi_3^2|, |\ta_4+\xi_4^2|, |\ta_5-\xi_5^2|, |\ta_6-\xi_6^2|\ll L$.
  This finishes estimates for $A_5$.

  Estimates for $A_6$ can be obtained similarly to $A_5$.
  Next, we consider $A_7$.
  Roughly speaking, derivative do not fall on the highest frequency in $A_7$, so we can put a part of them to some functions.
  And it turns out that we do not need to use either the refined Strichartz estimate (Proposition \ref{prop_stri1}) nor Bourgain type estimates.
  Actually, since $s>\theta+\frac{1}{2} $, it is enough to prove the following:
  \EQS{\label{eq_6.12}
    \sum_{\substack{N_2,N_4\ll N_1,\\ N_1^{-1}\le M\le N_1^{2/3}}}
    \|\p_x^{-1}\dP_M (P_{N_2}\ov{W} P_{N_4}u_1)\|_{L_x^\I}
    \les N_1^{\frac{1}{2}} K^3 \|w\|_{H_x^{\theta}}.
  }
  Remark that $N_2,N_4\in 2^{\N}$.
  In what follows, we show \eqref{eq_6.12} with a slightly better bound.
  That is, we show that the loss is at most $N_1^{\frac{1}{6}}$ although $N_1^{\frac{1}{2}}$ is acceptable.
  We only consider $W=\tilde{h}_j$ for $j=9,\dots, 15$ since the complex conjugate does not have any role in the subsequent argument.
  Recall that the $\tilde{h}_j$ are defined in \eqref{eq_w2}.

  \noindent
  \underline{Estimate for $\tilde{h}_9$.}
  Notice that
  \EQS{\label{eq_6.10}
    \begin{aligned}
      P_{N_2}(wz\p_x \bar{u}_1)
      &=P_{N_2}(wz\p_x P_{\les 1}\bar{u}_1)
       +\sum_{\ti{N}_2\gg 1} P_{N_2}(P_{\ll \ti{N}_2} (wz) \p_x P_{\ti{N}_2} \bar{u}_1)\\
      &\quad+\sum_{\ti{N}_2\gg 1} P_{N_2}(P_{\gts \ti{N}_2} (wz) \p_x P_{\ti{N}_2} \bar{u}_1).
    \end{aligned}
  }
  Then it is easy to see that
  \EQS{\label{eq_6.11}
    \begin{aligned}
      &\sum_{\substack{N_2,N_4\ll N_1,\\ N_1^{-1}\le M\le N_1^{2/3}}}
      \|\p_x^{-1}\dP_M (P_{N_2}( \bar{w} \bar{z} \p_x P_{\les 1}u_1) P_{N_4}u_1)\|_{L_x^\I}\\
      &\les \sum_{\substack{N_2,N_4\ll N_1,\\ N_1^{-1}\le M\le N_1^{2/3}}}
       \|P_{N_2}(\bar{w}\bar{z}\p_x P_{\les 1}u_1) P_{N_4}u_1\|_{L_x^1}\\
      &\les N_1^{0+} \|wz\|_{L_x^2}
      \| u_1\|_{L^2_x} \|u_1\|_{L_x^\I}
      \les N_1^{0+} K^3 \|w\|_{L_x^2}.
    \end{aligned}
  }
  By the same way,
  \EQS{\label{eq_6.1111}
    \begin{aligned}
      &\sum_{\substack{N_2,N_4\ll N_1,\\ N_1^{-1}\le M\le N_1^{2/3}}}\sum_{\ti{N}_2\gg 1}
      \|\p_x^{-1}\dP_M (P_{N_2}(P_{\gts \ti{N}_2}(\bar{w} \bar{z})\p_x P_{\ti{N}_2}u_1) P_{N_4} u_1 )\|_{L_x^\I}\\
      &\les \sum_{\substack{N_2,N_4\ll N_1,\\ N_1^{-1}\le M\le N_1^{2/3}}}\sum_{\ti{N}_2\gg 1}
       \|P_{N_2}(P_{\gts \ti{N}_2}(\bar{w} \bar{z})\p_x P_{\ti{N}_2}u_1) P_{N_4}u_1\|_{L_x^1}\\
      &\les \sum_{\substack{N_2,N_4\ll N_1,\\ N_1^{-1}\le M\le N_1^{2/3}}}\sum_{\ti{N}_2\gg 1}
     \ti{N}_2^{\frac{1}{4}}  \|P_{\gts \ti{N}_2}(wz)\|_{L^2_x} \ti{N}_2^{\frac{3}{4}}\|P_{\ti{N}_2}u_1\|_{L^2_x}  \|u_1\|_{L_x^\I}
      \les N_1^{0+} K^3 \|w\|_{H_x^{\frac{1}{4}}}.
    \end{aligned}
}
Finally we notice that $P_{N_2}(P_{\ll \ti{N}_2} (wz) \p_x P_{\ti{N}_2} \bar{u}_1)$ vanishes unless
$ N_2 \sim \ti{N}_2 \gg 1 $ and we estimate its contribution by splitting the sum over $ 1\ll N_2\les M $ and
 the sum over $M\ll N_2\ll N_1 $.
 Note that $N_4\sim M$ in the first case and $ N_4\sim N_2$ in the second case.
 It then follows that
  \EQS{\label{eq_6.1c}
    \begin{aligned}
      &\sum_{\substack{1\ll N_2\ll N_1 ,N_4\ll N_1,\\ N_1^{-1}\le M\le N_1^{2/3}}}
      \|\p_x^{-1}\dP_M (P_{N_2}(P_{\ll N_2}(\bar{w} \bar{z})\p_x P_{\sim N_2}u) P_{N_4}u_1)\|_{L_x^\I}\\
      &\les  \sum_{1\ll N_2\les M\sim N_4\ll N_1}
        M^{-\frac{1}{2}} \|P_{N_2}(P_{\ll N_2}(\bar{w} \bar{z})\p_x P_{\sim N_2}u) P_{N_4}u_1\|_{L_x^2}\\
      &\quad +\sum_{\substack{M\ll N_2\sim N_4\ll N_1,\\ N_1^{-1}\le M\le N_1^{2/3}}}
        \|P_{N_2}(P_{\ll N_2}(\bar{w} \bar{z})\p_x P_{\sim N_2}u) P_{N_4}u_1\|_{L_x^1}\\
      &\les  \sum_{1\ll N_2\les M\sim N_4\ll N_1}
        M^{-\frac{1}{2}}N_2 \|wz\|_{L^2_x} \|P_{\sim N_2}u_1\|_{L_x^\infty}
        \|P_{N_4}u_1\|_{L_x^\infty}\\
      &\quad +\sum_{\substack{M\ll N_2\sim N_4\ll N_1,\\ N_1^{-1}\le M\le N_1^{2/3}}}
        \|wz\|_{L^2_x} N_2^\frac{3}{4} \| P_{\sim N_2}u_1\|_{L^2_x} N_4^\frac{1}{4} \| P_{N_4}u_1\|_{L_x^\I}
      \les \log(N_1) K^3 \|w\|_{H_x^{\frac{1}{4}}}.
    \end{aligned}
}

  \noindent
  \underline{Estimate for $\tilde{h}_{10}$.} We rewrite this term as  $u_2^2\p_x \bar{w}=\partial_x(u_2^2 \bar{w})- 2\bar{w} u_2\p_x u_2$.
  The last term clearly has the same form as $\tilde{h}_9 $ and can be treated in exactly the same way.
  To estimate the contribution of the first term we proceed similarly as in \eqref{eq_6.1c} by separating the sum over $ N_2\les 1$, $ 1\ll N_2 \les M $ and $ N_2\gg M $.
  More precisely, we have
   \EQS{\label{eq_6.1cc}
    \begin{aligned}
      &\sum_{\substack{ N_2\ll N_1 ,N_4\ll N_1,\\ N_1^{-1}\le M\le N_1^{2/3}}}
      \|\p_x^{-1}\dP_M (\p_x P_{N_2} (\bar{u}_2^2 w) P_{N_4}u_1)\|_{L_x^\I}\\
      &\les  \sum_{\substack{1\ll N_2\les M ,N_4\ll N_1,\\ N_1^{-1}\le M\le N_1^{2/3}}}
        M^{-\frac{1}{2}}  \|\p_x P_{N_2} (\bar{u}_2^2 w)  P_{N_4}u_1\|_{L_x^2}\\
      &\quad +\Bigg(\sum_{\substack{0< N_2\les 1, 1\le N_4\ll N_1,\\ N_1^{-1}\le M\le N_1^{2/3}}}
        +\sum_{\substack{M\ll N_2\sim N_4\ll N_1,\\ N_1^{-1}\le M\le N_1^{2/3}}}\Bigg)
        \|\p_x P_{N_2} (\bar{u}_2^2 w) P_{N_4}u_1\|_{L_x^1}\\
      &\les
       \sum_{\substack{N_4\ll N_1,\\ N_1^{-1}\le M\le N_1^{2/3}}}
       (M^{\frac{1}{4}}\|u_2^2 \bar{w}\|_{H_x^\frac{1}{4}} \|P_{N_4}u_1\|_{L_x^\infty}
       +\|u_2^2 \bar{w}\|_{L_x^2} \|P_{N_4}u_1\|_{L_x^2})\\
      &\quad+\sum_{\substack{ N_1^{-1}\le M\le N_1^{2/3}}}
       \|u_2^2\bar{w}\|_{H_x^\frac{1}{4}} \|u_1\|_{H_x^{\frac{3}{4}}}
      \les N_1^\frac{1}{6}K^3 \|w\|_{H_x^{\frac{1}{4}}}.
    \end{aligned}
}
  \noindent
  \underline{Estimate for $\tilde{h}_{11}$.}
  We can easily close the estimate in this case since the derivative is landing on $\psi$.
  Indeed, since $\p_x \psi\in L^\I(\R;L^p(\R)) $ for any $p\in[1,\I]$, we have
  \EQQS{
    &\sum_{\substack{N_2,N_4\ll N_1,\\N_1^{-1}\le M\le N_1^{2/3}}}
    \|\p_x^{-1}\dP_M (P_{N_2}(\bar{w}\bar{z}\p_x \psi)) P_{N_4}u_1)\|_{L_x^\I}\\
    &\les (\log N_1)\sum_{N_2,N_4\ll N_1}
     \|P_{N_2}(\bar{w}\bar{z}\p_x \psi)\|_{L_x^1}
     \|P_{N_4}u_1\|_{L_x^\I}
    \les N_1^{0+} K^3\|w\|_{L_x^2}.
  }

  \noindent
  \underline{Estimate for $\tilde{h}_{12}$.}
  First we decompose $\tilde{h}_{12}=2w\psi\p_x \bar{u}_1$ as in \eqref{eq_6.10}:
  \EQQS{
    P_{N_2}(w\psi\p_x \bar{u}_1)
    &=P_{N_2}(w\psi\p_x P_{\les 1}\bar{u}_1)
     +\sum_{\ti{N}_2\gg1}P_{N_2}(P_{\ll \ti{N}_2}(w\psi)\p_x P_{\ti{N}_2}\bar{u}_1)\\
    &\quad+\sum_{\ti{N}_2\gg1}P_{N_2}(P_{\gts\ti{N}_2}(w\psi)
     \p_x P_{\ti{N}_2} \bar{u}_1)
    =:\tilde{h}_{12,1}+\tilde{h}_{12,2}+\tilde{h}_{12,3}.
  }
  By \eqref{eq_6.11},  \eqref{eq_6.1111}, and \eqref{eq_6.1c}, we can estimate $\tilde{h}_{12,1}, \tilde{h}_{12,2}$, and $\tilde{h}_{12,3}$, respectively.

  \noindent
  \underline{Estimate for $\tilde{h}_{13}$.}
  Notice that $\tilde{h}_{13}=2u_2\psi\p_x \bar{w}=2\p_x(u_2\psi \bar{w})-2\p_x u_2\psi \bar{w}-2 u_2\p_x \psi \bar{w}$.
  The second (resp. the third) term can be treated by the same way as $\tilde{h}_{12}$ (resp. $\tilde{h}_{11}$).
  On the other hand, we can apply the argument of $\tilde{h}_{10}$ to the first term thanks to $\|u_2\psi\bar{w}\|_{H_x^{\frac{1}{4}}}\les \|u_2\|_{H_x^{\frac{3}{4}+}}
  \|J_x^{\frac{1}{4}}\psi\|_{L_x^{\I}}
  \|w\|_{H_x^{\frac{1}{4}}}$, which is a direct consequence of \eqref{eq2.6}.

  \noindent
  \underline{Estimate for $\tilde{h}_{14}$.}
  This term is easily estimated as follows:
  \EQQS{
    &\sum_{\substack{ N_2, N_4\ll N_1,\\N_1^{-1}\le M\le N_1^{2/3}}}
    \|\p_x^{-1}\dP_M (P_{N_2}(\bar{w}\bar{\psi}\p_x \psi) P_{N_4}u_1)\|_{L_x^\I} \\
    &\les \sum_{\substack{ N_2, N_4\ll N_1,\\N_1^{-1}\le M\le N_1^{2/3}}}
    \|\dP_M (P_{N_2}(\bar{w}\bar{\psi}\p_x \psi) P_{N_4}u_1)\|_{L_x^1}
    \les N_1^{0+} K^3 \|w\|_{L_x^2}
  }
  since $J_x^{s+1+\e}\psi \in L^\I(\R^2)$.

  \noindent
  \underline{Estimate for $\tilde{h}_{15}$.} We rewrite this term as $\psi^2 \p_x \bar{w}= \p_x(\psi^2  \bar{w})-2\bar{w} \psi \p_x \psi $ and notice that the second term is exactly $ \tilde{h}_{14} $ whereas the contribution of the first term can be easily be treated
  as in \eqref{eq_6.1cc}.
  Indeed, by \eqref{eq2.6}, we have $\|\bar{\psi}^2 w\|_{H_x^{\frac{1}{4}}}\les
  \|J_x^{\frac{1}{4}}\psi\|_{L_x^{\I}}^2
  \|w\|_{H_x^{\frac{1}{4}}}$, which implies that
  \EQQS{
    &\sum_{\substack{ N_2,N_4\ll N_1,\\ N_1^{-1}\le M\le N_1^{2/3}}}
    \|\p_x^{-1}\dP_M (P_{N_2}\p_x (\bar{\psi}^2 w) P_{N_4}u_1)\|_{L_x^\I}
    \les N_1^\frac{1}{6}K^3 \|w\|_{H_x^{\frac{1}{4}}}.
}
  Finally we consider estimates for $A_8$.
  We would like to show the following:
  \EQQS{
   \sum_{ \substack{N_2,N_4\ll N_1,\\N_1^{-1}\le M\le N_1^{2/3}}}
    \|\p_x^{-1}\dP_M (P_{N_2}\bar{w} P_{N_4}U)\|_{L_x^\I}
    \les N_1^{1/2} K^3 \|w\|_{H_x^{\theta}}.
  }
  Again, we remark that $N_2,N_4\in 2^{\N}$.
  For that purpose, it suffices to consider $U=f_j(u_1)$ for $j=1,\dots,5$ since other estimates follows similarly and the complex conjugate does not have any role, where $f_j(u_1)$ is defined in \eqref{def_fu}.
  For $f_1(u)$, we decompose $f_1(u_1)=|u_1|^2 \p_x u_1$ as in \eqref{eq_6.10}, and then we apply the argument of $\tilde{h}_{9}$.
  Following the argument of $\tilde{h}_{11}$, we can estimate $f_2(u_1)=|u_1|^2 \p_x \psi $ easily.
  Notice that $f_3(u_1)=2\p_x u_1\RE(u_1\bar{\psi})$ can be estimated by the same way as in $\tilde{h}_{12}$.
  We also note that the estimate for $f_4(u_1)=2\p_x \psi \RE(u_1\bar{\psi})$ follows from the argument of $\tilde{h}_{14}$.
  The argument of $\tilde{h}_{15}$ can be applied to $f_5(u_1)$.
  In conclusion, the estimates for $\E_N^1(u_1,u_2)$ is completed.

  \noindent
  \textbf{Estimate for $J_t^{(3)}$.}
  This case follows from the estimates for $J_t^{(2)}$.

  \noindent
  \textbf{Estimate for $J_t^{(6)}$.}
  We estimate the following:
  \EQQS{
    J_t^{(6)}
    = N^{2\theta}
     \sum_{N_1,\dots,N_4\in 2^\N}
     \RE \int_0^t \int_\R
     P_N^2 P_{N_1} \bar{w}
     \p_x P_{N_3} u_2
     \RE( P_{N_2} w P_{N_4} \bar{\psi}) dxdt'.
  }
  We see from the analysis in the estimates for $J_t^{(10)}$ that $N_1\sim N_3\gg N_2\vee N_4$ is the most delicate case.
  Furthermore, by using the decomposition $P_{N_2} w P_{N_4} \bar{\psi}=\sum_{M\les N_2\vee N_4}\dP_M(P_{N_2} w P_{N_4} \bar{\psi})$ as in Case 4 in the estimates for $J_t^{(2)}$, the most difficult interaction is $N_1^{-1}\le M\le N_1^{2/3}$.
  Therefore, we are reduced to focusing on the following (which we still denote by $J_t^{(6)}$):
  \EQQS{
    J_t^{(6)}
    =N^{2\theta}
     \sum_{\substack{N_1\sim N_3\gg N_2\vee N_4,\\ N_1^{-1}\le M\le N_1^{2/3}}}
     \RE \int_0^t \int_\R
     P_N^2 P_{N_1} \bar{w}
     \p_x P_{N_3} u_2
     \RE \dP_M( P_{N_2} w P_{N_4} \bar{\psi}) dxdt',
  }
  where $N_1,\dots, N_4\in 2^\N $ and $ M\in 2^\Z$.
  In what follows, we show that this term is indeed canceled out by the correction term $\E_N^3(u_1,u_2)$.
  To see this, we first take the time derivative of $\E_N^3(u_1,u_2)$:
  \EQQS{
    &\int_0^t \frac{d}{dt}\E_N^3(u_1(t'),u_2(t'))dt'\\
    &=\sum\RE i \int_0^t \int_\R \p_x^{-1}P_N^2 P_{N_1} \p_t \bar{w}\p_x P_{N_3}u_2
     \RE\p_x^{-1}\dP_M(P_{N_2}w P_{N_4}\bar{\psi})dxdt'\\
     &\quad+\sum\RE i \int_0^t \int_\R \p_x^{-1}P_N^2 P_{N_1} \bar{w}
      \p_x P_{N_3}\p_t u_2
      \p_x^{-1} \RE\p_x^{-1}
      \dP_M(P_{N_2}w P_{N_4}\bar{\psi})dxdt'\\
     &\quad+\sum\RE i \int_0^t \int_\R \p_x^{-1}P_N^2 P_{N_1} \bar{w}\p_x P_{N_3}u_2
     \RE\p_x^{-1}\dP_M(P_{N_2}\p_t w P_{N_4}\bar{\psi})dxdt'\\
     &\quad+\sum\RE i \int_0^t \int_\R \p_x^{-1}P_N^2 P_{N_1} \bar{w}\p_x P_{N_3}u_2
      \RE\p_x^{-1}\dP_M(P_{N_2}w P_{N_4}\p_t \bar{\psi})dxdt'\\
    &=\sum\RE \int_0^t \int_\R \p_x P_N^2 P_{N_1} \bar{w}\p_x P_{N_3}u_2
     \RE\p_x^{-1}\dP_M(P_{N_2}w P_{N_4}\bar{\psi})dxdt'\\
     &\quad-\sum\RE\int_0^t \int_\R \p_x^{-1}P_N^2 P_{N_1} \bar{w}
      \p_x^3 P_{N_3} u_2
      \RE\p_x^{-1}
      \dP_M(P_{N_2}w P_{N_4}\bar{\psi})dxdt'\\
     &\quad-\sum\RE\int_0^t \int_\R \p_x^{-1}P_N^2 P_{N_1} \bar{w}\p_x P_{N_3}u_2
      \IM \p_x^{-1}\dP_M(P_{N_2}\p_x^2 w P_{N_4}\bar{\psi})dxdt'\\
     &\quad+\sum\RE i \int_0^t \int_\R \p_x^{-1}P_N^2 P_{N_1} \bar{w}\p_x P_{N_3}u_2
      \RE\p_x^{-1}\dP_M(P_{N_2}w P_{N_4}\p_t \bar{\psi})dxdt'\\
     &\quad-\sum\RE i \int_0^t \int_\R \p_x^{-1}P_N^2 P_{N_1} \ov{W}
      \p_x P_{N_3}u_2
      \RE\p_x^{-1}\dP_M(P_{N_2}w P_{N_4}\bar{\psi})dxdt'\\
     &\quad+\sum\RE i \int_0^t \int_\R \p_x^{-1}
       P_N^2 P_{N_1} \bar{w}\p_x P_{N_3}V
       \RE\p_x^{-1}\dP_M(P_{N_2}w P_{N_4}\bar{\psi})dxdt'\\
     &\quad+\sum\RE i \int_0^t \int_\R \p_x^{-1}P_N^2 P_{N_1} \bar{w}\p_x P_{N_3}u_2
     \RE
     \p_x^{-1}\dP_M(P_{N_2} W P_{N_4}\bar{\psi})dxdt'\\
    &=:\sum_{j=1}^7 A_j,
  }
  where $\sum$ denotes $\sum_{N_1\sim N_3\gg N_2\vee N_4}
   \sum_{N_1^{-1}\le M\le N_1^{2/3}}$ and $W,U$ and $V$ are defined in \eqref{def_WUV}.
  As in the estimates for $J_t^{(2)}$, first we need to check that $A_1+A_2+A_3+A_4$ can cancel out $J_{t}^{(6)}$ if we choose a constant $c_3\in\R$ appropriately.
  By integration by parts, we have
  \EQQS{
    A_2
    &=-\sum\RE \int_0^t \int_\R \p_x P_N^2 P_{N_1} \bar{w}
     \p_x P_{N_3} u_2
     \RE\p_x^{-1}
     \dP_M(P_{N_2}w P_{N_4}\bar{\psi})dxdt'\\
    &\quad-2\sum\RE\int_0^t \int_\R P_N^2 P_{N_1} \bar{w}
     \p_x P_{N_3} u_2
     \RE \dP_M(P_{N_2}w P_{N_4}\bar{\psi})dxdt'\\
    &\quad-\sum\RE\int_0^t \int_\R \p_x^{-1}P_N^2 P_{N_1} \bar{w}
     \p_x P_{N_3} u_2
     \RE\p_x \dP_M(P_{N_2}w P_{N_4}\bar{\psi})dxdt'
  }
  where $\sum$ denotes $\sum_{N_1\sim N_3\gg N_2\vee N_4}
   \sum_{N_1^{-1}\le M\le N_1^{2/3}}$.
  As for $A_4$, we notice that
  \EQS{\label{eq_psi}
   \begin{aligned}
     &\int_\R \p_x^{-1}P_N^2 P_{N_1} \bar{w}\p_x P_{N_3}u_2
     \RE\p_x^{-1}\dP_M(P_{N_2}w P_{N_4}\p_t \bar{\psi})dx\\
     &=\int_\R \p_x^{-1}P_N^2 P_{N_1} \bar{w}\p_x P_{N_3}u_2
     \IM \p_x^{-1}\dP_M(P_{N_2}w P_{N_4}\p_x^2 \bar{\psi})dx\\
     &\quad+\int_\R \p_x^{-1}P_N^2 P_{N_1} \bar{w}\p_x P_{N_3}u_2
     \RE\p_x^{-1}\dP_M(P_{N_2}w P_{N_4}(\p_t+i\p_x^2)\bar{\psi})dx.
   \end{aligned}
  }
  Moreover, as for $A_3$ and $A_4$, we also see from \eqref{eq_anti} that
  \EQQS{
    &-\p_x^{-1}\dP_M(\p_x^2P_{N_2}w P_{N_4}\bar{\psi})
    +\p_x^{-1}\dP_M(P_{N_2}w \p_x^2 P_{N_4}\bar{\psi})\\
    &=-\dP_M(\p_x P_{N_2}w P_{N_4}\bar{\psi})
     +\p_x^{-1}\dP_M(\p_x P_{N_2}w \p_x P_{N_4}\bar{\psi})\\
    &\quad+\dP_M(P_{N_2}w \p_x P_{N_4}\bar{\psi})
     -\p_x^{-1}\dP_M(\p_x P_{N_2}w \p_x P_{N_4}\bar{\psi})\\
    &=-\dP_M(\p_x P_{N_2}w P_{N_4}\bar{\psi})
     +\dP_M(P_{N_2}w \p_x P_{N_4}\bar{\psi}).
  }
  Notice that there are no terms having derivatives more than one, which allows us to close estimates under the assumption \eqref{hyp_psi2}.
  This observation will be used in the estimates of $A_1+\cdots+A_5$.
  By \eqref{eq_6.6} with a slight modification, we see from Remark \ref{rem_psi} that
  \EQQS{
    \sum_{\substack{N_2,N_4\ll N_1,\\ N_1^{-1}\le M\le N_1^{2/3}}}
    \|\p_x^{-1}\dP_M (P_{N_2}\bar{w} P_{N_4}(\p_t+i\p_x^2)\bar{\psi})\|_{L_x^\I}
    \les N_1^{0+} \|(i\p_t+\p_x^2)\psi\|_{L_x^2}
    \|w\|_{L_x^2}.
  }
  So, by \eqref{eq_anti} and \eqref{eq_psi}, $A_1+A_2+A_3+A_4$ takes the form
  \EQQS{
    &\sum_{n=1}^4 A_n\\
    &=-2\sum\RE\int_0^t \int_\R P_N^2 P_{N_1} \bar{w}
     \p_x P_{N_3} u_2
     \RE \dP_M(P_{N_2}w P_{N_4}\bar{\psi})dxdt'\\
    &\quad-\sum\RE\int_0^t \int_\R \p_x^{-1}P_N^2 P_{N_1} \bar{w}
     \p_x P_{N_3} u_2
     \RE\p_x \dP_M(P_{N_2}w P_{N_4}\bar{\psi})dxdt'\\
    &\quad-\sum\RE\int_0^t \int_\R \p_x^{-1}P_N^2 P_{N_1} \bar{w}
     \p_x P_{N_3} u_2
     \IM
     \dP_M(\p_x P_{N_2}w P_{N_4}\bar{\psi})dxdt'\\
    &\quad+\sum\RE\int_0^t \int_\R \p_x^{-1}P_N^2 P_{N_1} \bar{w}
     \p_x P_{N_3} u_2
     \IM
     \dP_M(P_{N_2}w \p_x P_{N_4}\bar{\psi})dxdt'\\
    &\quad+\sum\RE i\int_0^t  \int_\R \p_x^{-1}P_N^2 P_{N_1} \bar{w}\p_x P_{N_3}u_2
     \RE\p_x^{-1}\dP_M(P_{N_2}w P_{N_4}(\p_t+i\p_x^2)\bar{\psi})dxdt'\\
    &=:A_{1,1}+\cdots+A_{1,5},
  }
  where $\sum$ denotes $\sum_{N_1\sim N_3\gg N_2\vee N_4}
   \sum_{N_1^{-1}\le M\le N_1^{2/3}}$.
  Notice that $A_{1,1}$ coincides with $J_t^{(6)}$ up to a constant.
  In this step, we choose $c_3\in\R$ such that $A_{1,1}$ eliminates $J_{t}^{(6)}$.
  As for $A_{1,2}, A_{1,4}$ and $A_{1,5}$, it is easy to have
  \EQQS{
    \sum_{N>N_0}N^{2\theta}(|A_{1,2}|+|A_{1,4}|+|A_{1,5}|)
    \les K^2 T\|w\|_{L_T^\I H_x^\theta}^2.
  }
  In the case $N_4\gts N_2$ in $A_{1,3}$, we can distribute some derivatives to $\psi$.
  On the other hand, when $N_4\ll N_2$, we can use the Bourgain with the resonance relation $|\Om|\gts NN_2$, which is enough to close.
  The advantage of introducing $\E_N^3$ is that we can ``downgrade" a derivative $N$ to $N_2$ or $N_4$.
  Next, we observe the cancellation such as in the estimates for $J_t^{(2)}$.
  We consider $A_5$ and $A_6$.
  We set
  \EQQS{
    &A_5\\
    &=-\sum_{j=1}^{15}\sum_{\substack{N_1\sim N_3\gg N_2\vee N_4,\\ N_1^{-1}\le M\le N_1^{2/3}}}
     \RE i \int_0^t \int_\R \p_x^{-1}P_N^2 P_{N_1} \bar{\tilde{h}}_j
     \p_x P_{N_3}u_2
     \RE\p_x^{-1}\dP_M(P_{N_2}w P_{N_4}\bar{\psi})dxdt'\\
    &=:\sum_{j=1}^{15}A_{5,j}
  }
  and
  \EQQS{
    &A_6\\
    &=\sum_{j=1}^{10}\sum_{\substack{N_1\sim N_3\gg N_2\vee N_4,\\ N_1^{-1}\le M\le N_1^{2/3}}}
    \RE i \int_0^t\int_\R \p_x^{-1}
      P_N^2 P_{N_1} \bar{w}\p_x P_{N_3}f_j(u_2)
      \RE\p_x^{-1}\dP_M(P_{N_2}w P_{N_4}\bar{\psi})dxdt'\\
    &=:\sum_{j=1}^{10}A_{6,j},
  }
  where $f_j$ and $\tilde{h}_j$ are defined in \eqref{def_fu} and \eqref{eq_w2}, respectively.
  Here, we remark that we use the equation \eqref{eq_w2} not \eqref{eq_w} for $A_5$ so as to have $|u_2|^2\p_x w$.
  As in the estimates for $J_t^{(2)}$, we can see that  there is a cancellation between the  worst terms of $A_{5,3}$ and $A_{6,1}$.
  Similarly, so do $(A_{5,6},A_{6,3})$ and $(A_{5,9},A_{6,5})$.
  In the estimate of $A_{5,1}$ in $J_t^{(2)}$, the contribution of $\tilde{h}_1=w\bar{u}_1\p_x u_1$ is not a dangerous term since \eqref{eq_6.6} is available.
  As in \eqref{eq_6.6} we aim to establish an estimate of the type
  \EQQS{
  \sum_{N_1^{-1}\le M\le N_1^{2/3}}\Big\|\sum_{N_2,N_4\ll N_1}\p_x^{-1}\dP_M(P_{N_2}w P_{N_4}\bar{\psi})\Big\|_{L_x^\I} \les K N_1^{0+} \|w\|_{H_x^\theta} \, .
  }
  However, in the current case, if we merely follow the proof of \eqref{eq_6.6}, we can only obtain an estimate on the contribution of $ P_{N_4}\psi $ for $ M\ge 1 $.
  Indeed, it is easy to check that
  \EQS{\label{trf}
    \begin{aligned}
      &\sum_{1\le M\le N_1^{2/3}}
       \Big\|\sum_{N_2,N_4\ll N_1}\p_x^{-1}\dP_M(P_{N_2}w P_{N_4}\bar{\psi})\Big\|_{L_x^\I}\\
      & \les \sum_{1\le M\le N_1^{2/3}} \sum_{N_2,N_4\ll N_1}  M^{-\frac{1}{2}} \|P_{N_2}w P_{N_4}\bar{\psi}\|_{L_x^2}
       \les N_1^{0+} K \|w\|_{L_x^2} \, .
    \end{aligned}
  }
 To tackle the contribution of $ 0<M<1 $  we have to use  homogenous projections on $w$ and $ \psi $ together with the weight on low frequency for $ w$ and the fact that $ \p_x \psi(t) \in L^1(\R) $ for $t\in\R$ a.e.
 More precisely, we write
    \EQQS{
    &\sum_{N_1^{-1}\le M< 1}
     \Big\|\sum_{\substack{N_2,N_4\ll N_1}}\p_x^{-1}\dP_M(\dP_{N_2}w
     \dP_{N_4}\bar{\psi})\Big\|_{L_x^\I}\\
    &\les \sum_{\substack{0<N_2\les M,\\N_1^{-1}\le M< 1}}
     M^{-\frac{1}{2}} \|\dP_M(\dP_{N_2}w
     \dP_{\les M}\bar{\psi})\|_{L_x^2}
    +\sum_{\substack{M\ll N_2 \sim  N_4\ll N_1,\\N_1^{-1}\le M< 1}}
     \|\dP_M(\dP_{N_2}w \dP_{N_4}\bar{\psi})\|_{L_x^1}\\
    &=:R_1+R_2.
  }
  As for $R_1$, by the definition of the norm $\|\cdot\|_{\ov{H}^{\theta,\delta}}$, we have
  \EQQS{
    R_1
    &\le \sum_{\substack{0<N_2\les M,\\N_1^{-1}\le M< 1}}
     M^{-\frac{1}{2}}N_2^{\frac{1}{2}-\de}
     \|\dP_{N_2}w\|_{\ov{H}^{0,\de}} \|\dP_{\les M}\psi \|_{L_x^\I}
    \les N_1^\de K \|w\|_{\ov{H}^{0,\de}} \; .
  }
  On the other hand, we estimate $R_2$ as follows:
  \EQQS{
    R_2
    &\le \sum_{\substack{M\ll N_2 \sim  N_4\ll N_1,\\N_1^{-1}\le M< 1}}
       \|\dP_{N_2}w\|_{L^2_x} \|\dP_{N_4}\bar{\psi}\|_{L_x^2}\\
    &\les \sum_{\substack{M\ll N_2 \sim  N_4\ll N_1,\\N_1^{-1}\le M< 1}}
       (1\vee N_2^{-\de}) \|\dP_{N_2}w\|_{\ov{H}^{0,\de}} \|\p_x \dP_{N_4} \psi\|_{L_x^1}
    \les N_1^\de K \|w\|_{\ov{H}^{0,\de}},
  }
  which together with \eqref{trf} implies that
  \EQS{\label{trf2}
    \sum_{N_1^{-1}\le M< N_1^{2/3}}
     \Big\|\sum_{\substack{N_2,N_4\ll N_1}}\p_x^{-1}\dP_M(P_{N_2}w
      P_{N_4}\bar{\psi})\Big\|_{L_x^\I}
    \les N_1^\de K \|w\|_{\ov{H}^{0,\de}}.
  }
  Decomposing $\tilde{h}_1=w\bar{u}_1\p_x u_1$ we have
  \EQQS{
    &A_{5,1}\\
    &=-\sum \RE i \int_0^t \int_\R \p_x^{-1}P_N^2 P_{N_1}(P_{\ll N}(\bar{w}u_1)\p_x P_{\sim N}\bar{u}_1)
    \p_x P_{N_3}u_2
    \RE\p_x^{-1}\dP_M(P_{N_2}w P_{N_4}\bar{\psi})dxdt'\\
    &\quad-\sum \RE i \int_0^t \int_\R
     \p_x^{-1}P_N^2 P_{N_1}(P_{\gts N}(\bar{w}u_1)
     \p_x P_{\ll N}\bar{u}_1)
     \p_x P_{N_3}u_2
    \RE\p_x^{-1}\dP_M(P_{N_2}w P_{N_4}\bar{\psi})dxdt'\\
    &\quad-\sum \RE i \int_0^t \int_\R
     \p_x^{-1}P_N^2 P_{N_1}(P_{\gts N}(\bar{w}u_1)
     \p_x P_{\gts N}\bar{u}_1)
     \p_x P_{N_3}u_2
    \RE\p_x^{-1}\dP_M(P_{N_2}w P_{N_4}\bar{\psi})dxdt'\\
    &=:A_{5,1,1}+A_{5,1,2}+A_{5,1,3},
  }
  where $\sum$ denotes $\sum_{N_1\sim N_3\gg N_2\vee N_4}\sum_{N_1^{-1}\le M\le N_1^{2/3}}$.
  These three terms can be treated in a similar way to the estimates for $J_t^{(2)}$.
  We have to use \eqref{trf2} with \eqref{eq2.6} and Proposition \ref{prop_stri1}.
  For instance we have
  \EQQS{
    &\sum_{N>N_0}N^{2\theta}|A_{5,1,3}|\\
    &\les \sum_{N>N_0}\sum_{\substack{N_1\sim N_3\\ N^{-1}\les M\les N^{2/3}}}
    \sum_{\ti{N}_1\gts N}
    N^{2\theta}\ti{N}_1
    \|P_{\ti{N}_1}(\bar{w}u_1)\|_{L_{T,x}^2}
    \|P_{\sim \ti{N}_1}u_1\|_{L_T^\I L_x^2}
    \|P_{N_3}u_2\|_{L_T^2 L_x^\I}\\
    &\quad\quad\times
     \Big\|\sum_{N_2\vee N_4\ll N_1}\p_x^{-1}\dP_M(P_{N_2}w P_{N_4}\bar{\psi})\Big\|_{L_{T,x}^\I}\\
    &\les K \|w\|_{L_T^\I \ov{H}^{\theta,\de}}
     \sum_{\ti{N}_1\gts N_3\gts N_0}
     N_3^{2\theta+\de}\ti{N}_1
     \|P_{\ti{N}_1}(w\bar{u}_1)\|_{L_{T,x}^2}
     \|P_{\sim \ti{N}_1}u_1\|_{L_T^\I L_x^2}
     \|P_{N_3}u_2\|_{L_T^2 L_x^\I}\\
    &\les K \|w\|_{L_T^\I \ov{H}^{\theta,\de}}
     \|J_x^{2\te}u_2\|_{L_T^2 L_x^{\I}}
     \sum_{\ti{N}_1\gts N_0}\tilde{N}_1^{1+\de}
     \|P_{\ti{N}_1}(w\bar{u}_1)\|_{L_{T,x}^2}
     \|P_{\sim\tilde{N}_1}u_1\|_{L_T^\I L_x^2}\\
    &\les C(K)\|w\|_{L_T^\I \ov{H}^{\theta,\de}}
    \|w\|_{L_T^2 H_x^{1/4}}\|u_1\|_{L_T^\I H_x^s}
    \les C(K)T^{1/2}\|w\|_{L_T^\I \ov{H}^{\theta,\de}}
     \|w\|_{L_T^\I H_x^\theta}.
  }
  Here, in the fourth inequality, we used $2\theta<s-1/4$ and $3/4+\de<s$.
  Terms $A_{5,1,1}$ and $A_{5,1,2}$ can be estimated similarly to \eqref{eq_6.14}.
  We can follow the argument of the estimates for $J_t^{(2)}$ so that we obtain closed estimates for $W=\tilde{h}_j$, $j\in\{1,\dots,15\}\backslash\{10,13,15\}$.
  When $j=10,13,15$, we have to use Bourgain type estimates on some interactions to recover the derivative loss as we did in $A_{5,10}$ of the estimates for $J_t^{(2)}$.
  In particular, for those interactions, we have a strong resonant relation $|\Om|\gts N_1^2$.
  For instance, for $W=\tilde{h}_{10}=u_2^2 \p_x \bar{w}$ in $A_5$, we first decompose it as
  \EQQS{
    P_N (u_2^2 \p_x \bar{w})
    &= P_N(P_{\ll  N}(u_2^2) \p_x \bar{w})
     +P_N(P_{\gts  N}(u_2^2) \p_x \bar{w})\\
    &=P_N(P_{\ll  N}((P_{\gts N}u_2)^2) \p_x P_{\sim N} \bar{w})
     +P_N(P_{\ll  N}((P_{\ll N}u_2)^2)\p_x P_{\sim N} \bar{w})\\
    &\quad+P_N(P_{\gts  N}(u_2^2) \p_x \bar{w}).
  }
  We can treat the first term and the third term by the same way as \eqref{eq_6.9} with using \eqref{trf} instead of \eqref{eq_6.6}.
  Now for the second term, we have to use Bourgain-type estimates. Actually we can apply the argument of $A_{5,10}$ in the estimates for $J_t^{(2)}$ (see \eqref{tutu}).
  Note that we have a factor $ N_1^{\theta-s+} \le N_1^{-(\frac{1}{2}+)} $ in the middle of \eqref{tutu} since $s/2-1/8<s-1/2 $ for $s>3/4$ and we took $ \theta<s/2-1/8$.
  This additional gain allows us to avoid using \eqref{trf2} when utilizing Bourgain-type estimates, which is convenient since we need to take extensions of each function such as $\check{w}=\rho_T(w)$.
  In this case, we replace \eqref{trf2} with the trivial estimates
  \EQS{\label{tutu2}
   \|\p_x^{-1}\dP_M(P_{N_2}w P_{N_4}\bar{\psi})\|
   \les M^{-1/2} \|P_{N_2}w\|_{L^2_x}
   \|P_{N_4}\psi\|_{L^\infty_x}
  }
  that loses a factor $N_1^{\frac{1}{2}} $ since $ M\ge N^{-1} $.
  To be precise, following the arguments used in the analysis of $A_{5,10}$ in the estimates for $J_t^{(2)}$ with \eqref{tutu2} in hand, we can obtain
  \EQQS{
   &\sum_{N>N_0}
    \sum_{\substack{N_1\sim N_3\gg N_2\vee N_4,\\N_1^{-1}\le M\le N_1^{2/3}}} N^{2\theta}
    \Bigg|\RE i \int_0^t \int_\R \p_x^{-1}P_N^2 P_{N_1}(P_{\ll  N}(P_{\ll N}\bar{u}_2)^2 \p_x P_{\sim N} w)\\
    &\quad\quad\quad\times\p_x P_{N_3}u_2
    \RE\p_x^{-1}\dP_M(P_{N_2}w P_{N_4}\bar{\psi})dxdt'\Bigg|\\
   &\les T^{1/2} K^2 (\|w\|_{X^{\theta-1,1}}+\|w\|_{L_t^\I H_x^\theta})\|w\|_{L_t^\I H_x^\theta}.
  }
  Estimates for $W=\tilde{h}_{13},\tilde{h}_{15}$ follows similarly.
  By a similar argument, we can estimate $A_6$.
  Finally, as for $A_7$, it suffices to show the following:
  \EQS{\label{eq_6.15}
    \sum_{\substack{N_2,N_4\ll N_1,\\N_1^{-1}\le M\le N_1^{2/3}}}
    \|\p_x^{-1}\dP_M(P_{N_2}WP_{N_4}\bar{\psi})\|_{L_x^\I}
    \les N_1^{1/2}K^3\|w\|_{H_x^\theta},
  }
  where $W$ is defined by \eqref{def_WUV}.
  Most of the proofs follow from ones of \eqref{eq_6.12}.
  We only describe how to estimate $W=\tilde{h}_{15}=\psi^2 \p_x \bar{w}$ because this term requires a slight modification.
  We rewrite this term as $\p_x(\psi^2 \bar{w})-2\psi \p_x \psi \bar{w} $.
  Recalling that $\|\partial_x \psi \|_{L^2} \les K$, the contribution of the first term is easily estimated as
  \EQQS{
    &\sum_{\substack{1\le N_2,N_4\ll N_1,\\N_1^{-1}\le M\le N_1^{2/3}}}
     \|\p_x^{-1}\dP_M(P_{N_2}\p_x( \psi^2 \bar{w})P_{N_4}\bar{\psi})\|_{L_x^\I}\\
    &\les \sum_{\substack{0<N_2\les M,1\le N_4\ll N_1,\\N_1^{-1}\le M\le N_1^{2/3}}}
      M^{-1/2} N_2 \|\dP_{N_2}( \psi^2 \bar{w})\|_{L_x^2}
      \|P_{N_4}\psi\|_{L_x^\I} \\
    &\quad+\sum_{\substack{M\ll N_2,(1\vee N_2) \sim N_4\ll N_1,\\N_1^{-1}\le M\le N_1^{2/3}}}
       N_2 \|\dP_{N_2}( \psi^2 \bar{w})\|_{L_x^2}
      \|P_{N_4}\psi\|_{L_x^2} \\
    &\les K^2 \sum_{\substack{N_1^{-1}\le M\le N_1^{2/3}}}
     \|w\|_{L^2_x}
     \Bigl(M^{-1/2} K
     + \log(N_1)\|\partial_x \psi\|_{L^2_x}\Bigr)
    \le  N_1^{1/2} K^3  \|w\|_{L^2_x} \, ,
   }
whereas the second term can be estimated as
  \EQQS{
    &\sum_{\substack{1\le N_2,N_4\ll N_1,\\N_1^{-1}\le M\le N_1^{2/3}}}
     \|\p_x^{-1}\dP_M(P_{N_2}( \psi \p_x \psi\bar{w})P_{N_4}\bar{\psi})\|_{L_x^\I}\\
    & \les \sum_{\substack{1\le N_2,N_4\ll N_1,\\N_1^{-1}\le M\le N_1^{2/3}}}
     \|\dP_M(P_{N_2}( \psi \p_x \psi\bar{w})P_{N_4}\bar{\psi})\|_{L_x^1}\\
       &\les \sum_{\substack{1\le N_2,N_4\ll N_1,\\N_1^{-1}\le M\le N_1^{2/3}}} \|\psi\|_{L_x^\I}\|\p_x \psi\|_{L^2_x} \|w\|_{L^2_x} \|\psi\|_{L_x^\I}
    \les  N_1^{0+} K^3  \|w\|_{L^2_x} \, ,
   }
    This completes the estimates for $\E_N^3(u_1,u_2)$.
  Putting all the estimates we obtained above, Lemma \ref{extensionlem}, and \eqref{estdiffXregular}, we finally arrive at
  \EQQS{
    E^\theta(u_1(t),u_2(t),N_0)
    \les E^\theta(u_{0,1},u_{0,2},N_0)
    + CT^{1/4}(\|w\|_{L_T^\I H_x^\theta}^2+\|w\|_{L_T^\I \ov{H}^{\theta,\delta}}^2).
  }
  By taking supremum over $[0,T]$, we get \eqref{eq_dif_nonreg}, which concludes the proof.
\end{proof}

\section{Proof of the Main Theorems}\label{sec_proof}

In this section, we discuss the proofs of Theorems \ref{thm1}, \ref{thm2} and \ref{thm3}.

\subsection{Proof of Theorems \ref{thm1} and \ref{thm2}}

The proofs of Theorems \ref{thm1} and \ref{thm2} follow standard arguments (see, for instance, Section 6 of \cite{MT21}).
We omit the proof of Theorem \ref{thm1}, but we provide that of Theorem \ref{thm2} for the reader's convenience.

\begin{proof}[Proof of Theorem \ref{thm2}]
  Let $s>3/4$.
  We fix $\e,\de>0$ such that $s>3/4+2\de+\e$.
  We assume that $\psi$ satisfies \eqref{hyp_psi2} with this choice of $s$ and $\e$.
  We begin by showing the uniqueness result.

  \noindent
  \underline{Step 1. Unconditional Uniqueness.}
  Let $u_0 \in H^s(\R)$ and let $u_1,u_2\in L_T^\I H^s$ be two solutions to \eqref{eq2} emanating from the same initial data $u_0 $ for some $0<T<1$.
  According to Lemma \ref{lem47}, $ u_1$ and $ u_2 $ belong to $ Z^s_T $ and since $ u(0)-v(0) =0 $.
  We infer from Lemma \ref{lem_weight} that $w=u_1-u_2 $ belongs to $ L^\infty_T \overline{H}^{s,\delta}(\R) $.
  Then, by choosing $1/4<\theta<s/2-1/8 $, $N_0\gg 1 $ large enough and $T>0$ sufficiently small, Lemma \ref{LemZ22}, Propositions \ref{prop_coer} and \ref{prop_dif2} then ensure that
  \EQS{\label{71}
    \|u_1-u_2\|_{L_T^\I H_x^\theta}\lesssim
    \|u_1-u_2\|_{L_T^\I \overline{H}^{\theta,\de}}\lesssim \|u_1(0)-u_2(0)\|_{{\overline H}^{\theta,\de}}=0,
  }
  which implies that $u_1(t)=u_2(t)$ on $t\in[0,T]$.
  For the whole time interval, we repeat this argument a finite number of times.
  This is possible due to the coincidence of the initial data.

  \noindent
  \underline{Step 2. Local existence.}
    Let $u_0\in H^s(\R)$.
    Define $\{u_{0,n}\}_{n\ge 1}\subset H^{\I}(\R)$ by $u_{0,n}=P_{\le 2^n}u_0$ so that $u_{0,n}\to u_0$ in $H^s(\R)$ as $n\to\I$.
    Notice that $P_{\le 1}(u_{0,n}-u_{0,m})=0$ for any $n,m\in\N$ and that $\|u_{0,n}\|_{H_x^{s'}}\le \|u_{0}\|_{H_x^{s'}}$ for $0\le s'\le s$.
    Let $u_n\in C([0,T_{1,n}];H^{\I}(\R))$ be a solution to \eqref{eq2} with the initial data $u_{0,n}$ for $n\in \N_{\ge 1}$.
    Moreover, we can assume that $\liminf_{t\to T_{1,n}}\|u_n(t)\|_{H_x^2}=\I$ or $T_{1,n}=\I$ holds for each $n\in \N_{\ge 1}$.
    Such solutions can be obtained via the classical energy method at least in $H^2(\R)$, or by Theorem \ref{thm1}.
    We see from Proposition \ref{prop_apri} with $\om_N\equiv 1$ that there exists an increasing continuous function $G:\R_+\to\R_+$ such that
    \EQS{\label{eq_aps2}
      \|u_n\|_{L_T^\I H_x^s}^2
      \le \|u_{0,n}\|_{H_x^s}^2
       + T^{1/4}G(\|u_n\|_{L_T^\I H_x^{s_0}}+K_{s})
         \|u_n\|_{L_T^\I H_x^s}^2
       + T\|\Psi\|_{L_T^\I H_x^{s+\e}}
    }
    for any $0\le T\le T_{1,n}$ and $s\ge s_0>3/4$,
    where $K_s$ is defined in \eqref{eq77}.
    Remark that we do not need $\|\p_x\psi\|_{L_{t}^\I L_x^1}$ in order to deduce \eqref{eq_aps2}, whereas it is required to obtain \eqref{71}.
    By setting
    \EQQS{
      T_{2,n}:=\min \big\{\big(4 G(2\|u_0\|_{H_x^{s_0}}+K_{s_0})\big)^{-4},\|u_0\|_{H_x^{s_0}}(2\|\Psi\|_{L_t^\I H_x^{s_0+\e}}+1)^{-1},T_{1,n}\big\},
    }
    we see from \eqref{eq_aps2} with $s=s_0$ and the continuity argument that
    \EQS{\label{eq_aps}
      \|u_n\|_{L_T^\I H_x^{s_0}}^2\le 2 \|u_{0,n}\|_{H_x^{s_0}}^2
      \le 2\|u_0\|_{H_x^{s_0}}^2.
    }
    for $0\le T\le T_{2,n}$ and $n\in \N_{\ge 1}$.
    We put again
    \EQQS{
      T_3(s,n)
        &:=\inf\{T\in [0,T_{1,n}]:
         \|u_n\|_{L_T^\I H_x^s}^2> 3\|u_0\|_{H_x^s}^2\},\\
      \ti{T}_{3}(s)
        &:=\min \big\{\big(4G(2\|u_0\|_{H_x^{s_0}}+K_s)\big)^{-4},\|u_0\|_{H_x^{s_0}}(2\|\Psi\|_{L_t^\I H_x^{s+\e}}+1)^{-1},T_3(s,n)\big\}.
    }
    Notice that $T_3(s,n)>0$ and $\|u_n\|_{L_T^\I H_x^s}^2\le 3\|u_0\|_{H_x^s}^2$ for $0\le T\le T_3(s,n)$.
    On the other hand, by \eqref{eq_aps2}, \eqref{eq_aps}, and the definition of $\ti{T}_{3}(s)$, we obtain
    \EQS{\label{eq_aps3}
      \|u_n\|_{L_T^\I H_x^s}^2
      \le 2 \|u_{0,n}\|_{H_x^s}^2
      \le 2\|u_0\|_{H_x^s}^2
      <3\|u_0\|_{L_T^\I H_x^s}^2.
    }
    for $0\le T\le \ti{T}_{3}(s)$ and $n\in \N_{\ge 1}$, which implies that $\ti{T}_{3}(s)<T_3(s,n)$.
    In particular, $\ti{T}_{3}(s)$ is independent of $n$.
    For simplicity, we write $T=\ti{T}_{3}(s)$.
    Thanks to this estimate, there exists $C>0$ such that for any $n,m\in\N_{\ge 1}$,
    \EQQS{
      \ti{K}_s=K_s+\|u_n\|_{L_T^\I H_x^s}+\|u_m\|_{L_T^\I H_x^s} \le CK_s+C\|u_0\|_{H_x^s}=:K_s'.
    }
    Note that we can control $N_0\gg 1$, which is used in Proposition \ref{prop_coer}, by this $K_s'$.
    By a similar argument to \eqref{71}, we obtain
    \EQS{\label{71.1}
      \|u_n-u_m\|_{L_{T'}^\I H_x^\te}
      \les \|u_{0,n}-u_{0,m}\|_{H_x^{\te}}
    }
    for $T'\sim G(K_s')^{-4}$.
    By \eqref{71.1}, interpolation, and the triangle inequality with \eqref{eq_aps3}, we also obtain for any $\theta<s'<s$,
    \EQQS{
      \|u_n-u_m\|_{L_{T'}^\I H_x^{s'}}
      &\le \|u_n-u_m\|_{L_{T'}^\I H_x^\theta}^{1-\eta}
      \|u_n-u_m\|_{L_{T'}^\I H_x^s}^{\eta}\\
      &\le G(K_{s}')\|u_n-u_m\|_{L_{T'}^\I H_x^\theta}^{\eta}
      \to0
    }
    as $n,m\to \I$, where $\eta=(s'-\theta)/(s-\theta)\in(0,1)$.
    This implies that $\{u_n\}_{n\ge 1}$ is a Cauchy sequence in $C([0,T'];H^{s-}(\R))$, and there exists $u\in C([0,T'];H^{s-}(\R))$ such that $u_n \to u$ in $C([0,T'];H^{s-}(\R))$.
    Moreover, \eqref{eq_aps3} (with a standard argument of functional analysis) ensures that $u\in L^\I([0,T'];H^s(\R))$ and $u\in C_w([0,T'];H^s(\R))$.
    It is easy to see that this $u$ solves \eqref{eq2} on $[0,T']$ with the initial data $u(0,x)=u_0(x)$.
    Moreover, by Step 1, this is the unique solution.

  \noindent
  \underline{Step 3. Strong continuity.}
  We first show the strong continuity at $t=0$.
  We see from the proof of \eqref{eq_aps2} that
  \EQQS{
    \limsup_{t\to 0}\|u(t)\|_{H_x^s}
    \le \|u_0\|_{H_x^s}+\lim_{t\to 0}t^{1/4}C(K_{s}')\|u_n\|_{L_T^\I H_x^s}^2
    \le \|u_0\|_{H_x^s}.
  }
  This, together with the weak continuity in $H^s(\R)$, shows that $\lim_{t\to 0}\|u(t)\|_{H_x^s}$ exists and $\lim_{t\to 0}\|u(t)\|_{H_x^s}=\|u_0\|_{H_x^s}$.
  Therefore, we have $u(t)\to u_0$ in $H^s(\R)$ as $t\to0$.
  Other cases $t=t'\in(0,T']$ follow from the same argument with \eqref{eq_aps2} starting at $t=t'$, where $T'$ is obtained by Step 2.
  Iterating this process employed above a finite number of times, we obtain $u_n\to u$ in $C([0,T'];H^{s-}(\R))$ and $u\in C([0,T'];H^s(\R))$.

  \noindent
  \underline{Step 4. Continuous dependence.}
  We denote the flow map defined by the above argument by $S$, i.e., $S:H^s(\R)\ni u_0\mapsto u\in C([0,T];H^s(\R))$.
  Let $u_0\in H^s(\R)$ and $\{u_{0,n}\}_{n\in \N_{\ge 1}}\subset H^s(\R)$.
  Assume that $u_{0,n}\to u_0$ in $H^s(\R)$ as $n\to \I$ and that $u,u_n\in C([0,T];H^s(\R))$ are the solutions to \eqref{eq2} with initial data $u_0,u_{0,n}$ for $n\in \N_{\ge1}$, respectively.
  Here, $T=T'$ is obtained by Step 2.
  Without loss of generality, we may assume that $\sup_{n\ge 1}\|u_{0,n}\|_{H^s}\le 2\|u_0\|_{H^s}$.
  Now, we show that $u_n\to u$ in $C([0,T];H^s(\R))$ as $n\to\I$.
  For that purpose, we first construct the frequency envelope.
  We define a sequence of positive numbers $\{\om_N\}_{N\in 2^{\N}}$ such that
  $\om_N\le \om_{2N}\le 2^{\e/2} \om_N$, $\om_N\to\I$ as $N\to\I$, and
  \EQQS{
    \|u_0\|_{H_\om^s}+\sup_{n\ge 1}\|u_{0,n}\|_{H_\om^s}<K''
  }
  for some $K''<\I$.
  For details, see Lemma 4.1 in \cite{KT1} (and also Lemma 4.6 in \cite{MT21}).
  By the same argument as \eqref{eq_aps3}, we have
  \EQS{\label{eq_aps4}
    \|u\|_{L_T^\I H_\om^s}^2
    \le 2\|u_0\|_{H_\om^s}^2,\quad
    \sup_{n\ge 1}\|u_n\|_{L_T^\I H_\om^s}^2
    \le 2\sup_{n\ge 1}\|u_{0,n}\|_{H_\om^s}^2
  }
  We notice that
  \EQQS{
    u-u_n
    &=S(u_0)-S(u_{0,n})\\
    &=[S(u_0)-P_{\le N}S(u_0)]
     +[P_{\le N}S(u_0)-P_{\le N}S(P_{\le N}u_0)]\\
    &\quad+[P_{\le N}S(P_{\le N}u_0)-P_{\le N}S(P_{\le N}u_{0,n})]
     +[P_{\le N}S(P_{\le N}u_{0,n})-P_{\le N}S(u_{0,n})]\\
    &\quad +[P_{\le N}S(u_{0,n})-S(u_{0,n})]
    =:A_1+\cdots+A_5.
  }
  Let $\ga>0$.
  Now we show that $\|u-u_n\|_{L_T^\I H_x^s}<\ga$ by choosing $N$ and $n$ appropriately.
  By choosing $N_1$ sufficiently large, we see from \eqref{eq_aps4} that
  \EQQS{
    &\|A_1\|_{L_T^\I H_x^s}
    =\|P_{> N_1}u\|_{L_T^\I H_x^s}
    \le \om_{N_1}^{-1}\|u\|_{L_T^\I H_\om^s}
    \le \om_{N_1}^{-1}C(K'')<\frac{\ga}{5},\\
    &\sup_{n\ge 1}\|A_5\|_{L_T^\I H_x^s}
    \le \om_{N_1}^{-1}C(K'')<\frac{\ga}{5},
  }
  On the other hand, by \eqref{71.1}, we choose $N_2$ sufficently large so that
  \EQQS{
    \|A_2\|_{L_T^\I H_x^s}
    &\le CN_2^{s-\theta}\|S(u_0)-S(P_{\le N_2}u_0)\|_{L_T^\I H_x^\theta}\\
    &\le CN_2^{s-\theta} \|P_{>N_2}u_0\|_{H_x^\theta}
    \le \om_{N_2}^{-1}C(K'')<\frac{\ga}{5},\\
    \sup_{n\ge 1}\|A_4\|_{L_T^\I H_x^s}
    &\le C(K_{s_0,s_0,\om}) N_2^{s-\theta}
     \sup_{n\ge 1} \|P_{>N_2}u_{0,n}\|_{H_x^\theta}
    \le \om_{N_2}^{-1}C(K'')<\frac{\ga}{5}.
  }
  Note that $P_{\le 1}(u_0-P_{\le N_2}u_0)=0$ if $N_2> 2$, which allows us to use \eqref{71.1}.
  Set $N:=\max\{N_1,N_2\}$.
  Let us now tackle the estimate on $A_3$.
  Note that here we cannot use \eqref{71.1} since $u_{0,n}-u_0$ is not guaranteed to belong to $\overline{H}^{\theta,\de}(\R) $.
  Instead, we make use of \eqref{eq_dif2} since both $ S(P_{\le N} u_0) $ and $ S(P_{\le N} u_{0,n})$ belong to $ H^1(\R) $.
  This ensures that for sufficiently small $ T'=T'(N)>0$,
  \EQQS{
    \|A_3\|_{L^\infty_{T'} H_x^s}
    \le N^s \|A_3\|_{L^\infty_{T'} L^2_x}
    \le C(N) \|u_0-u_{0,n} \|_{L^2_x}
  }
  for some $C(N)>0$.
  By repeating this argument finitely many times (depending on $N$), we obtain
  \EQQS{
    \|A_3\|_{L^\infty_T H_x^s}
    \le C'(N) \|u_0-u_{0,n} \|_{L^2_x}.
  }
  Then, by choosing $n\ge 1$ sufficiently large, we obtain
  \EQQS{
    \|A_3\|_{L_T^\I H_x^s}<\frac{\ga}{5},
  }
  which implies that $\|u-u_n\|_{L_T^\I H_x^s}<\ga$.
  This completes the proof of the continuity of the flow map.
\end{proof}

\subsection{Proof of Theorem \ref{thm3}}

In this subsection, we give a proof of Theorem \ref{thm3}.
For that purpose, we first state a useful lemma, which was essentially proved in \cite{G05}.

\begin{lem}\label{lem_Zhidkov1}
  Let $s\in \R$ and $f\in Y^s(\R)$.
  Then, there exist $\psi\in C_b^\I(\R)$ and $\vp\in H^s(\R)$ such that $\psi'\in H^\I(\R)$ and $f(x)=\psi(x)+\vp(x)$ for $x\in \R$ a.e.
  Moreover, the maps $f\mapsto \psi$ and $f\mapsto \vp$ can be defined as bounded linear operators in the following sense: the map $f\mapsto \psi$ is continuous from $Y^s(\R)$ to $Y^{s'}(\R)$ for any $s'\in\R$, and the map $f\mapsto \vp$ is continuous from $Y^s(\R)$ to $H^s(\R)$.
\end{lem}

\begin{proof}[Proof of Theorem \ref{thm3}]

  First, we notice that $H^s(\R)\hookrightarrow Y^s(\R)$ since $s>3/4$.

  \noindent
  \underline{Step 1. Local existence.}
  Let $v_0\in Y^s(\R)$.
  We see from Lemma \ref{lem_Zhidkov1} that there exist $\psi\in C_b^\I(\R)$ and $u_0\in H^s(\R)$ such that $v_0=\psi+u_0$ and $\psi'\in H^\I(\R)$.
  Obviously, this $\psi$ is independent of $t$.
  By Lemma \ref{lem_Zhidkov1}, $\psi$ satisfies \eqref{hyp_psi}.
  Moreover, we have
  \EQQS{
    \|J_x^{s+3}\psi\|_{L^{\I}}
    +\|\Psi\|_{H^{s+1}}
    \le C(1+\|u_0\|_{Y^s})^3,\quad
    \|\vp\|_{H^s}
    \le C\|u_0\|_{Y^s}.
  }
  Then, by Theorem \ref{thm1}, there exist $T=T(\|v_0\|_{Y^s})>0$ and $u\in C([0,T];H^s(\R))$ such that $u$ is the solution to \eqref{eq2} with the initial data $u(0)=u_0\in H^s(\R)$.
  By setting $v(t):=u(t)+\psi$, $v$ is  solution to \eqref{eq1} emanating from $v(0)=v_0$ and satisfies $v\in C([0,T];Y^s(\R))$.
  We also have $v(t)-v_0=u(t)-u_0\in C([0,T];H^s(\R))$.

  \noindent
  \underline{Step 2. Uniqueness.}
  Let $v_1,v_2\in C([0,\de];Y^s(\R))$ be two solutions to \eqref{eq1} emanating from  the same initial data $v_0\in Y^s(\R)$.
  Assume that $v_1-v_0$ and $v_2-v_0\in C([0,\de];H^s(\R))$.
  By the decomposition due to Lemma \ref{lem_Zhidkov1}, there exists $ \psi\in  H^\I(\R)$ such $v_j(t)-\psi $ belong to $C([0,\de];H^s(\R))$  and satisfy \eqref{eq2} for $j=1,2$.
  The unconditional uniqueness of Theorem \ref{thm1} then ensures that  $v_1(t)-\psi=v_2(t)-\psi$, i.e., $v_1(t)=v_2(t)$ on $[0,\de]$.

  \noindent
  \underline{Step 3. Continuous dependence.}
  Let $\{v_{0,n}\}_{n\in \N}\subset Y^s(\R)$ and $v_0\in Y^s(\R)$.
  Assume that $v_{0,n}\to v_0$ in $Y^s(\R)$.
  Then by Lemma \ref{lem_Zhidkov1} there exists $ \{u_{0,n}\}_{n\in \N} \subset H^s(\R) $, $u_0\in H^s(\R)$, $\{\psi_n\}_{n\in \N} \subset C_b^\I(\R) $ and $ \psi\in C_b^\I(\R) $ such that
  $ v_{0,n}=u_{0,n}+\psi_n $, $ v_0=u_0+\psi$, $u_{0,n}-u_0\to 0 $ in $ H^s(\R) $ and $ \psi_n-\psi\to 0$ in $ Y^s(\R) $ as $n\to +\infty$.
  Moreover, we have $\psi_n',\psi'\in H^\I(\R)$.
  By the above argument, we can construct solutions $v_n=u_n+\psi_n $ and $ v=u+\psi $ to \eqref{eq1} with $v_n(0)=v_{0,n}$ and $v(0)=v_0$, respectively, where  $ u_n, u\in C([0,T]; H^s(\R)) $ are the solutions to \eqref{eq2} associated with respectively $(u_{0,n},\psi_n)$ and $ (u_0,\psi) $.
  From \eqref{eq_dif1Z1} in Proposition \ref{prop_dif1Z1}, we infer that  for $ s'=(s-1)\wedge 0 $ it holds
  \EQQS{
    \|u_n-u\|_{L^\infty_T H_x^{s'}} \to 0 \quad \text{as} \quad n\to +\infty.
  }
  Therefore, proceeding exactly as in the proof of Theorem \ref{thm2} with the help of the frequency envelopes, we get that
  \EQQS{
    \|u_n-u\|_{L^\infty_T H_x^{s}} \to 0 \quad \text{as} \quad n\to +\infty,
  }
 which leads to
  \EQQS{
    \|v_n-v\|_{L_T^\I Y_x^{s}}
    &\le \|u_n-u\|_{L_T^\I H_x^{s}}+\|\psi_n-\psi_0\|_{Y^{s}}
    \to 0\quad \text{as} \quad n\to +\infty.
  }
  This completes the proof.
\end{proof}

\section*{Acknowlegdements}

The second author was supported by the EPSRC New Investigator Award (grant no. EP/V003178/1), JSPS KAKENHI Grant Number 23K19019, JP20J12750, JSPS Overseas Challenge Program for Young Researchers and Iizuka Takeshi Scholarship Foundation and a grant from the Harris Science Research Institute of Doshisha University. A large part of this work was conducted during a visit of the second author at Institut Denis Poisson (IDP) of Universit\'e de Tours in France. The second author is deeply grateful to IDP for its kind hospitality.

\end{document}